\pdfoutput=1
\documentclass{amsart}

\usepackage[utf8]{inputenc}
\usepackage[english]{babel}

\usepackage{braket}
\usepackage{amsmath}
\usepackage{amstext}
\usepackage{amssymb}
\usepackage{amsthm}
\usepackage{mathtools}
\usepackage{stmaryrd}
\usepackage{mathrsfs}
\usepackage{extarrows}
\usepackage{afterpage}

\usepackage{microtype}

\usepackage{graphicx}
\usepackage{amsfonts}

\usepackage{units}
\usepackage{enumerate}
\usepackage{url}
\usepackage{multirow}

\usepackage{tikz}
\usetikzlibrary{matrix,arrows,calc,fit}

\usepackage{hyperref}

\usepackage[margin=1cm]{caption}

\allowdisplaybreaks[4]

\theoremstyle{definition}
\newtheorem{theorem}{Theorem}[section]
\newtheorem{definition}[theorem]{Definition}
\newtheorem{remark}[theorem]{Remark}
\newtheorem{example}[theorem]{Example}
\newtheorem{lemma}[theorem]{Lemma}
\newtheorem{proposition}[theorem]{Proposition}
\newtheorem{corollary}[theorem]{Corollary}

\newenvironment{provedtheorem}
  {\pushQED{\qed}\begin{theorem}}
  {\popQED\end{theorem}}
\newenvironment{provedcorollary}
  {\pushQED{\qed}\begin{corollary}}
  {\popQED\end{corollary}}

\newtheorem{innercustomthm}{Theorem}
\newenvironment{customthm}[1]
  {\renewcommand\theinnercustomthm{#1}\innercustomthm}
  {\endinnercustomthm}

\newcommand{\N}{\mathbb{N}}
\newcommand{\Z}{\mathbb{Z}}
\newcommand{\R}{\mathbb{R}}
\newcommand{\C}{\mathbb{C}}
\newcommand{\Q}{\mathbb{Q}}
\newcommand{\K}{\mathbb{K}}

\newcommand{\M}{\mathcal{M}}
\renewcommand{\S}{\mathcal{S}}
\newcommand{\I}{\mathcal{I}}
\newcommand{\F}{\mathcal{F}}
\newcommand{\A}{\mathcal{A}}

\newcommand{\W}{{\bf W}}
\newcommand{\GW}{{\bf G_W}}

\newcommand{\Y}{{\bf Y}}
\newcommand{\YW}{{\bf Y_W}}
\newcommand{\X}{\mathbf X}
\DeclareMathOperator{\Br}{Br}
\DeclareMathOperator{\Ind}{Ind}

\newcommand{\de}{\partial}

\DeclareMathOperator{\id}{id}
\DeclareMathOperator{\rk}{rk}
\DeclareMathOperator{\interior}{int}
\DeclareMathOperator{\pt}{pt}

\newcommand{\Cr}{\text{Cr}}

\DeclareMathOperator{\ch}{char}

\renewcommand{\Im}{\text{im}\,}
\newcommand{\coker}{\text{coker}\,}

\newcommand{\w}{\mathit w}
\newcommand{\ph}{\varphi}
\newcommand{\Sf}{\mathcal S_{\varphi}}

\newcommand{\XW}{\mathbf X_{\mathbf W}}

\newcommand{\<}{\langle}
\renewcommand{\>}{\rangle}

\newcommand{\xor}{\veebar}

\renewcommand{\pmod}[2][]{\;#1(\text{mod}\; #2#1)}

\title{Weighted sheaves and homology of Artin groups}

\author{Giovanni Paolini}
\address{Giovanni Paolini}
\address{\textup{Scuola Normale Superiore, Piazza dei Cavalieri 7, 56126 Pisa (Italy)}}
\email{giovanni.paolini@sns.it}

\author{Mario Salvetti}
\address{Mario Salvetti}
\address{\textup{Department of Mathematics, University of Pisa, Largo Pontecorvo 5, 56127 Pisa (Italy)}}
\email{salvetti@dm.unipi.it}

\thanks{The final publication in \emph{Algebraic \& Geometric Topology} is available at \url{http://doi.org/10.2140/agt.2018.18.3943}.}

\begin{document}

\begin{abstract}
  In this paper we expand the theory of weighted sheaves over posets, and use it to study the local homology of Artin groups.
  First, we use such theory to relate the homology of classical braid groups with the homology of certain independence complexes of graphs.
  Then, in the context of discrete Morse theory on weighted sheaves, we introduce a particular class of acyclic matchings. Explicit formulas for the homology of the corresponding Morse complexes are given, in terms of the ranks of the associated incidence matrices.
  We use such method to perform explicit computations for the new affine case $\smash{\tilde C_n}$, as well as for the cases $A_n$, $B_n$ and $\smash{\tilde{A}_n}$ (which were already done before by different methods).
\end{abstract} 

\maketitle

\tableofcontents

\section{Introduction}

The topological theory of Artin groups started in the 60's with classical braid groups \cite{fox1962braid, arnold1969cohomology, arnold1970some}, and was broadened to general Artin groups in the 70's \cite{fuchs1970cohomologies, brieskorn1972artin, deligne1972immeubles, brieskorn1973groupes, cohen1976homology, weinstein1978cohomologies, goryunov1978cohomologies}.
It has received much attention in connection with problems in singularity theory, homotopy theory, hyperplane arrangements, and combinatorics (see \cite{paris2012k} for some references), up to recent applications of some special cases to conjectures about hyperbolic groups \cite{agol2013virtual}.

Thanks to a good topological model for the $K(\pi,1)$ space of Artin groups \cite{salvetti1987topology, salvetti1994homotopy}, a combinatorial free resolution for these groups was found \cite{de1996cohomology}, and several (co)homology computations were carried out \cite{de1999arithmetic, de2001arithmetic, callegaro2004integral, callegaro2005cohomology, callegaro2006homology, callegaro2008cohomology, callegaro2008cohomology2, callegaro2010k} (equivalent resolutions also appeared in \cite{squier1994homological} and \cite{ozornova2017discrete, paolini2017classifying}).
The computational methods used in these papers are essentially based on filtering the algebraic complex and using the associated spectral sequence. 

In \cite{salvetti2013combinatorial} a more combinatorial method of calculation was introduced, based on the application of discrete Morse theory to a particular class of sheaves over posets (called \emph{weighted sheaves}).
It is interesting to notice that similar ideas were considered later in \cite{curry2016discrete}, even if there the authors were mainly interested in computational aspects.

In this paper we expand the theory of weighted sheaves. 
First, we observe  a surprising relationship between some twisted (co)homology of the classical braid groups and the (co)homology of certain \emph{independence complexes}, which in their simplest form have already been studied in a combinatorial context 
(see for example \cite{babson2007proof, kozlov2007combinatorial, 
engstrom2009complexes}).
This relationship is obtained in an indirect way, after localization.
The exact formula we are going to prove is the following (it was announced in \cite{salvetti2015some} without proof).

\begin{customthm}{\ref{maintheorem}}
  \[ H_*(\Br_{n+1};R)_{\varphi_d} \cong \tilde H_{*-d+1}\left(\Ind_{d-2}(A_{n-d});\frac{R}{(\varphi_d)}\right). \]
\end{customthm}
\noindent Here on the left we have the homology with local coefficients of the classical braid group and on the right we have the homology of some independence complex with trivial coefficients (we indicate by $\varphi_d$ the $d$-th cyclotomic polynomial).

As said before, our method is based on the notions of weighted sheaves over posets and of \emph{weighted matchings}, which are interesting objects by themselves.
In our situation, the $d$-weight $v_d(\sigma)$ of a simplex $\sigma$ is the maximal power of the $d$-th cyclotomic polynomial which divides the Poincar\'e polynomial of the (finite) parabolic subgroup generated by the vertices of the simplex.
We allow matchings between simplices of the same weight.
Applying discrete Morse theory, one obtains a Morse complex which still computes the local homology.

Then, by applying these methods to other Coxeter graphs (in particular to finite and affine type cases), we come up with the notion of \emph{precise matching}.
In all the cases we consider, it is possible to produce acyclic weighted matchings such that the Morse complex turns out to have a very nice property: the boundary of a critical $p$-simplex $\sigma$ has a non-zero coefficient along a critical $(p-1)$-simplex $\tau$ if and only if $v_d(\sigma)=v_d(\tau)+1.$ 
Then the homology of the Morse complex reduces to the computation of the ranks of the incidence matrices among critical simplices (Theorems \ref{thm:homology-phi-component} and \ref{thm:homology-L}).
We show that the existence of precise matchings can also be interpreted in terms of second-page collapsing of a spectral sequence which is naturally associated to the weighted sheaf (Proposition \ref{prop:precise-matching-spectral-sequence}).

As we discuss at the beginning of Section \ref{sec:precise-matchings-artin-groups}, the existence of precise matchings for a certain Artin group has strong implications for its homology.
In particular we want to highlight the following non-trivial result.

\begin{customthm}{\ref{cor:squarefree-torsion}}
  Let $\GW$ be an Artin group (corresponding to some finitely generated Coxeter group $\W$) that admits a $\varphi$-precise matching for all cyclotomic polynomials $\varphi=\varphi_d$ (with $d\geq 2$). Then the homology $H_*(\XW;R)$ does not have $\varphi^k$-torsion for $k\geq 2$.
\end{customthm}

\noindent Here $\XW$ (introduced in Section \ref{sec:weighted-sheaves-artin-groups}) is a finite CW-complex which is a deformation retract of the \emph{orbit space} associated to $\W$, and $\pi_1(\XW)=\GW$.
When $\W$ is finite and in a few other cases $\XW$ it is known to be a $K(\GW,1)$, whereas this property is only conjectured in general \cite{deligne1972immeubles, brieskorn1973groupes, okonek1979dask, van1983homotopy, hendriks1985hyperplane, salvetti1987topology, salvetti1994homotopy, charney1995k, callegaro2010k, paris2012k}.

Theorem \ref{cor:squarefree-torsion} can be seen as a result about the homology of the infinite cyclic covering of $\XW$ and its monodromy.
It is known that the conclusion of this theorem holds for all Artin groups of finite type, due to geometric reasons (see Remark \ref{rmk:squarefree-torsion-finite-type}).
Our construction provides a new combinatorial condition that applies to a wider class of Artin groups (including, among others, the affine groups of type $\smash{\tilde A_n}$ and $\smash{\tilde C_n}$).

The methods we develop allow explicit and complete computations, in general in a more direct way compared to other known methods. We use them to compute the twisted homology in the case of affine Artin groups of type $\smash{\tilde{C}_n}$, obtaining the following result.

\begin{customthm}{\ref{thm:homology-tCn}}
  Let $\GW$ be an Artin group of type $\tilde C_n$.
  Then the $\varphi_d$-primary component of $H_*(\GW;R)$ is trivial for $d$ odd, and for $d$ even is as follows:
  \[ H_m(\GW;R)_{\varphi_d} \cong
    \begin{cases}
      \left( R/(\varphi_d) \right)^{\oplus m+k-n+1} & \text{if $n-k \leq m \leq n-1$}, \\
      0 & \text{otherwise},
    \end{cases}
  \]
  where $n=k\frac d2 + r$.
\end{customthm}

We also carry out complete computations in the cases $A_n$, $B_n$ and $\smash{\tilde A_n}$: these cases were already done before by different methods \cite{frenkel1988cohomology, de1999arithmetic, de2001arithmetic, callegaro2008cohomology}, but we need them to perform computations of more complicated cases.

In a recent preprint \cite{paolini2017local}, the first author succeeded in constructing precise matchings for all remaining Artin groups of finite and affine type.
Therefore precise matchings seem to be a suitable and effective tool for studying the homology of Artin groups.

The paper is structured as follows.
In Section \ref{sec:constructions} we recall the most important definitions and results of \cite{salvetti2013combinatorial}. In particular we introduce the general framework of weighted sheaves over posets, and then explain how it can be used to compute the homology of Artin groups.
In Section \ref{sec:braid-groups} we study the case of braid groups (i.e.\ Artin groups of type $A_n$), and relate their homology with that of certain independence complexes.
In Section \ref{sec:precise-matchings} we introduce precise matchings and derive a general formula for the homology of the Morse complex.
In Section \ref{sec:precise-matchings-artin-groups} we apply the theory developed in Section \ref{sec:precise-matchings} to the case of Artin groups. We construct precise matchings for Artin groups of type $A_n$, $B_n$, $\smash{\tilde A_n}$ and $\smash{\tilde C_n}$, and use such matchings to explicitly compute the homology.

\section{Preliminaries}
\label{sec:constructions}
In this section we are going to recall some definitions and constructions from \cite{salvetti2013combinatorial} (see also \cite{moroni2012some}).

\subsection{Weighted sheaves over posets}
Let $(P, \preceq)$ be a finite poset.

\begin{definition}
  Define a \emph{sheaf of rings} over $P$  (or a \emph{diagram of rings} over $P$) as a collection
  \[ \{A_x \mid x\in P\} \]
  of commutative rings together with a collection of ring homomorphisms
  \[ \{\rho_{x,y} \colon A_y \to A_x \mid x\preceq y\} \]
  satisfying
  \[ \rho_{x,x} = \id_{A_x}; \]
  \[ x\preceq y\preceq z\ \Rightarrow \ \rho_{x,z}=\rho_{x,y}\ \rho_{y,z}. \]
\end{definition}

In other words, a sheaf of rings over $P$ is a functor from $P^{\text{op}}$ to the category of commutative rings.

Fix a PID $R$; usually we take $R=\Q[q^{\pm 1}]$.
The divisibility relation
\[ p_1 \mid p_2\ \Leftrightarrow \ (p_1)\supseteq (p_2) \]
gives $R$ the structure of a small category, and to $R/R_*$ the structure of a poset (where $R_*$ is the group of units in $R$) with minimum element the class of the units and maximum element $0$.
Any functor $\w \colon (P,\preceq) \to (R,\mid)$, which maps every $x\in P$ to some $\w(x) \in R \setminus \{0 \}$,
defines a sheaf over $P$ by the collection
\[ \{ R/(\w(x)) \mid x\in P \} \]
and
\[ \{ i_{x,y}\colon R/(\w(y)) \to R/(\w(x)) \mid x \preceq y \} \]
where $i_{x,y}$ is induced by the identity of $R$.

\begin{definition}
  Given a poset $P$, a PID $R$ and a morphism $\w \colon (P,\preceq) \to (R,\mid)$ as above, we call the triple $(P,R,\w)$  a \emph{weighted sheaf} over $P$ and the coefficients $\w(x)$ the \emph{weights} of the sheaf.
\end{definition}

\begin{remark}
  Consider the \emph{poset topology} over $P$, where a basis for the open sets is given by the upper intervals $\mathcal B = \{ P_{> p},\ p\in P\}$. Let $(P,R,\w)$ be a weighted sheaf. Then one can see $\w$ as a functor from $(\mathcal B,\supseteq)$ to $(R,\mid)$ and thus a weighted sheaf defines a sheaf in the usual sense.
  \label{truesheaf}
\end{remark}

From now on our poset will be a simplicial complex $K$ defined over a finite set $S$, with the partial ordering
\[ \sigma\preceq \tau \, \Leftrightarrow \, \sigma \subseteq \tau. \]
We adopt the convention that $K$ contains the empty simplex $\varnothing$.
A weighted sheaf over $K$ is given by assigning to each simplex $\sigma\in K$ a weight $\w(\sigma)\in R$, with
\[ \sigma \preceq \tau \, \Rightarrow \,  \w(\sigma)\mid \w(\tau). \]
Let $C^0_*(K; R)$ be the $1$-shifted standard algebraic complex computing the simplicial homology of $K$, i.e.
\[ C^0_k(K; R) = \bigoplus_{\substack{
    \sigma \in K\\
    |\sigma|=k}} R\,e_{\sigma}^0 \]
where  $e_{\sigma}^0$ is a generator associated to a given orientation of $\sigma$ ($C^0_0(K; R)=R \,e_{\varnothing}^0$).
The boundary is given by
\[ \partial^0(e_{\sigma}^0) \, = \, \sum_{\mathclap{|\tau|=k-1}} \, [\sigma:\tau]\ e_{\tau}^0, \]
where $[\sigma :\tau ]$ is the incidence number (which is equal to $\pm 1$ if $\tau\prec \sigma$, and otherwise is equal to $0$).

\begin{definition}
  The \emph{weighted complex} associated to the weighted sheaf $(K,R,\w)$ is the algebraic complex $L_* = L_*(K)$ defined by
  \[ L_k = \bigoplus_{|\sigma|=k}
    \frac{R}{(w(\sigma))} \, \bar{e}_{\sigma}, \]
  with boundary $\partial\colon L_k \to L_{k-1}$ induced by $\partial^0$:
  \[ \partial(a_{\sigma}\bar{e}_{\sigma}) = \sum_{\tau\prec\sigma} \,
    [\sigma:\tau] \, i_{\tau,\sigma}(a_{\sigma}) \, \bar{e}_{\tau}. \]
  \label{def:weightedcomplex}
\end{definition}

There is a natural projection $\pi\colon C^0_*(K;R) \to L_*$ which maps a generator $e_\sigma$ in $C^0_*$ to the generator $\bar e_\sigma$ in $L_*$.

\begin{remark}  The diagram $\{ R / (w(\sigma)) \mid  \sigma\in K \}$ also defines a sheaf over the poset $K$ in the sense of Remark \ref{truesheaf}. The sheaf cohomology associated to the open covering given by the upper segments coincides with the homology of the weighted sheaf. 
\end{remark}

\subsection{Decomposition and filtration}

Let $\S=(K,R,\w)$ be a weighted sheaf.
For any irreducible $\ph\in R,$ we define the $\ph$-primary component $\Sf=(K,R,\w_{\ph})$ of the weighted sheaf $\S$ by setting
\[ w_{\ph}(\sigma) = \ph^{v_{\ph}(\sigma)}, \]
where
\[ v_{\ph}(\sigma) = \text{maximal $r$ such that $\ph^r$ divides $\w(\sigma)$}. \]
Since $R$ is a PID and $w(\sigma) \neq 0$, such a maximal value exists.
Notice that $\Sf$ is a weighted sheaf. 
The weighted complex $(L_*)_\varphi$ associated to $\Sf$ is called the $\ph$-primary component of the weighted complex $L_*$.
In degree $k$ one has:
\[ (L_k)_\varphi = \bigoplus_{|\sigma|=k} \,
  \frac{R}{(\ph^{v_{\ph}(\sigma)})} \, \bar{e}_{\sigma}. \]

The complex $(L_*)_\varphi$ has a natural increasing filtration by the following subcomplexes:
\[ F^s (L_*)_\varphi = \bigoplus_{v_{\ph}( \sigma)\leq s} \frac{R}{(\ph^{v_{\ph}( \sigma)})} \, \bar{e}_{\sigma}. \]
This filtration is associated to an increasing filtration of the simplicial complex $K$:
\[ K_{\ph,s} = \{ \sigma \in K \mid v_{\ph} (\sigma)\leq s \}. \]
Then $F^s(L_*)_\varphi$ is the weighted complex associated to the weighted sheaf
\[ (K_{\ph,s}, \,R, \,\w_{\ph}|_{K_{\ph,s}}). \]

\begin{theorem}[\cite{salvetti2013combinatorial}]
  Let $(K,R,\w)$ be a weighted sheaf, with associated weighted complex $L_*$. 
  For any irreducible $\ph\in R$, there exists a spectral sequence
  \[ E^r_{p,q} \ \Rightarrow \ H_*((L_*)_\varphi) \]
  that abuts to the homology of the $\ph$-primary component of the associated algebraic complex $L_*$.
  Moreover the $E^1$-term
  \begin{equation*}
    E^1_{p,q} = H_{p+q}(F^p/F^{p-1})\cong H_{p+q}(K_{\ph,p}, K_{\ph,p-1}; R/(\ph^p))
    \label{spectral}
  \end{equation*}
  is isomorphic to the relative homology with \emph{trivial coefficients} of the simplicial pair $(K_{\ph,p}, K_{\ph,p-1})$.
  \label{thm:spectral-sequence}
\end{theorem}

\subsection{Discrete Morse theory for weighted complexes}

Given $x$, $y$ in a poset $(P, \prec)$, the notation $x\lhd y$ means that $y$ \emph{covers} $x$, i.e.\ $x\prec y$ and there is no other $z\in P$ such that $x\prec z\prec y$.

Recall the basic facts of discrete Morse theory (see for example \cite{forman1998morse, kozlov2007combinatorial}).
A \emph{matching} in a poset $P$ is a set $\M\subseteq P\times P$ such that $(x,y)\in M \Rightarrow x\lhd y$ and each $x\in P$ belongs to at most one pair of $\M$.
An \emph{alternating path} is a sequence
\begin{equation*}\label{altpath}
y_0\rhd x_1\lhd y_1\rhd x_2\lhd y_2\rhd  \dots \rhd x_{m}\lhd y_m(\rhd x_{m+1})
\end{equation*}
such that each pair $x_i\lhd y_i$ belongs to $\M$ and no pair $x_i\lhd y_{i-1}$ belongs to $\M$. 
A \emph{cycle} is a closed alternating path with $y_0=y_m$.
An \emph{acyclic matching} over $P$ is a matching with no cycles. 
We are going to describe a variant of discrete Morse theory which is suitable for our situation.

\begin{definition}
  A \emph{weighted acyclic matching} on a weighted sheaf $(P,R,\w)$ over $P$ is an acyclic matching $\M$ on $P$ such that 
  \[ (x, y) \in \M\ \Rightarrow\ (\w(x)) = (\w(y)) \]
  (in \cite{salvetti2013combinatorial} we required $\w(x)=\w(y)$, but the above generalization does not change the results).
\end{definition}

The standard theory generalizes as follows.
Let $\S=(K,R,\w)$ be a weighted sheaf over a finite simplicial complex $K$, and let $\mathcal M$ an acyclic weighted matching.
A \emph{critical simplex} of $K$ is a simplex $\sigma$ which does not belong to any pair of $\mathcal M$.
The following definition is equivalent to \cite[Definition 3.3]{salvetti2013combinatorial}.

\begin{definition}
The \emph{Morse complex} of $\S$ with respect to $\M$ is defined as the torsion complex
\[ L^{\mathcal M}_* \,=
\bigoplus_{\sigma \text{ critical} } \frac{R}{(\w(\sigma))} \ \bar{e}_{\sigma} \]
with boundary
\[ \de^\M(\bar e_\sigma) = \, \sum_{\mathclap{\substack{
    \tau \text{ critical}\\
    |\tau| \,=\, |\sigma| - 1}}} \; [\sigma:\tau]^\M \, \bar e_\tau, \]
  where $[\sigma:\tau]^\M \in \Z$ is given by the sum over all alternating paths
  \[ \sigma \vartriangleright \tau_1 \vartriangleleft \sigma_1 \vartriangleright \tau_2 \vartriangleleft \sigma_2 \vartriangleright \dots \vartriangleright \tau_m \vartriangleleft \sigma_m \vartriangleright \tau \]
  from $\sigma$ to $\tau$ of the quantity
  \[ (-1)^m [\sigma:\tau_1][\sigma_1:\tau_1][\sigma_1:\tau_2][\sigma_2:\tau_2] \cdots [\sigma_m:\tau_m][\sigma_m:\tau]. \]
  The boundary map $\de^\M$ is extended by $R$-linearity, where some care should be taken since each $\bar e_\tau$ lives in a component with a possibly different torsion ($\de^\M$ is the same as the boundary map of \cite{salvetti2013combinatorial}).
  \label{def:morse-complex}
\end{definition}

\begin{theorem}[\cite{salvetti2013combinatorial}]
  Let $\mathcal S=(K,R,\w)$ be a weighted sheaf over a simplicial complex $K$. Let $\M$ be an acyclic matching for $\S$.
  Then there is an isomorphism 
  \[ H_*(L_*,\partial) \cong H_*(L^{\mathcal M}_*,\partial^{\mathcal M}) \]
  between the homology of the weighted complex $L_*$ associated to $\S$ and the homology of the Morse complex $L^{\M}_*$ of $\S$.
  \label{thm:morsecomplex}
\end{theorem}

\begin{remark}
If we forget about the weights, the matching $\M$ is in particular an acyclic matching for the simplicial complex $K$.
Therefore the algebraic complex $C^0_*(K;A)$, which computes the (shifted and reduced) simplicial homology of $K$ with coefficients in some ring $A$, admits a ``classical'' algebraic Morse complex
\[ C^0_*(K;A)^\M = \, \bigoplus_{\mathclap{\sigma \text{ critical}}} \, A \, \bar e_\sigma \]
with boundary map
\[ \delta^\M(\bar e_\sigma) = \, \sum_{\mathclap{\substack{
    \tau \text{ critical}\\
    |\tau| \,=\, |\sigma| - 1}}} \; [\sigma:\tau]^\M \, \bar e_\tau. \]
\end{remark}

\begin{remark}
  Set $C^0_* = C^0_*(K;R)$ and let $\pi\colon C^0_* \to L_*$ be the natural projection.
  The composition with the projection $L_* \to (L_*)_\varphi$, for some fixed irreducible element $\varphi\in R$, yields a natural projection $\pi_\varphi\colon C^0_* \to (L_*)_\varphi$.
  Similarly at the level of Morse complexes we have the projection $\bar \pi_\varphi \colon (C^0_*)^\M \to (L_*)_\varphi^\M$ which sends a generator $\bar e_\sigma$ in $(C^0_*)^\M$ to the generator $\bar e_\sigma$ in $\smash{(L_*)_\varphi^\M}$.
  Then, by construction, the maps induced in homology by the projections $\pi_\varphi$ and $\bar \pi_\varphi$ commute with the Morse isomorphisms of \cite{kozlov2007combinatorial} and \cite{salvetti2013combinatorial}:
  \begin{center}
  \begin{tikzpicture}
    \matrix (m) [matrix of math nodes,row sep=3em,column sep=3em,minimum width=1em]
    {  H_*(C^0_*) & H_*((L_*)_\varphi) \\
      H_*((C^0_*)^\M) & H_*((L_*)^\M_\varphi). \\};
    \path[-stealth]
      (m-1-1) edge [->] node [above] {\scriptsize $(\pi_\varphi)_*$} (m-1-2)
      (m-1-1) edge [->] node [left] {\scriptsize $\cong$} (m-2-1)
      (m-1-2) edge [->] node [left] {\scriptsize $\cong$} (m-2-2)
      (m-2-1) edge [->] node [above] {\scriptsize $(\bar\pi_\varphi)_*$} (m-2-2);
  \end{tikzpicture}
  \end{center}
  \label{rmk:projections}
\end{remark}

Let $\varphi\in R$ be an irreducible element, and let $\M$ be a weighted acyclic matching on $\Sf$.
Then the filtration on the weighted complex induces a filtration on the Morse complex:
\[ F^p (L_*)_\varphi^\M \, = \bigoplus_{
  \substack{
      \sigma \text{ critical}\\
      v_{\ph}(\sigma)\leq p
  }}
  \frac{R}{(\ph^{v_{\ph}(\sigma)})} \, \bar{e}_{\sigma}. \]
Consider also the following quotient complex:
\[ \mathcal F^p (L_*)_\varphi^\M \,=\, F^p (L_*)_\varphi^\M \, / \, F^{p-1} (L_*)_\varphi^\M \, \cong \bigoplus_{
\substack{
    \sigma \text{ critical}\\
    v_{\ph}(\sigma) = p
}}
\frac{R}{(\ph^{p})} \ \bar{e}_{\sigma}. \]

\begin{theorem}[\cite{salvetti2013combinatorial}]
  Let $\S$ and $\varphi$ be as above and let $\M$ be an acyclic matching for $\Sf$.
  Then the $E^1$-page of the spectral sequence of Theorem \ref{thm:spectral-sequence} is identified with
  \[ E^1_{p,q} \cong H_{p+q}(\F^p (L_*)_\varphi^\M), \]
  where $(L_*)_\varphi^\M$ is the Morse complex of $\S_{\varphi}.$ 
  The differential
  \[ d^1_{p,q}\colon E^1_{p,q} \to E^1_{p-1,q} \]
  is induced by the boundary of the Morse complex, and thus it is also computed by using alternating paths.
  \label{thm:spectral-sequence-morse-complex}
\end{theorem}

\subsection{Weighted sheaves for Artin groups}
\label{sec:weighted-sheaves-artin-groups}

We use the constructions given above in the context of Artin groups.
Let $(\W,S)$ be a Coxeter system, with $|S|=n$ finite, and let $\Gamma$ be the corresponding Coxeter graph (with $S$ as vertex set).
Recall that $\W$ can always be realized as a group of reflections in some $\R^n$ (for example through the \emph{Tits representation}, see \cite{bourbaki1968elements, vinberg1971discrete}, see also \cite{paris2012k}), so that it has the following naturally associated objects:
\begin{enumerate}[(i)]
\item a hyperplane arrangement in $\R^n$ 
\[ \mathcal A \, =\, \{ H \mid H \text{ is the fixed point set of some reflection in $\W$} \}; \]

\item configuration spaces
\begin{align*}
  &\Y = \left(\interior(U)+i\R^n\right) \setminus \bigcup_{H\in \mathcal A} \ H_{\C}, \\
  &\YW  = \Y/\W
\end{align*}
 where $U=\W\cdot \overline{C_0}$ is the {\it Tits cone} (here $C_0$ is a fixed chamber of the arrangement, and $\overline{C_0}$ is its topological closure);

\item a simplicial complex (defined over the finite set $S$)
\[ K = K_{\W} = \{ \sigma \subseteq S \mid  \text{the parabolic subgroup $\W_\sigma$ generated by $\sigma$ is finite} \}. \]
\end{enumerate}

One can define the \emph{Artin group} $\GW$ of type $\W$ as the fundamental group $\pi_1(\YW)$.
It has a presentation
\[ \GW= \< \, g_s, s\in S \mid g_sg_{s'}g_sg_{s'}\ldots = g_{s'}g_sg_{s'}g_s\ldots \, \> \]
(both the products have $m(s,s')$ factors, where $m(s,s')$ is the order of $ss'$ in $\W$).
Recall also the following results \cite[Theorems 1.4 and 1.8]{salvetti1994homotopy}.

\begin{enumerate}[(i)]
  \item The orbit space $\YW$ deformation retracts onto a \emph{finite} $CW$-complex $\X_{\W}$ given by a union
  \[ Q \,=\, \bigcup_{
    \mathclap{\substack{\sigma\in K_\W}}
    } \, Q_\sigma \]
  of convex polyhedra with explicit identifications of their faces. 
  
  \item Consider the action of the Artin group $\GW$ on the ring $R=\Q[q^{\pm 1}]$ given by
  \[ g_s\ \mapsto \  [\text{multiplication by } -q]\quad \forall s\in S. \]
  Then the homology $H_*(\X_{\W};R)$ is computed by the algebraic complex 
  \[ C_k = \, \bigoplus_{\mathclap{\substack{\sigma \in K_\W\\
      |\sigma|=k}}}
    \, R\, e_\sigma \]
  with boundary
  \[  \partial(e_\sigma) \, = \, \sum_{\mathclap{\substack{
    \tau \lhd \sigma }}}
    \; [\sigma:\tau]\, \frac{\W_\sigma(q)}{\W_\tau(q)}\, e_\tau, \]
  where $\displaystyle\W_\sigma(q) = \sum_{w \in \W_\sigma} q^{l(w)}$ is the Poincaré polynomial of $\W_\sigma$.
\end{enumerate}

\begin{remark}
  The $R$-module structure on $H_*(\XW; R)$ is given by the transformation $\mu_q$ induced by $q$-multiplication.
  If the order of $\mu_q$ is $n,$ then the homology groups decompose into cyclic factors which are either free, of the form $R^k,$ or torsion, of the form $R / (\ph_d)$, where $\ph_d$ is the $d$-th cyclotomic polynomial with $d\mid n$.
  Therefore we are interested in localizing to cyclotomic polynomials.
\end{remark}

\begin{theorem}[\cite{salvetti2013combinatorial}]
  To a Coxeter system $(\W,S)$ we can associate a weighted sheaf $(K,R,\w)$ over the simplicial complex $K=K_{\W}$, by setting $\w(\sigma) = \W_\sigma(q)$.
  Also, for any cyclotomic polynomial $\varphi = \varphi_d$, the map 
  \[ \w_\varphi(\sigma) = \text{maximal power of $\varphi$ which divides $\W_\sigma(q)$} \]
  defines a weighted sheaf $(K,R,\w_\varphi)$ over $K=K_{\W}$.
\end{theorem}

The homology of the associated weighted complex is strictly related to the homology of $\XW$.
Specifically, set $C_*^0 = C_*^0(K; R)$ and consider the diagonal map
\[ \Delta\colon C_* \to C_*^0, \quad  e_\sigma  \mapsto \W_\sigma(q) \, e_\sigma^0. \]
By the formula for the boundary map it follows that $\Delta$ is an injective
chain-complex homomorphism, so there is an exact sequence of complexes:
\[ 0 \longrightarrow C_* \xlongrightarrow{\Delta} C^0_* \xlongrightarrow{\pi} L_* \longrightarrow 0, \]
where
\[ L_k \, = \, \bigoplus_{\mathclap{\substack{
    \sigma \in K_\W \\
    |\sigma| = k}}}
  \; \frac{R}{(\W_\sigma(q))}\; \bar{e}_\sigma \]
is the quotient complex.
Passing to the associated long exact sequence we get:
\begin{equation}
  \dots \xrightarrow{\pi_*} H_{k+1}(L_*) \rightarrow H_k(C_*) \xrightarrow{\Delta_*} H_k(C_*^0) \xrightarrow{\pi_*} H_k(L_*) \rightarrow H_{k-1}(C_*) \xrightarrow{\Delta_*} \dots
  \label{eq:long-rank1}
\end{equation}
Then the homology of $L_*$ can be used to compute the homology of $C_*$.

The orbit space $\YW$ (and thus the CW-complex $\XW \simeq \YW$) is conjectured to be a classifying space for the Artin group $\GW$ \cite{deligne1972immeubles, brieskorn1973groupes, van1983homotopy, paris2012k}.
This conjecture was proved for Artin groups of finite type \cite{deligne1972immeubles}, for affine Artin groups of type $\tilde A_n$, $\tilde B_n$ and $\tilde C_n$ \cite{okonek1979dask, callegaro2010k}, and for some other families of Artin groups \cite{hendriks1985hyperplane, charney1995k}.
Whenever the conjecture holds (in particular this is true for all the cases we consider in this paper), the homology $H_*(\XW; R)$ coincides with the twisted homology $H_*(\GW; R)$ of the Artin group $\GW$.

\section{Homology of braid groups and independence complexes}
\label{sec:braid-groups}

In this section we are going to show how the twisted homology of braid groups (i.e.\ Artin groups of type $A_n$) is related to the homology of suitable independence complexes.
Recall that a Coxeter graph of type $A_n$ is a linear graph with $n$ vertices, usually labeled $1$, $2$, \dots, $n$ (see Figure \ref{fig:An}), and that the corresponding Artin group is the braid group on $n+1$ strands (which we denote by $\Br_{n+1}$).
With a slight abuse of notation, by $A_n$ we will sometimes indicate the graph itself.

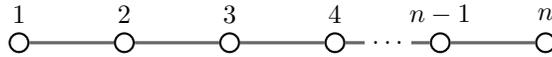
\begin{figure}[htbp]
  \begin{tikzpicture}
\begin{scope}[every node/.style={circle,thick,draw,inner sep=2.5}, every label/.style={rectangle,draw=none}]
  \node (1) at (0.0,0) [label={above,minimum height=13}:$1$] {};
  \node (2) at (1.4,0) [label={above,minimum height=13}:$2$] {};
  \node (3) at (2.8,0) [label={above,minimum height=13}:$3$] {};
  \node (4) at (4.2,0) [label={above,minimum height=13}:$4$] {};
  \node (n-1) at (5.6,0) [label={above,minimum height=13}:$n-1$] {};
  \node (n) at (7.0,0) [label={above,minimum height=13}:$n$] {};
\end{scope}
\begin{scope}[every edge/.style={draw=black!60,line width=1.2}]
  \path (1) edge node {} (2);
  \path (2) edge node {} (3);
  \path (3) edge node {} (4);
  \path (4) edge node [fill=white, rectangle, inner sep=3.0] {$\ldots$} (n-1);
  \path (n-1) edge node {} (n);
\end{scope}
\begin{scope}[every edge/.style={draw=black, line width=3}]
\end{scope}
\begin{scope}[every node/.style={draw,inner sep=11.5,yshift=-4}, every label/.style={rectangle,draw=none,inner sep=6.0}, every fit/.append style=text badly centered]
\end{scope}
\end{tikzpicture}
   \caption{A Coxeter graph of type $A_n$.} 
  \label{fig:An}
\end{figure}

The homology of $\Br_{n+1}$, with coefficients in the representation $R=\Q[q^{\pm 1}]$ described in the previous section, has been computed in \cite{frenkel1988cohomology, de2001arithmetic}.
See also \cite{callegaro2006homology} for coefficients in $\Z[q^{\pm 1}]$ (but we will not address this case).

\begin{theorem}[\cite{frenkel1988cohomology, de2001arithmetic}]
  The $\varphi_d$-primary component of the twisted ho\-mo\-lo\-gy of a braid group is given by
  \[ 
    H_*(\Br_{n+1}; R)_{\varphi_d} =
    \begin{cases}
      R/(\varphi_d) & \text{if $n \equiv 0$ or $-1$ $\pmod{d}$}, \\
      \,0 & \text{otherwise},
    \end{cases}
  \]
  where the non vanishing term is in degree $(d-2)k$ if $n=dk$ or $n=dk-1$.
  \label{teo:simpleind}
\end{theorem}

Recall that, if $G$ is a graph with vertex set $VG$, an \emph{independent set} of $G$ is a subset of $VG$ consisting of pairwise non-adjacent vertices.
Also, the \emph{independence complex} $\Ind(G)$ of $G$ is the abstract simplicial complex with $VG$ as set of vertices and whose simplices are all the non-empty independent sets of $G$.
Thus $\Ind(G)$ is the clique complex of the complement graph of $G$.
In contrast with the simplicial complex $K$ introduced in Section \ref{sec:constructions}, the simplicial complex $\Ind(G)$ does not contain the empty simplex, and the dimension of a simplex $\sigma\in\Ind(G)$ is given by $|\sigma|-1$.
The homotopy type of $\Ind(A_n)$ has been computed in \cite{kozlov2007combinatorial} by means of discrete Morse theory, and the result is the following.
\begin{proposition}[{\cite[Proposition 11.16]{kozlov2007combinatorial}}]
\[ \Ind(A_n) \simeq 
  \begin{cases}
    S^{k-1} & \text{if $n=3k$ or $n=3k-1$}, \\
    \{\pt\} & \text{if $n=3k+1$}.
  \end{cases} \]
\end{proposition}

By comparing this with Theorem \ref{teo:simpleind} we obtain the following relation between the homology of the independence complex of $A_n$ and the $\varphi_3$-primary component of the twisted homology of $\Br_{n+1}$.

\begin{provedcorollary}
  \[ H_*(\Br_{n+1};R)_{\varphi_3} \cong \tilde H_{*-2}\left(\Ind(A_{n-3});\frac{R}{(\varphi_3)}\right), \]
  where on the left we have local coefficients and on the right we have trivial coefficients.
  \label{cor:homology-phi3}
\end{provedcorollary}

In general, following \cite{salvetti2015some}, define the \emph{$r$-independence complex} of a graph $G$ as
\[ \begin{array}{ll} \Ind_r(G)  = \!\!\!\! & \left\{ \mbox{full subgraphs $G'\subseteq G$ such that each connected component} \right.\\
   & \left. \mbox{ of $G'$  has at most $r$ vertices}\right\}.
\end{array}
\]
So $\Ind_r(G)$ is an abstract simplicial complex on the set of vertices $VG$ of the graph $G$, which coincides with $\Ind(G)$ for $r=1$. The case $r=0$ also makes sense: $\Ind_0(G)=\varnothing$ for any graph $G$.
We are going to prove the following generalization of Corollary \ref{cor:homology-phi3}.

\begin{theorem} \label{maintheorem}
  \[ H_*(\Br_{n+1};R)_{\varphi_d} \cong \tilde H_{*-d+1}\left(\Ind_{d-2}(A_{n-d});\frac{R}{(\varphi_d)}\right), \]
  where on the left we have local coefficients and on the right we have trivial coefficients.
\end{theorem}

From the expression of the $\varphi_d$-primary component of the local homology of the braid group (Theorem \ref{teo:simpleind}) we obtain the following consequence.

\begin{corollary}\label{cor:independents}
  \[ \tilde H_*(\Ind_{d-2}(A_n)) \cong
    \begin{cases}
      \tilde H_*(S^{dk-2k-1}) & \text{for $n=dk$ or $n=dk-1$}, \\
      \, 0 & \text{otherwise}.
    \end{cases}
  \]
\end{corollary}

Therefore the knowledge of the twisted homology of the braid group gives the homology (with trivial coefficients) of $\Ind_d(A_n)$.
Conversely the knowledge of the homology of $\Ind_d(A_n)$ gives the twisted homology of the braid group.

The Poincaré polynomial of a Coxeter group of type $A_k$ is given by
\[ \W_{A_k}(q)= [k+1]_q!\quad \text{where} \quad  [k]_q=\frac{q^k-1}{q-1} = \prod_{\substack{d\mid k\\ d\geq 2}}\ \varphi_d \]
(see for example \cite{bjorner2006combinatorics}).
Consider now a simplex $\sigma\subseteq S\cong \{ 1,\dots, n \}$. Denote by $\Gamma(\sigma)$ the subgraph of $A_n$ induced by $\sigma$. Denote by $\Gamma_1(\sigma)$, $\Gamma_2(\sigma)$, \dots, $\Gamma_m(\sigma)$ the connected components of $\Gamma(\sigma)$, and by $n_1$, $n_2$, \dots $n_m$ their cardinalities (see Figure \ref{fig:An-components}).
The $i$-th connected component is a Coxeter graph of type $A_{n_i}$, so the entire Coxeter graph induced by $\sigma$ has Poincaré polynomial
\begin{align*}
  \W_\sigma (q) &= [n_1+1]_q! \cdot [n_2+1]_q! \cdots [n_m+1]_q! \\
  &= \;\prod_{i=1}^m \; \prod_{d\geq 2} \, \varphi_d^{\lfloor \frac{n_i+1}{d} \rfloor} \\
  &= \; \prod_{d\geq 2} \varphi_d^{\;\sum_{i=1}^m \lfloor \frac{n_i+1}{d} \rfloor}.
\end{align*}
Then we are interested in the homology of the weighted complex $(L_*)_{\varphi}$ associated to the weighted sheaf $(K,R,\w_\varphi)$, where $\varphi=\varphi_d$ is a cyclotomic polynomial (with $d\geq 2$) and
\[ \w_\varphi(\sigma) = \varphi^{v_\varphi(\sigma)}, \quad v_\varphi(\sigma) = \sum_{i=1}^m \textstyle \left\lfloor \frac{n_i+1}{d} \right\rfloor. \]
Notice that only the connected components with at least $d-1$ vertices contribute to the $\varphi$-weight.
Therefore, the weighted complex $(L_*)_{\varphi}$ is generated by the subgraphs having at least one component with $\geq d-1$ vertices.

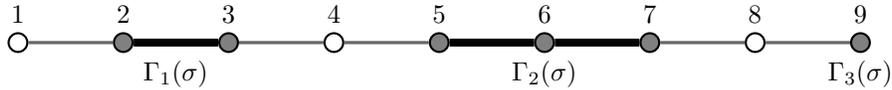
\begin{figure}[htbp]
  \begin{tikzpicture}
\begin{scope}[every node/.style={circle,thick,draw,inner sep=2.5}, every label/.style={rectangle,draw=none}]
  \node (1) at (0.0,0) [label={above,minimum height=13}:$1$] {};
  \node (2) at (1.4,0) [label={above,minimum height=13}:$2$,fill=black!50] {};
  \node (3) at (2.8,0) [label={above,minimum height=13}:$3$,fill=black!50] {};
  \node (4) at (4.2,0) [label={above,minimum height=13}:$4$] {};
  \node (5) at (5.6,0) [label={above,minimum height=13}:$5$,fill=black!50] {};
  \node (6) at (7.0,0) [label={above,minimum height=13}:$6$,fill=black!50] [label={below}:$\Gamma_2(\sigma)$] {};
  \node (7) at (8.4,0) [label={above,minimum height=13}:$7$,fill=black!50] {};
  \node (8) at (9.8,0) [label={above,minimum height=13}:$8$] {};
  \node (9) at (11.2,0) [label={above,minimum height=13}:$9$,fill=black!50] [label={below}:$\Gamma_3(\sigma)$] {};
\end{scope}
\begin{scope}[every edge/.style={draw=black!60,line width=1.2}]
  \path (1) edge node {} (2);
  \path (3) edge node {} (4);
  \path (4) edge node {} (5);
  \path (7) edge node {} (8);
  \path (8) edge node {} (9);
\end{scope}
\begin{scope}[every edge/.style={draw=black, line width=3}]
  \path (2) edge node {} node [below=2] {$\Gamma_1(\sigma)$} (3);
  \path (5) edge node {} (6);
  \path (6) edge node {} (7);
\end{scope}
\begin{scope}[every node/.style={draw,inner sep=11.5,yshift=-4}, every label/.style={rectangle,draw=none,inner sep=6.0}, every fit/.append style=text badly centered]
\end{scope}
\end{tikzpicture}
   \caption{An example with $n=9$ and $\sigma=\{2,3,5,6,7,9\} \in \Ind_3(A_9)$. In this case $|\Gamma_1(\sigma)|=2,\ |\Gamma_2(\sigma)|=3,\ |\Gamma_3(\sigma)|=1$.} 
  \label{fig:An-components}
\end{figure}

\begin{theorem}
\label{thm:casoAn}
There is a weighted acyclic matching $\mathcal M$ on $K$ such that
the set of critical simplices is given by
\begin{eqnarray*}
  \Cr(\M) \!\! &=& \!\! \{ \sigma \mid \Gamma(\sigma) = \Gamma_1 \sqcup\dots\sqcup\Gamma_{m-1} \sqcup A_{d-1}, \; |\Gamma_i|\leq d-2\} \,\cup\, \Ind_{d-2}(A_n) \\
  \!\! &=& \!\! \{ \tau \sqcup A_{d-1} \mid \tau \in \Ind_{d-2}(A_{n-d}) \} \,\cup\, \Ind_{d-2}(A_n),
\end{eqnarray*}
where $A_{d-1}$ is the linear graph on the vertices $n-d+2$, \dots, $n$.
\end{theorem}

\begin{proof}
First notice that removing the $d$-th vertex from an $A_k$ component leaves the $\varphi$-weight unchanged:
\[ v_\varphi(A_k) = \textstyle \left\lfloor\frac{k+1}{d}\right\rfloor = 1+\left\lfloor\frac{k+1-d}{d}\right\rfloor = v_\varphi(A_{d-1}\sqcup A_{k-d}) \]
(see Figure \ref{fig:Ak-matching}).
\begin{figure}[htbp] 
  \begin{tikzpicture}
\begin{scope}[every node/.style={circle,thick,draw,inner sep=2.5}, every label/.style={rectangle,draw=none}]
  \node (0) at (0.0,0) [label={above,minimum height=13}:$ $,fill=black!50] {};
  \node (1) at (1.3,0) [label={above,minimum height=13}:$ $,fill=black!50] {};
  \node (2) at (2.6,0) [label={above,minimum height=13}:$ $,fill=black!50] {};
  \node (3) at (3.9,0) [label={above,minimum height=13}:$ $,fill=black!50] {};
  \node (4) at (5.2,0) [label={above,minimum height=13}:$ $,fill=black!50] {};
  \node (5) at (6.5,0) [label={above,minimum height=13}:$ $,fill=black!50] {};
  \node (6) at (7.8,0) [label={above,minimum height=13}:$ $,fill=black!50] {};
  \node (7) at (9.1,0) [label={above,minimum height=13}:$ $,fill=black!50] {};
\end{scope}
\begin{scope}[every edge/.style={draw=black!60,line width=1.2}]
\end{scope}
\begin{scope}[every edge/.style={draw=black, line width=3}]
  \path (0) edge node {} (1);
  \path (1) edge node [fill=white, rectangle, inner sep=3.0, minimum height = 0.5cm] {$\ldots$} (2);
  \path (2) edge node {} (3);
  \path (3) edge node {} node [below=2] {$k$\;vertices} (4);
  \path (4) edge node [fill=white, rectangle, inner sep=3.0, minimum height = 0.5cm] {$\ldots$} (5);
  \path (5) edge node {} (6);
  \path (6) edge node {} (7);
\end{scope}
\begin{scope}[every node/.style={draw,inner sep=11.5,yshift=-4}, every label/.style={rectangle,draw=none,inner sep=6.0}, every fit/.append style=text badly centered]
\end{scope}
\end{tikzpicture}
   \vskip0.3cm
  \begin{tikzpicture}
\begin{scope}[every node/.style={circle,thick,draw,inner sep=2.5}, every label/.style={rectangle,draw=none}]
  \node (0) at (0.0,0) [label={above,minimum height=13}:$ $,fill=black!50] {};
  \node (1) at (1.3,0) [label={above,minimum height=13}:$ $,fill=black!50] [label={below}:$d-1$\;vertices] {};
  \node (2) at (2.6,0) [label={above,minimum height=13}:$ $,fill=black!50] {};
  \node (3) at (3.9,0) [label={above,minimum height=13}:$ $] {};
  \node (4) at (5.2,0) [label={above,minimum height=13}:$ $,fill=black!50] {};
  \node (5) at (6.5,0) [label={above,minimum height=13}:$ $,fill=black!50] {};
  \node (6) at (7.8,0) [label={above,minimum height=13}:$ $,fill=black!50] {};
  \node (7) at (9.1,0) [label={above,minimum height=13}:$ $,fill=black!50] {};
\end{scope}
\begin{scope}[every edge/.style={draw=black!60,line width=1.2}]
  \path (2) edge node {} (3);
  \path (3) edge node {} (4);
\end{scope}
\begin{scope}[every edge/.style={draw=black, line width=3}]
  \path (0) edge node {} (1);
  \path (1) edge node [fill=white, rectangle, inner sep=3.0, minimum height = 0.5cm] {$\ldots$} (2);
  \path (4) edge node [fill=white, rectangle, inner sep=3.0, minimum height = 0.5cm] {$\ldots$} (5);
  \path (5) edge node {} node [below=2] {$k-d$\;vertices} (6);
  \path (6) edge node {} (7);
\end{scope}
\begin{scope}[every node/.style={draw,inner sep=11.5,yshift=-4}, every label/.style={rectangle,draw=none,inner sep=6.0}, every fit/.append style=text badly centered]
\end{scope}
\end{tikzpicture}
   \caption{The $\varphi$-weight of an $A_k$ component remains the same if we remove the $d$-th vertex, splitting $A_k$ into $A_{d-1} \sqcup A_{k-d}$.} 
  \label{fig:Ak-matching}
\end{figure}
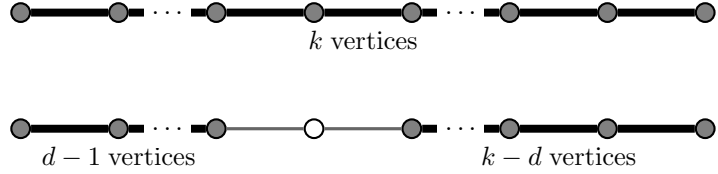

Let $K_0 = \Ind_{d-2}(A_n) \subseteq K$ (this is the set of the simplices $\sigma \in K$ such that all the connected components of the induced subgraph $\Gamma(\sigma)$ have cardinality $< d-1$).
Let us define on $K' = K \setminus K_0$ the following matching $\M$.
For $\sigma \in K'$, with $\Gamma(\sigma) = \Gamma_1 \sqcup \dots \sqcup \Gamma_m$, set
\[ i(\sigma) = \min\{i \mid |\Gamma_i|\geq d-1\}. \]
Then match $\sigma$ with the simplex $\tau$ obtained by adding or removing from $\sigma$ the $d$-th vertex of the component $\Gamma_{i(\sigma)}$.
By the remark at the beginning of the proof, $\mathcal M$ is a weighted matching. We prove that it is acyclic. 
In fact, suppose that an alternating path contains some subpath
\[ \tau \lhd \sigma \rhd \tau' \lhd \sigma' \]
with $\tau \neq \tau'$.
Then either $i(\tau) < i(\tau')$ or (set $j=i(\tau) = i(\tau')$)
\[ \Gamma_j(\tau) = \{ a, a+1,\dots, a+d-2 \},\quad \Gamma_j(\tau') = \{a+1,\dots, a+d-1 \} \]
for some $a\leq n-d+1.$ In this case the first vertex of $\Gamma_j(\tau')$ is greater than the first vertex of $\Gamma_j(\tau)$.
Therefore an alternating path in $K'$ cannot be closed.

The set of critical elements of $\M$ in $K'$ is given by
\[ \Cr(\M)= \{ \sigma \mid \Gamma(\sigma) = \Gamma_1 \sqcup\dots\sqcup\Gamma_{m-1} \sqcup A_{d-1}, \; |\Gamma_i|\leq d-2\}, \]
where  $A_{d-1}$ is as in the statement of the theorem.
\end{proof}

\begin{proof}[Proof of Theorem \ref{maintheorem}]
  \label{proof:maintheorem}
  Consider the matching $\M$ of Theorem \ref{thm:casoAn}.
  Since $K_0 = \Ind_{d-2}(A_n)$ does not contribute to the weighted complex $(L_*)_\varphi$, we concentrate on the poset $K' = K \setminus K_0$.
  Notice that there are no non-trivial alternating paths between critical elements of $K'$.
  Therefore the boundaries between simplices in
  \[ \Cr'(\M) = \{ \tau \sqcup A_{d-1} \mid \tau \in \Ind_{d-2}(A_{n-d}) \} \]
  are the same as the boundaries of $\Ind_{d-2}(A_{n-d})$.
  In the long exact sequence \eqref{eq:long-rank1} the algebraic complex $C^0_*$ has vanishing homology in all degrees (because $K$ is the full simplicial complex on $n$ vertices), so we have an isomorphism $H_{*+1}(L_*) \cong H_*(C_*)$. Therefore
  \begin{align*}
    H_*(\Br_{n+1};R)_{\varphi_d} &\cong H_*(C_*)_{\varphi_d} \\
    & \cong H_{*+1}(L_*)_{\varphi_d} \\
    & \cong \tilde H_{*-d+1} \left(\Ind_{d-2}(A_{n-d}); \frac{R}{(\varphi_d)}\right).
  \end{align*}
  In the last isomorphism there is a $(d-1)$-shift in degree due to the loss of $d-1$ vertices when passing from $\Cr'(\M)$ to $\Ind_{d-2}(A_{n-d})$; there is then a further $1$-shift in degree due to the fact that in $L_*$ a simplex $\sigma$ has dimension $|\sigma|$, and in $\Ind_{d-2}(A_{n-d})$ it has dimension $|\sigma|-1$; finally, it is necessary to pass to reduced homology because the empty simplex is missing in $\Ind_{d-2}(A_{n-d})$.
\end{proof}

For the sake of completeness we also determine the homotopy type of the $r$-independence complex $\Ind_{d-2}(A_n)$ (obtaining again Corollary \ref{cor:independents} as a consequence).
This is a straightforward generalization of the case $d=3$ proved in \cite[Proposition 11.16]{kozlov2007combinatorial}.

\begin{proposition}
\[ \Ind_{d-2}(A_n) \simeq 
  \begin{cases}
    S^{dk-2k-1} & \text{if $n=dk$ or $n=dk-1$}, \\
    \{\pt\} & \text{otherwise}.
  \end{cases} \]
\end{proposition}

\begin{proof}
  We denote here for brevity $K_0 = \Ind_{d-2}(A_n)$.
  Let $n=qd+r$ be the euclidean division of $n$ by $d$.

  Let $P=\{ c_d>c_{2d}>c_{3d}>\dots> c_{qd}> c_*\}$ be a linearly ordered set, where $c_*$ is the minimum element.
  Define a map $f\colon K_0\to P$ as follows.
  If $\sigma\in K_0$ does not contain any multiples of $d$, then set $f(\sigma)=c_*$.
  Otherwise, set $f(\sigma)=c_{jd}$ if $jd\in\sigma$ but $j'd\not\in\sigma$ for $j'<j$.
  Clearly $f$ is a poset map (if we remove a vertex from $\sigma$, $j$ increases or remains the same).

  In $f^{-1}(c_d)$ consider the matching $\mathcal M_1=\{ (\sigma\setminus \{1\}, \sigma\cup \{ 1 \} ) \mid \sigma \in f^{-1}(c_d) \}$,
  which is justified by:
  \[ \sigma\in f^{-1}(c_d)\, \Rightarrow\, \{2,\dots,d-1\}\not\subseteq\sigma. \]
  Similarly, in $f^{-1}(c_{jd})$ for $j\leq q$, consider the matching
  \[ \mathcal M_j=\Big\{\Big(\sigma \setminus \{ (j-1)d+1\},\; \sigma \cup \{ (j-1)d+1\} \Big) \;\big|\; \sigma \in f^{-1}(c_{jd}) \Big\}, \]
  justified by:
  \[ \sigma\in f^{-1}(c_{jd})\, \Rightarrow\, \{(j-1)d+2,\dots,jd-1\}\not\subseteq\sigma. \]
  Each $\mathcal M_j$ is acyclic in $f^{-1}(c_{jd})$, thus
  \[ \M = \bigcup_{j=1}^q  \M_j \]
  is an acyclic matching on $K_0$ by the Patchwork Theorem (see for example \cite[Lemma 4.2]{jonsson2008simplicial}, \cite[Theorem 11.10]{kozlov2007combinatorial}).
  Since each $\mathcal M_j$ is a perfect matching on $f^{-1}(c_{jd})$ for $j=1,\dots, q$, and $K_1=f^{-1}(c_*)$ is a subcomplex of $K_0$, it follows that $K_0$ deformation retracts onto $K_1$. 

  Notice now that $K_1$ is the join of $q$ copies of $\Ind_{d-2}(A_{d-1})\cong S^{d-3}$ and one copy of $\Ind_{d-2}(A_{r})$.
  For $r=0$ we get
  \[ K_0\cong (S^{d-3})^{*q}\cong S^{q(d-3)+q-1} = S^{qd-2q-1} \]
  and for $r=d-1$ we get
  \[ K_0\cong (S^{d-3})^{*(q+1)}\cong S^{(q+1)(d-3)+q} = S^{(q+1)d-2(q+1)-1}, \]
  which give the first case in the corollary.
  If $r$ is not $0$ or $d-1$ then $\Ind_{d-2}(A_r)$ is a full simplex $\Delta^{r-1}$ on $r$ vertices, and therefore $K_0$ is contractible.
\end{proof}

\section{Precise matchings}
\label{sec:precise-matchings}

In this section we are going to introduce the notion of \emph{precise matching} on a weighted sheaf.
The motivation comes from the study of the twisted homology of some families of Artin groups, as we will show in the subsequent sections.
For now we go back to the general framework of weighted sheaves over simplicial complexes (later we will specialize in the case of Artin groups).

Assume from now on that the PID $R$ contains some field $\K$. Our main case of interest is $\K=\Q$ and $R=\Q[q^{\pm1}]$.
Let $\S=(K,R,\w)$ be a weighted sheaf over the finite simplicial complex $K$, with associated weighted complex $L_*$. Given a fixed irreducible element $\varphi$ of $R$, let $\S_\varphi$ be the $\varphi$-primary component of $\S$ and let $(L_*)_\varphi$ be its associated weighted complex.
Let $\M$ be a weighted acyclic matching for $\S_\varphi$.

Let $G^\M$ be the incidence graph of the corresponding Morse complex: the vertices of $G^\M$ are the critical simplices of $K$, and there is an (oriented) edge $\sigma \to \tau$ whenever $[\sigma:\tau]^\M$
is not $0$ in $R$ (or equivalently in $\K$), where $[\sigma:\tau]^\M \in \Z$ is the incidence number between $\sigma$ and $\tau$ in the Morse complex of $K$.
In other words, there is an edge $\sigma \to \tau$ if $[\sigma:\tau]^\M$ is not multiple of $\ch\K = \ch R$.
When $\K=\Q$, this simply means that $[\sigma:\tau]^\M \neq 0$.

Let $\I$ be the set of connected components of $G^\M$ (computed ignoring the orientation of the edges).
Recall that $v_\varphi(\sigma)$ is the maximal $k\in\N$ such that $\varphi^k$  divides $\w(\sigma)$.

\begin{definition}
  The matching $\M$ is \emph{$\varphi$-precise} (or simply \emph{precise}) if, for any edge $\sigma\to \tau$ of $G^\M$, we have that $v_\varphi(\sigma) = v_\varphi(\tau)+1$.
  \label{def:precise-matching}
\end{definition}

In other words $\M$ is precise if, for any two simplices $\sigma$ and $\tau$ lying in the same connected component $i\in\I$, the following relation holds:
\[ v_\varphi(\sigma) - v_\varphi(\tau) = |\sigma| - |\tau|. \]
Equivalently the quantity $|\sigma| - v_\varphi(\sigma)$, as a function of $\sigma$, is constant within a fixed connected component of $G^\M$.
This definition is motivated by the fact that precise matchings exist in many cases of interest (we will see this in Section \ref{sec:precise-matchings-artin-groups}), and that the homology of the Morse complex is much simpler to compute (and takes a particularly nice form) when the matching is precise.
The name ``precise'' has been chosen because for a generic matching one only has $v_\varphi(\sigma) \geq v_\varphi(\tau)$ (when $\sigma\to \tau$ is an edge of $G^\M$), and we require $v_\varphi(\sigma)$ to be precisely $v_\varphi(\tau)+1$.

Assume from now on in this section that $\M$ is a $\varphi$-precise matching.
To simplify the notation, set $(A_*, \de) = ((L_*)_\varphi^\M, \de^\M)$ and $(V_*, \delta) = (C^0_*(K,\K)^\M, \delta^\M)$.
Our aim is to derive a formula for the homology of the Morse complex $A_*$.
Since the differential $\delta$ vanishes between simplices in different connected components of $G^\M$, the complex $(V_*, \delta)$ splits as follows:
\[ (V_*, \delta) = \, \bigoplus_{i\in \I} \; (V_*^i, \delta^i), \]
where
\[ V^i_* = \, \bigoplus_{\mathclap{\substack{
    \sigma \text{ critical}\\
    \sigma \in i}}} \; \K \, \bar e_\sigma \]
and the boundary map $\delta^i\colon V^i_*\to V^i_*$ is the restriction of $\delta$ to $V^i_*$.
The differential $\de$ of $A_*$ is induced by $\delta$, and thus it also vanishes between simplices in different connected components.
Therefore we have an analogous splitting for $(A_*, \de)$:
\[ (A_*, \de) = \, \bigoplus_{i\in \I} \, (A_*^i, \de^i), \]
where
\[ A^i_* = \bigoplus_{\substack{
    \sigma \text{ critical}\\
    \sigma \in i}} \frac{R}{(w_\varphi(\sigma))} \, \bar e_\sigma \]
and the boundary map $\de^i\colon A^i_*\to A^i_*$ is simply the restriction of $\de$ to $A^i_*$.

Fix now a connected component $i\in \I$.
Since $\M$ is precise, there exists some $k\in \Z$ (which depends on $i$) such that $v_\varphi(\sigma) = |\sigma| + k$ for all $\sigma\in i$. Therefore in degree $m$ we have
\[ A_m^i = \bigoplus_{\sigma \in \Cr^i_m} \frac{R}{(\varphi^{m+k})} \, \bar e_\sigma
   = \left( \bigoplus_{\sigma \in \Cr^i_m} \K \, \bar e_\sigma \right) \otimes_{\K} \frac{R}{(\varphi^{m+k})}
   = V_m^i \otimes_\K \frac{R}{(\varphi^{m+k})},
\]
where $\Cr^i_m$ is the set of critical simplices $\sigma$ such that $\sigma \in i$ and $|\sigma|=m$.
By construction the boundary $\de^i_m \colon A^i_m \to A^i_{m-1}$ factors accordingly:
\[ \de^i_m = \delta^i_m \otimes_\K \pi_m, \]
where
\[ \pi_m\colon \frac{R}{(\varphi^{m+k})} \to \frac{R}{(\varphi^{m+k-1})} \]
is the projection induced by the identity $R \to R$.
Since $\Im \delta^i_{m+1} \subseteq \ker \delta^i_m$, each $V^i_m$ splits (as a vector space over $\K$) as a direct sum of linear subspaces:
\[ V^i_m = W^i_{m,1} \oplus W^i_{m,2} \oplus W^i_{m,3}, \]
where $W^i_{m,1} = \Im \delta^i_{m+1}$ and $W^i_{m,1}\oplus W^i_{m,2} = \ker \delta^i_m$.
Then
\begin{align*}
  \ker( \de^i_m ) &= \ker( \delta^i_m \otimes_\K \pi_m ) \\
  &= \left( \left( W^i_{m,1} \oplus W^i_{m,2} \right) \otimes_\K \frac{R}{(\varphi^{m+k})} \right) \oplus \left(W^i_{m,3} \otimes_\K \frac{(\varphi^{m+k-1})}{(\varphi^{m+k})} \right); \\
  \Im( \de^i_{m+1} ) &= \Im( \delta^i_{m+1} \otimes_\K \pi_{m+1} ) \\
  &= W^i_{m,1} \otimes_\K \frac{R}{(\varphi^{m+k})}.
\end{align*}
Therefore the homology of $(A^i_*, \delta^i)$ is given, as an $R$-module, by
\begin{align*}
  H_m(A^i_*) &= \frac{ \ker(\de^i_m) }{ \Im(\de^i_{m+1}) } \\
  &= \left( W^i_{m,2} \otimes_\K \frac{R}{(\varphi^{m+k})} \right) \oplus \left(W^i_{m,3} \otimes_\K \frac{(\varphi^{m+k-1})}{(\varphi^{m+k})} \right) \\
  &\cong \left( H_m(V_*^i, \delta^i) \otimes_\K \frac{R}{(\varphi^{m+k})} \right) \oplus \left( \K^{\rk \delta^i_m} \otimes_\K \frac{R}{(\varphi)} \right).
\end{align*}
In the last isomorphism we used the fact that $\dim W^i_{m,3} = \dim V^i_m - \dim \ker \delta^i_m = \rk \delta^i_m$.

Recall that the previous formula holds for a fixed connected component $i\in \I$, and $k$ depends on $i$.
Since we now need to take the direct sum over the connected components, let $k_i$ be the value of $k$ for the component $i$.

\begin{theorem}
  The homology of $(L_*)_\varphi$ is given, as an $R$-module, by
  \[ H_m((L_*)_\varphi, \de) \cong
  \left( \bigoplus_{i\in \I} H_m(V_*^i) \otimes_\K \frac{R}{(\varphi^{m+k_i})} \right) \oplus \left( \K^{\rk \delta_m} \otimes_\K \frac{R}{(\varphi)} \right). \]
  \label{thm:homology-phi-component}
\end{theorem}

\begin{proof}
  By Theorem \ref{thm:morsecomplex}, $H_m((L_*)_\varphi, \de) \cong H_m((L_*)_\varphi^\M, \de^\M)$.
  Using what we have done in this section, we have that
  \begin{align*}
    H_m((L_*)_\varphi^\M, \de^\M) &= H_m(A_*, \de) \\
    &= \bigoplus_{i\in \I} H_m(A_*^i, \de) \\
    &\cong \left( \bigoplus_{i\in \I} H_m(V_*^i) \otimes_\K \frac{R}{(\varphi^{m+k_i})} \right) \oplus \left( \left( \bigoplus_{i\in \I} \K^{\rk \delta^i_m} \right) \otimes_\K \frac{R}{(\varphi)} \right) \\
    &\cong \left( \bigoplus_{i\in \I} H_m(V_*^i) \otimes_\K \frac{R}{(\varphi^{m+k_i})} \right) \oplus \left( \K^{\rk \delta_m} \otimes_\K \frac{R}{(\varphi)} \right). \qedhere
  \end{align*}
\end{proof}

Let us see how the existence of a precise matching can be interpreted in terms of the spectral sequence associated to the weighted sheaf (see Theorem \ref{thm:spectral-sequence}).

\begin{proposition}
  If a $\varphi$-precise matching $\M$ exists, then the spectral sequence $E_{p,q}^r$ associated to the weighted sheaf $\S_\varphi$ collapses at the $E^2$-page.
  \label{prop:precise-matching-spectral-sequence}
\end{proposition}

\begin{proof}
  By Theorem \ref{thm:spectral-sequence-morse-complex}, the $E^1$-page can be computed through the Morse complex of our matching $\M$:
  \[ E^1_{p,q} \cong  H_{p+q}(\F^p (L_*)_\varphi^\M), \]
  and the differential $d^1_{p,q}$ is induced by the boundary of the Morse complex.
  The spectral sequence then splits as a direct sum over the connected components of $\smash{G^\M}$:
  \[ E_{p,q}^r = \bigoplus_{i\in \I} E_{p,q}^{r,i}, \]
  where $E^{1,i}_{p,q} \cong H_{p+q}(\F^p A^i_*)$.
  Since the matching is precise, for $m \neq p-k_i$ we have
  \[ \F^p A^i_m \,=\, F^p A^i_m \,/\, F^{p-1} A^i_m = 0. \]
  This means that the page $E^{0,i}_{p,q}\cong E^{1,i}_{p,q}$ is non-trivial only in the row $q = -k_i$, and the entire spectral sequence $E_{p,q}^r$ collapses at the $E^2$-page.
\end{proof}

What we have done so far in this section assumed $\varphi$ to be some fixed irreducible element of the PID $R$.
In order to recover the full homology of $L_*$ we need to make $\varphi$ vary among all equivalence classes of irreducible elements of $R$ modulo the units.
Suppose from now on to have a $\varphi$-precise matching $\M_\varphi$ on $\S_\varphi$, for each $\varphi$.
The following results follows immediately from Theorem \ref{thm:homology-phi-component}, provided that we add a ``$\varphi$'' subscript (or superscript) to all the quantities that depend on the matching $\M_\varphi$.

\begin{provedtheorem}
  The homology of $L_*$ is given, as an $R$-module, by
  \[ H_m(L_*, \de) \cong
  \bigoplus_\varphi \left( \bigoplus_{i\in \I_\varphi} H_m(V_*^{\varphi,i}) \otimes_\K \frac{R}{(\varphi^{m+k_{\varphi,i}})} \right) \oplus \left( \K^{\rk \delta_m^\varphi} \otimes_\K \frac{R}{(\varphi)} \right). \qedhere \]
  \label{thm:homology-L}
\end{provedtheorem}

For later applications, we finally need to study how the isomorphism of Theorem \ref{thm:homology-L} behaves with respect to the projection $\pi\colon C^0_* \to L_*$ of Remark \ref{rmk:projections}.
This projection is the direct sum over $\varphi$ of the projections
\[ \pi_\varphi \colon C^0_* \to (L_*)_\varphi. \]
Instead of studying the induced map $(\pi_\varphi)_*\colon H_m(C^0_*) \to H_m((L_*)_\varphi)$, we study the map
\[ (\bar \pi_\varphi)_* \colon H_m((C^0_*)^{\M}) \to H_m((L_*)^{\M}_\varphi) \]
between the Morse complexes (here $\M=\M_\varphi$ is a precise matching which depends on $\varphi$).
For $i\in \I_\varphi$, let $\pi_i \colon (C^0_*)^\M \to V^{\varphi,i}_* \otimes_\K R \subseteq (C^0_*)^\M$ be the projection on the subcomplex corresponding to the connected component $i$, and let $(\pi_i)_*$ be the map induced in homology.
Let $[c] \in H_m((C^0_*)^\M)$, for some cycle $c \in \ker \delta^\M \subseteq (C^0_m)^\M$.
Applying the map $(\bar \pi_\varphi)_* \colon H_m((C_*^0)^\M) \to H_m((L_*)^\M_\varphi)$ we obtain
\begin{align*}
  (\bar \pi_\varphi)_* ([c])
  &= (\bar \pi_\varphi)_* \left( \sum_{i\in \I_\varphi} (\pi_i)_* ([c]) \right) \\
  &= \sum_{i\in \I_\varphi} (\bar\pi_\varphi)_* \Big((\pi_i)_* ([c]) \Big).
\end{align*}
Applying the isomorphism of Theorem \ref{thm:homology-phi-component}, this element is sent to
\[ \sum_{i\in \I_\varphi} \Big((\pi_i)_* ([c])\Big) \otimes_\K [1]\; \in\, \bigoplus_{i\in \I_\varphi} H_m(V_*^{\varphi,i}) \otimes_\K \frac{R}{(\varphi^{m+k_{\varphi,i}})}. \]

We are going to use these computations to prove the following result which describes the kernel and the cokernel of $\pi_*$.

\begin{proposition}
  The cokernel of $\pi_*\colon H_m(C^0_*) \to H_m(L_*)$ is given by
  \[ \coker \pi_* \cong \, \bigoplus_\varphi \, \left( \frac{R}{(\varphi)} \right)^{\oplus \rk\delta_m^\varphi}. \]
  In addition, the kernel of $\pi_*$ is a free $R$-module isomorphic to $H_m(C^0_*)$.
  \label{prop:kernel-and-cokernel}
\end{proposition}

\begin{proof}
  Throughout the proof, consider the following $R$-modules identified one with each other, without explicitly mentioning the isomorphisms between them:
  \begin{align*}
   H_m(L_*) &\cong \bigoplus_\varphi H_m((L_*)_\varphi) \\
   &\cong \bigoplus_\varphi H_m((L_*)^\M_\varphi) \\
   &\cong \bigoplus_\varphi \left( \bigoplus_{i\in \I_\varphi} H_m(V_*^{\varphi,i}) \otimes_\K \frac{R}{(\varphi^{m+k_{\varphi,i}})} \right) \oplus \left( \K^{\rk \delta_m^\varphi} \otimes_\K \frac{R}{(\varphi)} \right).
  \end{align*}
  Recall that the matching $\M$ depends on $\varphi$, although we write $\M$ instead of $\M_\varphi$ in order to make the notations more readable.
  Also recall that the isomorphisms $H_m((L_*)_\varphi) \cong H_m((L_*)^\M_\varphi)$ occur in the commutative diagram of Remark \ref{rmk:projections}.
  
 We want to show that the image of $\pi_*\colon H_m(C^0_*) \to H_m(L_*)$ is given by
  \[ \Im \pi_* = \, \bigoplus_\varphi \left( \bigoplus_{i\in \I_\varphi} H_m(V_*^{\varphi,i}) \otimes_\K \frac{R}{(\varphi^{m+k_{\varphi,i}})} \right)
  \subseteq H_m(L_*). \]
Let $\psi_\varphi\colon H_*(C^0_*) \to H_*((L_*)^\M_\varphi)$ be the map defined as the composition
  \[ H_*(C^0_*) \xlongrightarrow{\cong} H_*((C^0_*)^\M) \xrightarrow{(\bar \pi_\varphi)_*} H_*((L_*)^\M_\varphi). \]
  By commutativity of the diagram of Remark \ref{rmk:projections}, the image of $\pi_* = \bigoplus_\varphi(\pi_\varphi)_*$ is the same as the image of
  \[ \bigoplus_\varphi \psi_\varphi \colon H_*(C^0_*) \to \bigoplus_\varphi H_*((L_*)^\M_\varphi). \]
  We have already proved that, for any $[c]\in H_m((C^0_*)^\M)$,
  \[ (\bar \pi_\varphi)_*([c]) = \sum_{i\in \I_\varphi} \Big((\pi_i)_* ([c])\Big) \otimes_R [1]\; \in\, \bigoplus_{i\in \I_\varphi} H_m(V_*^{\varphi,i}) \otimes_\K \frac{R}{(\varphi^{m+k_{\varphi,i}})}, \]
  which means in particular that
  \[ \Im (\bar \pi_\varphi)_* \subseteq \bigoplus_{i\in \I_\varphi} H_m(V_*^{\varphi,i}) \otimes_\K \frac{R}{(\varphi^{m+k_{\varphi,i}})}. \]
  Therefore we immediately have the inclusion
  \[ \Im(\pi_*) \subseteq \sum_\varphi \Im \psi_\varphi = \sum_\varphi \Im (\bar \pi_\varphi)_* \subseteq \bigoplus_\varphi \left( \bigoplus_{i\in \I_\varphi} H_m(V_*^{\varphi,i}) \otimes_\K \frac{R}{(\varphi^{m+k_{\varphi,i}})} \right). \]
  To prove the opposite inclusion, we show that any element of the form
  \[ [c] \otimes_\K [1] \in H_m(V_*^{\varphi,i}) \otimes_\K \frac{R}{(\varphi^{m+k_{\varphi,i}})} \]
  is in the image of $\pi_*$ (for any fixed $\varphi$ and $i$).
  To do so, choose $\alpha\in R$ such that $\alpha \equiv 1 \pmod{\varphi^{m+k_{\varphi,i}}}$ and $\alpha \equiv 0 \pmod{\eta^{m+k_{\eta, j}}}$ for any irreducible element $\eta \neq \varphi$ which divides some weight $\w(\sigma)$, and for any connected component $j\in \I_{\eta}$ (there is only a finite number of such $\eta$ up to multiplication by units, because $K$ is finite).
  The element $c\otimes_\K \alpha$ is a cycle in $V_m^{\varphi,i} \otimes_\K R \subseteq C^0_m(K,\K)^\M \otimes_\K R \cong (C^0_m)^\M$. Then, if $[\tilde c]$ is the preimage of $[c\otimes_\K \alpha]$ under the isomorphism $H_*(C^0_*) \xlongrightarrow{\cong} H_*((C^0_*)^\M)$, we have that:
  \begin{align*}
    \psi_\varphi([\tilde c]) &= (\bar \pi_\varphi)_*([c \otimes_\K \alpha]) = [c] \otimes_\K [1] \in H_m(V_*^{\varphi,i}) \otimes_\K \frac{R}{(\varphi^{m+k_{\varphi,i}})}; \\
    \psi_\eta([\tilde c]) &= (\bar \pi_\eta)_*([c' \otimes_\K \alpha]) = 0 \quad \text{for any $\eta\neq \varphi$},
  \end{align*}
  where $[c']$ is the image of $[c]$ under the isomorphism
  \[ H_*((C^0_*)^{\M_\varphi}) \xlongrightarrow{\cong} H_*(C^0_*) \xlongrightarrow{\cong} H_*((C^0_*)^{\M_\eta}). \]
  Therefore $[c]\otimes_\K [1]$ is in the image of $\pi_*$.
  We have thus proved that
  \[ \Im \pi_* = \, \bigoplus_\varphi \left( \bigoplus_{i\in \I_\varphi} H_m(V_*^{\varphi,i}) \otimes_\K \frac{R}{(\varphi^{m+k_{\varphi,i}})} \right). \]
  Then the cokernel of $\pi_*$ can be easily computed:
  \[ \coker \pi_* = \frac{H_m(L_*)}{\Im \pi_*} = \bigoplus_\varphi \left( \K^{\rk \delta_m^\varphi} \otimes_\K \frac{R}{(\varphi)} \right) = \bigoplus_\varphi \, \left( \frac{R}{(\varphi)} \right)^{\oplus \rk\delta_m^\varphi}. \]
  
  The $R$-module $H_m(C^0_*) \cong H_m(C^0_*(K; \K)) \otimes_\K R$ is free and finitely generated, because $H_m(C^0_*(K; \K))$ is a finite-dimensional vector space over $\K$ (recall that $K$ is a finite simplicial complex).
  The kernel of $\pi_*$ is a submodule of $H_m(C^0_*)$, so it is itself a free $R$-module with lower or equal rank.
  Let $[c_1], \dots, [c_k]$ be an $R$-base of $H_m(C^0_*)$.
  Consider the non-zero ideal
  \[ I = \bigcap_\varphi \bigcap_{i\in\I_\varphi} (\varphi^{m+k_i}) \subseteq R, \]
  where $\varphi$ varies among the (finitely many) irreducible elements which divide some weight $\w(\sigma)$ (for $\sigma \in K$).
  Fix any non-zero element $\alpha \in I$.
  Then the elements $\alpha[c_1], \dots, \alpha[c_k]$ generate a free submodule of $\ker \pi_*$ of rank $k=\rk H_m(C^0_*)$.
  Therefore $\ker \pi_*$ and $H_m(C^0_*)$ have the same rank, so they are isomorphic as $R$-modules.
\end{proof}

\section{Precise matchings for Artin groups}
\label{sec:precise-matchings-artin-groups}

Consider now the case of Artin groups, as in Section \ref{sec:weighted-sheaves-artin-groups}.
For a Coxeter system $(\W,S)$ we have constructed a weighted sheaf $\S = (K,R,\w)$
with
\[ K = \{ \sigma\subset S \mid \mbox{the parabolic subgroup $\W_\sigma$ is finite}\} \]
and $R = \Q[q^{\pm 1}]$.
The associated weighted complex $L_*$ fits into the short exact sequence
\[ 0 \to C_* \xrightarrow{\Delta} C^0_* \xrightarrow{\pi} L_* \to 0, \]
which gives rise to the long exact sequence \eqref{eq:long-rank1}:
\begin{equation*}
  \dots \xrightarrow{\pi_*} H_{k+1}(L_*) \rightarrow H_k(C_*) \xrightarrow{\Delta_*} H_k(C_*^0) \xrightarrow{\pi_*} H_k(L_*) \rightarrow H_{k-1}(C_*) \xrightarrow{\Delta_*} \dots
\end{equation*}
In order to compute $H_*(C_*) = H_*(\XW; R)$, we split this long exact sequence into the short exact sequences
\[ 0 \to \coker \pi_* \to H_m(C_*) \xrightarrow{\Delta_*} \ker \pi_* \to 0, \]
where on the left we have the cokernel of $\pi_* \colon H_{m+1}(C^0_*) \to H_{m+1}(L_*)$ and on the right we have the kernel of $\pi_*\colon H_m(C^0_*) \to H_m(L_*)$.
Since $\ker \pi_*$ is a free $R$-module, these short exact sequences split:
\[ H_m(C_*) \cong \coker \pi_* \oplus \ker \pi_*. \]

Recall that the only irreducible elements of $R$ that occur in the factorization of the weights are the cyclotomic polynomials $\varphi_d$ for $d\geq 2$.
As in Section \ref{sec:precise-matchings}, suppose from now on that we have constructed a $\varphi$-precise matching $\M_\varphi$ for each cyclotomic polynomial $\varphi = \varphi_d$ (with $d\geq 2$).
Then we have an explicit description of $\coker\pi_*$ and $\ker\pi_*$ thanks to Proposition \ref{prop:kernel-and-cokernel}, and we obtain the following result.

\begin{provedtheorem}
  Under the above hypothesis, the homology of $\XW$ with coefficients in the representation $R=\Q[q^{\pm 1}]$ is given by
  \[ H_m(\XW; R) \cong \left( \bigoplus_\varphi \left( \frac{R}{(\varphi)} \right)^{\oplus \rk \delta_{m+1}^\varphi} \right) \oplus H_m(C^0_*). \qedhere \]
  \label{thm:homology-artin-groups}
\end{provedtheorem}

In particular the term $H_m(C^0_*)$ gives the free part of the homology, and the other direct summands give the torsion part.
The torsion part actually takes a very particular form, and we are going to highlight this in the following result.

\begin{provedtheorem}
  Let $\GW$ be an Artin group that admits a $\varphi$-precise matching for all cyclotomic polynomials $\varphi=\varphi_d$ (with $d\geq 2$). Then the homology $H_*(\XW;R)$ does not have $\varphi^k$-torsion for $k\geq 2$.
  \label{cor:squarefree-torsion}
\end{provedtheorem}

We are particularly interested in Artin groups of finite and affine type.
When $\GW$ is an Artin group of finite type with $n$ generators, $K$ is the full simplicial complex on $S\cong\{1,\dots,n\}$ and therefore $C^0_*$ has trivial homology in every dimension.
Thus the formula of Theorem \ref{thm:homology-artin-groups} reduces to
\[ H_m(\XW; R) \cong \bigoplus_\varphi \left( \frac{R}{(\varphi)} \right)^{\oplus \rk \delta_{m+1}^\varphi}. \]
When $\GW$ is an Artin group of affine type with $n+1$ generators, $K$ is obtained from the full simplicial complex on $S \cong \{0,1, \dots, n\}$ by removing the single top-dimensional simplex.
Then we have
\[ H_m(\XW; R) \cong 
  \begin{cases}
    R & \text{for } m = n, \\
    \bigoplus_\varphi \left( \frac{R}{(\varphi)} \right)^{\oplus \rk \delta_{m+1}^\varphi} & \text{for } m < n.
  \end{cases}
\]

\begin{remark}
  When $\GW$ is an Artin group of finite type, the corresponding reflection arrangement of hyperplanes $\A$ is finite.
  In this case it is well known that there is an $R$-module isomorphism between the twisted homology $H_*(\XW; R)$ and the homology with constant coefficients $H_*(F; \Q)$ of the Milnor fiber $F$ of $\A$ \cite{callegaro2005cohomology}. The $q$-multiplication on the homology of $\XW$ corresponds to the action of the monodromy operator on the homology of $F$.
  If $N=|\A|$, the square of the defining polynomial of the arrangement is $\W$-invariant, thus the order of the monodromy of the Milnor fibration divides $2N$. It follows that the polynomial $q^{2N}-1$ must annihilate the homology.
  Since $q^{2N}-1$ is square-free in characteristic $0,$ the homology cannot have $\varphi^k$-torsion for $k\geq 2$.
  So the conclusion of Theorem \ref{cor:squarefree-torsion} is not surprising in the case of Artin groups of finite type.
  \label{rmk:squarefree-torsion-finite-type}
\end{remark}

In the rest of this paper we are going to construct precise matchings for Artin groups of type $A_n$, $B_n$, $\smash{\tilde A_n}$ and $\smash{\tilde C_n}$. For each of these cases we are then going to: describe the critical simplices with respect to the constructed matching; find all alternating paths and incidence numbers between critical simplices; determine the ranks $\rk \delta^\varphi_*$ and use Theorem \ref{thm:homology-artin-groups} to compute the homology $H_*(\XW;R)$.
The final results are stated in Theorem \ref{teo:simpleind} (case $A_n$), Theorem \ref{thm:homology-Bn} (case $B_n$), Theorem \ref{thm:homology-tAn} (case $\smash{\tilde A_n}$), and Theorem \ref{thm:homology-tCn} (case $\smash{\tilde C_n}$).

\medskip

Let us first introduce some notation.
We will have $S=\{1,2,\dots,n\}$ for the finite cases ($A_n$ and $B_n$), and $S=\{0,1,2,\dots,n\}$ for the affine cases ($\smash{\tilde A_n}$ and $\smash{\tilde C_n}$).
A simplex $\sigma \subseteq S$ will be also represented as a string of bits $\epsilon_i\in\{0,1\}$ (for $i\in S$), where $\epsilon_i = 1$ if $i\in S$ and $\epsilon_i=0$ if $i\not\in S$.
For example, if $S=\{1,2,3,4\}$, the string representation of $\sigma = \{1,2,4\}$ is $1101$.
Also, for $\sigma\subseteq S$ and $v\in S$, let
\[ \sigma \xor v =
  \begin{cases}
    \sigma \cup \{v\} & \text{if } v \not \in \sigma, \\
    \sigma \setminus \{v\} & \text{if } v \in \sigma.
  \end{cases}
\]
(it can be regarded as the bitwise xor between the string representation of $\sigma$ and the string with $\epsilon_v=1$ and $\epsilon_i=0$ for $i\neq v$).

\subsection{Case $A_n$}
\label{sec:precise-matching-An}
Many properties of the homology $H_*(\XW;R)$ in the case $A_n$ have been thoroughly discussed in Section \ref{sec:braid-groups}.
Using precise matchings we are going to obtain a new proof of the formula for the homology of braid groups (Theorem \ref{teo:simpleind}).

Let $S=\{1,2,\dots,n\}$ be the set of vertices of the Coxeter graph of type $A_n$, as in Section \ref{sec:braid-groups}, Figure \ref{fig:An}.
In this case $K$ is the full simplicial complex on $S$.
Fix now an integer $d\geq 2$ and set $\varphi=\varphi_d$. The $\varphi$-weight of a simplex $\sigma\in K$ has been computed in Section \ref{sec:braid-groups} and is as follows:
\[ v_\varphi(\sigma) = \sum_{i=1}^m \omega_\varphi(A_{n_i}), \]
where $n_i$ is the size of the $i$-th (linear) connected component of the subgraph induced by $\sigma$, and $\omega_\varphi(A_k)$ stands for the $\varphi$-weight of a connected component of type $A_k$, given by
\[ \omega_\varphi(A_k) = \left\lfloor \frac{k + 1}{d} \right\rfloor. \]
Fix also an integer $f$ with $0\leq f \leq d-1$. Let $K^A_{n,f}\subseteq K$ be as follows:
\[ K^A_{n,f} = \{ \sigma \in K \mid 1,2,\dots,f \in \sigma \}. \]
Notice that $K^A_{n,0}=K$, and $K^A_{n,f}$ is not a subcomplex of $K$ for $f\geq 1$ (but it is still a subposet of $K$, so it makes sense to define a matching on it).
We are going to construct a $\varphi$-precise matching on $K^A_{n,f}$. In particular, for $f=0$, we will get a $\varphi$-precise matching for $K$. The precise matchings on $\smash{K^A_{n,f}}$ for $f\geq 1$ will become useful when treating the cases $B_n$, $\smash{\tilde A_n}$ and $\smash{\tilde C_n}$.

For a fixed $d$, the matching will be constructed recursively in $n$ and $f$ for $n\geq 0$ and $0\leq f \leq d-1$.
We will write $K_{n,f}$ for $K^A_{n,f}$ throughout Section \ref{sec:precise-matching-An}.
The matching is as follows.
\begin{enumerate}[(a)]
 \item If $\{1,\dots,d-1\} \subseteq \sigma$ then match $\sigma$ with $\sigma\xor d$ (unless $n=d-1$, in which case $\sigma$ is critical).
 Here $\sigma$ is matched with a simplex which also occurs in case (a).
 Notice that for $f=d-1$ case (a) always applies, thus in the subsequent cases we can assume $f\leq d-2$.
 \item Otherwise, if $n=f$ then $\sigma$ is critical.
 \item Otherwise, if $f+1\in \sigma$ then match $\sigma$ with $\sigma\setminus\{f+1\}$.
 Notice that $\sigma\setminus\{f+1\}$ occurs in case (d).
 \item Otherwise, if $\{f+2,\dots,d-1\} \nsubseteq \sigma$ then match $\sigma$ with $\sigma\cup\{f+1\}$.
 Notice that $\sigma\cup\{f+1\}$ occurs in case (c).
 \item We are left with the simplices $\sigma$ such that $\{1,\dots,f,f+2,\dots,d-1\}\subseteq \sigma$ and $f+1\not\in\sigma$. If we ignore the vertices $1,\dots,f+1$ we are left with the simplices on the vertex set $\{f+2,\dots,n\}$ which contain $f+2,\dots,d-1$; relabeling the vertices, these are the same as the simplices on the vertex set $\{1,\dots,n-f-1\}$ which contain $1,\dots,d-2-f$.
 Then construct the matching recursively as in $K_{n-f-1,\; d-2-f}$.
\end{enumerate}

\begin{example}
  For $n=5$, $d=3$ and $f=1$, $K_{n,f}$ contains $2^4=16$ simplices of which $14$ are matched and $2$ are critical.
  For instance, consider the simplex $\sigma = \{1, 4, 5\}$. Case (e) applies because $1\in \sigma$ and $2\not\in\sigma$; the recursion requires us to consider the new simplex $\sigma' = \{2,3\} \in K_{3,0}$. Again case (e) applies because $1 \not\in\sigma'$ and $2\in\sigma'$; the recursion requires us to consider the new simplex $\sigma'' = \{1,2\} \in K_{2,1}$. Finally case (a) applies because $\{1,2\}\subseteq \sigma''$, and $\sigma''$ is critical because $\sigma'' = \{1,2\}$. Therefore $\sigma = \{1, 4, 5\} \in K_{5,1}$ is also critical.
  See Table \ref{table:An-matching-example} for an explicit description of the matching for $K_{5,1}$, $d=3$.
  \begin{table}[htbp]

{\renewcommand{\arraystretch}{1.4}

\begin{tabular}{rcl@{\hspace{13pt}}|c|c}

\multicolumn{3}{c|}{Simplices} & $v_\varphi(\sigma)$ & Step \\

\hline

\begin{tikzpicture}
\begin{scope}[every node/.style={circle,thick,draw,inner sep=2.5}, every label/.style={rectangle,draw=none}]
  \node (1) at (0.0,0) [label={above,minimum height=13}:$ $,fill=black!50] {};
  \node (2) at (0.7,0) [label={above,minimum height=13}:$ $,fill=black!50] {};
  \node (3) at (1.3,0) [label={above,minimum height=13}:$ $,fill=black!50] {};
  \node (4) at (2.0,0) [label={above,minimum height=13}:$ $,fill=black!50] {};
  \node (5) at (2.6,0) [label={above,minimum height=13}:$ $,fill=black!50] {};
\end{scope}
\begin{scope}[every edge/.style={draw=black!60,line width=1.2}]
\end{scope}
\begin{scope}[every edge/.style={draw=black, line width=3}]
  \path (1) edge node {} (2);
  \path (2) edge node {} (3);
  \path (3) edge node {} (4);
  \path (4) edge node {} (5);
\end{scope}
\begin{scope}[every node/.style={draw,inner sep=11.5,yshift=-4}, every label/.style={rectangle,draw=none,inner sep=6.0}, every fit/.append style=text badly centered]
\end{scope}
\end{tikzpicture}
 & $\longrightarrow$ & 
\begin{tikzpicture}
\begin{scope}[every node/.style={circle,thick,draw,inner sep=2.5}, every label/.style={rectangle,draw=none}]
  \node (1) at (0.0,0) [label={above,minimum height=13}:$ $,fill=black!50] {};
  \node (2) at (0.7,0) [label={above,minimum height=13}:$ $,fill=black!50] {};
  \node (3) at (1.3,0) [label={above,minimum height=13}:$ $] {};
  \node (4) at (2.0,0) [label={above,minimum height=13}:$ $,fill=black!50] {};
  \node (5) at (2.6,0) [label={above,minimum height=13}:$ $,fill=black!50] {};
\end{scope}
\begin{scope}[every edge/.style={draw=black!60,line width=1.2}]
  \path (2) edge node {} (3);
  \path (3) edge node {} (4);
\end{scope}
\begin{scope}[every edge/.style={draw=black, line width=3}]
  \path (1) edge node {} (2);
  \path (4) edge node {} (5);
\end{scope}
\begin{scope}[every node/.style={draw,inner sep=11.5,yshift=-4}, every label/.style={rectangle,draw=none,inner sep=6.0}, every fit/.append style=text badly centered]
\end{scope}
\end{tikzpicture}
 & 2 & (a) \\

\begin{tikzpicture}
\begin{scope}[every node/.style={circle,thick,draw,inner sep=2.5}, every label/.style={rectangle,draw=none}]
  \node (1) at (0.0,0) [label={above,minimum height=13}:$ $,fill=black!50] {};
  \node (2) at (0.7,0) [label={above,minimum height=13}:$ $,fill=black!50] {};
  \node (3) at (1.3,0) [label={above,minimum height=13}:$ $,fill=black!50] {};
  \node (4) at (2.0,0) [label={above,minimum height=13}:$ $,fill=black!50] {};
  \node (5) at (2.6,0) [label={above,minimum height=13}:$ $] {};
\end{scope}
\begin{scope}[every edge/.style={draw=black!60,line width=1.2}]
  \path (4) edge node {} (5);
\end{scope}
\begin{scope}[every edge/.style={draw=black, line width=3}]
  \path (1) edge node {} (2);
  \path (2) edge node {} (3);
  \path (3) edge node {} (4);
\end{scope}
\begin{scope}[every node/.style={draw,inner sep=11.5,yshift=-4}, every label/.style={rectangle,draw=none,inner sep=6.0}, every fit/.append style=text badly centered]
\end{scope}
\end{tikzpicture}
 & $\longrightarrow$ & 
\begin{tikzpicture}
\begin{scope}[every node/.style={circle,thick,draw,inner sep=2.5}, every label/.style={rectangle,draw=none}]
  \node (1) at (0.0,0) [label={above,minimum height=13}:$ $,fill=black!50] {};
  \node (2) at (0.7,0) [label={above,minimum height=13}:$ $,fill=black!50] {};
  \node (3) at (1.3,0) [label={above,minimum height=13}:$ $] {};
  \node (4) at (2.0,0) [label={above,minimum height=13}:$ $,fill=black!50] {};
  \node (5) at (2.6,0) [label={above,minimum height=13}:$ $] {};
\end{scope}
\begin{scope}[every edge/.style={draw=black!60,line width=1.2}]
  \path (2) edge node {} (3);
  \path (3) edge node {} (4);
  \path (4) edge node {} (5);
\end{scope}
\begin{scope}[every edge/.style={draw=black, line width=3}]
  \path (1) edge node {} (2);
\end{scope}
\begin{scope}[every node/.style={draw,inner sep=11.5,yshift=-4}, every label/.style={rectangle,draw=none,inner sep=6.0}, every fit/.append style=text badly centered]
\end{scope}
\end{tikzpicture}
 & 1 & (a) \\

\begin{tikzpicture}
\begin{scope}[every node/.style={circle,thick,draw,inner sep=2.5}, every label/.style={rectangle,draw=none}]
  \node (1) at (0.0,0) [label={above,minimum height=13}:$ $,fill=black!50] {};
  \node (2) at (0.7,0) [label={above,minimum height=13}:$ $,fill=black!50] {};
  \node (3) at (1.3,0) [label={above,minimum height=13}:$ $,fill=black!50] {};
  \node (4) at (2.0,0) [label={above,minimum height=13}:$ $] {};
  \node (5) at (2.6,0) [label={above,minimum height=13}:$ $,fill=black!50] {};
\end{scope}
\begin{scope}[every edge/.style={draw=black!60,line width=1.2}]
  \path (3) edge node {} (4);
  \path (4) edge node {} (5);
\end{scope}
\begin{scope}[every edge/.style={draw=black, line width=3}]
  \path (1) edge node {} (2);
  \path (2) edge node {} (3);
\end{scope}
\begin{scope}[every node/.style={draw,inner sep=11.5,yshift=-4}, every label/.style={rectangle,draw=none,inner sep=6.0}, every fit/.append style=text badly centered]
\end{scope}
\end{tikzpicture}
 & $\longrightarrow$ & 
\begin{tikzpicture}
\begin{scope}[every node/.style={circle,thick,draw,inner sep=2.5}, every label/.style={rectangle,draw=none}]
  \node (1) at (0.0,0) [label={above,minimum height=13}:$ $,fill=black!50] {};
  \node (2) at (0.7,0) [label={above,minimum height=13}:$ $,fill=black!50] {};
  \node (3) at (1.3,0) [label={above,minimum height=13}:$ $] {};
  \node (4) at (2.0,0) [label={above,minimum height=13}:$ $] {};
  \node (5) at (2.6,0) [label={above,minimum height=13}:$ $,fill=black!50] {};
\end{scope}
\begin{scope}[every edge/.style={draw=black!60,line width=1.2}]
  \path (2) edge node {} (3);
  \path (3) edge node {} (4);
  \path (4) edge node {} (5);
\end{scope}
\begin{scope}[every edge/.style={draw=black, line width=3}]
  \path (1) edge node {} (2);
\end{scope}
\begin{scope}[every node/.style={draw,inner sep=11.5,yshift=-4}, every label/.style={rectangle,draw=none,inner sep=6.0}, every fit/.append style=text badly centered]
\end{scope}
\end{tikzpicture}
 & 1 & (a) \\

\begin{tikzpicture}
\begin{scope}[every node/.style={circle,thick,draw,inner sep=2.5}, every label/.style={rectangle,draw=none}]
  \node (1) at (0.0,0) [label={above,minimum height=13}:$ $,fill=black!50] {};
  \node (2) at (0.7,0) [label={above,minimum height=13}:$ $,fill=black!50] {};
  \node (3) at (1.3,0) [label={above,minimum height=13}:$ $,fill=black!50] {};
  \node (4) at (2.0,0) [label={above,minimum height=13}:$ $] {};
  \node (5) at (2.6,0) [label={above,minimum height=13}:$ $] {};
\end{scope}
\begin{scope}[every edge/.style={draw=black!60,line width=1.2}]
  \path (3) edge node {} (4);
  \path (4) edge node {} (5);
\end{scope}
\begin{scope}[every edge/.style={draw=black, line width=3}]
  \path (1) edge node {} (2);
  \path (2) edge node {} (3);
\end{scope}
\begin{scope}[every node/.style={draw,inner sep=11.5,yshift=-4}, every label/.style={rectangle,draw=none,inner sep=6.0}, every fit/.append style=text badly centered]
\end{scope}
\end{tikzpicture}
 & $\longrightarrow$ & 
\begin{tikzpicture}
\begin{scope}[every node/.style={circle,thick,draw,inner sep=2.5}, every label/.style={rectangle,draw=none}]
  \node (1) at (0.0,0) [label={above,minimum height=13}:$ $,fill=black!50] {};
  \node (2) at (0.7,0) [label={above,minimum height=13}:$ $,fill=black!50] {};
  \node (3) at (1.3,0) [label={above,minimum height=13}:$ $] {};
  \node (4) at (2.0,0) [label={above,minimum height=13}:$ $] {};
  \node (5) at (2.6,0) [label={above,minimum height=13}:$ $] {};
\end{scope}
\begin{scope}[every edge/.style={draw=black!60,line width=1.2}]
  \path (2) edge node {} (3);
  \path (3) edge node {} (4);
  \path (4) edge node {} (5);
\end{scope}
\begin{scope}[every edge/.style={draw=black, line width=3}]
  \path (1) edge node {} (2);
\end{scope}
\begin{scope}[every node/.style={draw,inner sep=11.5,yshift=-4}, every label/.style={rectangle,draw=none,inner sep=6.0}, every fit/.append style=text badly centered]
\end{scope}
\end{tikzpicture}
 & 1 & (a) \\

\begin{tikzpicture}
\begin{scope}[every node/.style={circle,thick,draw,inner sep=2.5}, every label/.style={rectangle,draw=none}]
  \node (1) at (0.0,0) [label={above,minimum height=13}:$ $,fill=black!50] {};
  \node (2) at (0.7,0) [label={above,minimum height=13}:$ $] {};
  \node (3) at (1.3,0) [label={above,minimum height=13}:$ $,fill=black!50] {};
  \node (4) at (2.0,0) [label={above,minimum height=13}:$ $,fill=black!50] {};
  \node (5) at (2.6,0) [label={above,minimum height=13}:$ $,fill=black!50] {};
\end{scope}
\begin{scope}[every edge/.style={draw=black!60,line width=1.2}]
  \path (1) edge node {} (2);
  \path (2) edge node {} (3);
\end{scope}
\begin{scope}[every edge/.style={draw=black, line width=3}]
  \path (3) edge node {} (4);
  \path (4) edge node {} (5);
\end{scope}
\begin{scope}[every node/.style={draw,inner sep=11.5,yshift=-4}, every label/.style={rectangle,draw=none,inner sep=6.0}, every fit/.append style=text badly centered]
\end{scope}
\end{tikzpicture}
 & $\longrightarrow$ & 
\begin{tikzpicture}
\begin{scope}[every node/.style={circle,thick,draw,inner sep=2.5}, every label/.style={rectangle,draw=none}]
  \node (1) at (0.0,0) [label={above,minimum height=13}:$ $,fill=black!50] {};
  \node (2) at (0.7,0) [label={above,minimum height=13}:$ $] {};
  \node (3) at (1.3,0) [label={above,minimum height=13}:$ $,fill=black!50] {};
  \node (4) at (2.0,0) [label={above,minimum height=13}:$ $,fill=black!50] {};
  \node (5) at (2.6,0) [label={above,minimum height=13}:$ $] {};
\end{scope}
\begin{scope}[every edge/.style={draw=black!60,line width=1.2}]
  \path (1) edge node {} (2);
  \path (2) edge node {} (3);
  \path (4) edge node {} (5);
\end{scope}
\begin{scope}[every edge/.style={draw=black, line width=3}]
  \path (3) edge node {} (4);
\end{scope}
\begin{scope}[every node/.style={draw,inner sep=11.5,yshift=-4}, every label/.style={rectangle,draw=none,inner sep=6.0}, every fit/.append style=text badly centered]
\end{scope}
\end{tikzpicture}
 & 1 & (e) $\rightsquigarrow$ (a) \\

\begin{tikzpicture}
\begin{scope}[every node/.style={circle,thick,draw,inner sep=2.5}, every label/.style={rectangle,draw=none}]
  \node (1) at (0.0,0) [label={above,minimum height=13}:$ $,fill=black!50] {};
  \node (2) at (0.7,0) [label={above,minimum height=13}:$ $] {};
  \node (3) at (1.3,0) [label={above,minimum height=13}:$ $,fill=black!50] {};
  \node (4) at (2.0,0) [label={above,minimum height=13}:$ $] {};
  \node (5) at (2.6,0) [label={above,minimum height=13}:$ $,fill=black!50] {};
\end{scope}
\begin{scope}[every edge/.style={draw=black!60,line width=1.2}]
  \path (1) edge node {} (2);
  \path (2) edge node {} (3);
  \path (3) edge node {} (4);
  \path (4) edge node {} (5);
\end{scope}
\begin{scope}[every edge/.style={draw=black, line width=3}]
\end{scope}
\begin{scope}[every node/.style={draw,inner sep=11.5,yshift=-4}, every label/.style={rectangle,draw=none,inner sep=6.0}, every fit/.append style=text badly centered]
\end{scope}
\end{tikzpicture}
 & $\longrightarrow$ & 
\begin{tikzpicture}
\begin{scope}[every node/.style={circle,thick,draw,inner sep=2.5}, every label/.style={rectangle,draw=none}]
  \node (1) at (0.0,0) [label={above,minimum height=13}:$ $,fill=black!50] {};
  \node (2) at (0.7,0) [label={above,minimum height=13}:$ $] {};
  \node (3) at (1.3,0) [label={above,minimum height=13}:$ $] {};
  \node (4) at (2.0,0) [label={above,minimum height=13}:$ $] {};
  \node (5) at (2.6,0) [label={above,minimum height=13}:$ $,fill=black!50] {};
\end{scope}
\begin{scope}[every edge/.style={draw=black!60,line width=1.2}]
  \path (1) edge node {} (2);
  \path (2) edge node {} (3);
  \path (3) edge node {} (4);
  \path (4) edge node {} (5);
\end{scope}
\begin{scope}[every edge/.style={draw=black, line width=3}]
\end{scope}
\begin{scope}[every node/.style={draw,inner sep=11.5,yshift=-4}, every label/.style={rectangle,draw=none,inner sep=6.0}, every fit/.append style=text badly centered]
\end{scope}
\end{tikzpicture}
 & 0 & (e) $\rightsquigarrow$ (c)/(d) \\

\begin{tikzpicture}
\begin{scope}[every node/.style={circle,thick,draw,inner sep=2.5}, every label/.style={rectangle,draw=none}]
  \node (1) at (0.0,0) [label={above,minimum height=13}:$ $,fill=black!50] {};
  \node (2) at (0.7,0) [label={above,minimum height=13}:$ $] {};
  \node (3) at (1.3,0) [label={above,minimum height=13}:$ $,fill=black!50] {};
  \node (4) at (2.0,0) [label={above,minimum height=13}:$ $] {};
  \node (5) at (2.6,0) [label={above,minimum height=13}:$ $] {};
\end{scope}
\begin{scope}[every edge/.style={draw=black!60,line width=1.2}]
  \path (1) edge node {} (2);
  \path (2) edge node {} (3);
  \path (3) edge node {} (4);
  \path (4) edge node {} (5);
\end{scope}
\begin{scope}[every edge/.style={draw=black, line width=3}]
\end{scope}
\begin{scope}[every node/.style={draw,inner sep=11.5,yshift=-4}, every label/.style={rectangle,draw=none,inner sep=6.0}, every fit/.append style=text badly centered]
\end{scope}
\end{tikzpicture}
 & $\longrightarrow$ & 
\begin{tikzpicture}
\begin{scope}[every node/.style={circle,thick,draw,inner sep=2.5}, every label/.style={rectangle,draw=none}]
  \node (1) at (0.0,0) [label={above,minimum height=13}:$ $,fill=black!50] {};
  \node (2) at (0.7,0) [label={above,minimum height=13}:$ $] {};
  \node (3) at (1.3,0) [label={above,minimum height=13}:$ $] {};
  \node (4) at (2.0,0) [label={above,minimum height=13}:$ $] {};
  \node (5) at (2.6,0) [label={above,minimum height=13}:$ $] {};
\end{scope}
\begin{scope}[every edge/.style={draw=black!60,line width=1.2}]
  \path (1) edge node {} (2);
  \path (2) edge node {} (3);
  \path (3) edge node {} (4);
  \path (4) edge node {} (5);
\end{scope}
\begin{scope}[every edge/.style={draw=black, line width=3}]
\end{scope}
\begin{scope}[every node/.style={draw,inner sep=11.5,yshift=-4}, every label/.style={rectangle,draw=none,inner sep=6.0}, every fit/.append style=text badly centered]
\end{scope}
\end{tikzpicture}
 & 0 & (e) $\rightsquigarrow$ (c)/(d) \\

\begin{tikzpicture}
\begin{scope}[every node/.style={circle,thick,draw,inner sep=2.5}, every label/.style={rectangle,draw=none}]
  \node (1) at (0.0,0) [label={above,minimum height=13}:$ $,fill=black!50] {};
  \node (2) at (0.7,0) [label={above,minimum height=13}:$ $] {};
  \node (3) at (1.3,0) [label={above,minimum height=13}:$ $] {};
  \node (4) at (2.0,0) [label={above,minimum height=13}:$ $,fill=black!50] {};
  \node (5) at (2.6,0) [label={above,minimum height=13}:$ $,fill=black!50] {};
\end{scope}
\begin{scope}[every edge/.style={draw=black!60,line width=1.2}]
  \path (1) edge node {} (2);
  \path (2) edge node {} (3);
  \path (3) edge node {} (4);
\end{scope}
\begin{scope}[every edge/.style={draw=black, line width=3}]
  \path (4) edge node {} (5);
\end{scope}
\begin{scope}[every node/.style={draw,inner sep=11.5,yshift=-4}, every label/.style={rectangle,draw=none,inner sep=6.0}, every fit/.append style=text badly centered]
\end{scope}
\end{tikzpicture}
 & & \emph{(critical)} & 1 & (e) $\rightsquigarrow$ (e) $\rightsquigarrow$ (a) \\

\begin{tikzpicture}
\begin{scope}[every node/.style={circle,thick,draw,inner sep=2.5}, every label/.style={rectangle,draw=none}]
  \node (1) at (0.0,0) [label={above,minimum height=13}:$ $,fill=black!50] {};
  \node (2) at (0.7,0) [label={above,minimum height=13}:$ $] {};
  \node (3) at (1.3,0) [label={above,minimum height=13}:$ $] {};
  \node (4) at (2.0,0) [label={above,minimum height=13}:$ $,fill=black!50] {};
  \node (5) at (2.6,0) [label={above,minimum height=13}:$ $] {};
\end{scope}
\begin{scope}[every edge/.style={draw=black!60,line width=1.2}]
  \path (1) edge node {} (2);
  \path (2) edge node {} (3);
  \path (3) edge node {} (4);
  \path (4) edge node {} (5);
\end{scope}
\begin{scope}[every edge/.style={draw=black, line width=3}]
\end{scope}
\begin{scope}[every node/.style={draw,inner sep=11.5,yshift=-4}, every label/.style={rectangle,draw=none,inner sep=6.0}, every fit/.append style=text badly centered]
\end{scope}
\end{tikzpicture}
 & & \emph{(critical)} & 0 & (e) $\rightsquigarrow$ (e) $\rightsquigarrow$ (e) $\rightsquigarrow$ (b)

\end{tabular}

}

\vskip0.3cm

     \caption{Matching in the case $A_n$ with $n=5$, $d=3$, $f=1$.
    The last columns indicates the case where simplices occur.
    When case (e) is reached, the arrow ``$\rightsquigarrow$'' indicates how the recursion continues (after $K_{5,1}$, the recursion involves $K_{3,0}$, $K_{2,1}$ and $K_{0,0}$).}
    \label{table:An-matching-example}
  \end{table}
\end{example}

\begin{remark}
  A peculiarity of this matching is that, if $\sigma' \to \tau'$ is in the matching, then $\sigma' = \tau' \cup \{v\}$ with $v\equiv f+1$ or $v\equiv 0 \pmod d$. This can be easily checked by induction.
  \label{rmk:removed-vertex}
\end{remark}

\begin{lemma}
  The matching described above is an acyclic weighted matching on $K_{n,f}$.
  \label{lemma:An-acyclic-weighted}
\end{lemma}

\begin{proof}
  \textbf{Part 1: the matching is acyclic.}
  The proof is by induction on $n$, the case $n=0$ being trivial.
  For $n\leq f$, in $K_{n,f}$ there are either $1$ or $0$ simplices.
  Assume from now on $n>f$.
  Let $P = \{p_\text{a}, p_\text{e}, p_\text{c,d}\}$ be a three-element totally ordered poset, with the order given by $p_\text{a} > p_\text{e} > p_\text{c,d}$.
  Consider the map $\eta\colon K_{n,f} \to P$ which sends $\sigma\in K_{n,f}$ to the $p_\text{x}$ such that $\sigma$ occurs in case (x) (here cases (c) and (d) are united).
  For instance, $\eta(\sigma) = p_\text{a}$ if and only if $\{1,\dots, d-1\} \subseteq \sigma$.
  Notice that the map $\eta$ is compatible with the matching, i.e.\ two matched simplices lie in the same fiber of $\eta$.
  
  Let us prove that $\eta$ is a poset map. Given two simplices $\sigma \geq \tau$, we want to prove that $\eta(\sigma) \geq \eta(\tau)$.
  If $\eta(\tau) = p_\text{a}$ then $\{1,\dots,d-1\} \subseteq \tau \subseteq \sigma$, so $\eta(\sigma) = p_\text{a}$ also.
If $\eta(\tau) = p_\text{e}$ then $\{1,\dots,f,f+2,\dots,d-1\}\subseteq \tau\subseteq \sigma$, thus $\eta(\sigma)\in \{ p_\text{a}, p_\text{e} \}$.
  Finally, if $\eta(\tau) = p_\text{c,d}$ there is nothing to prove.
  
  Since $\eta$ is a poset map, our matching is acyclic if and only if it is acyclic on each fiber of $\eta$.
  The restriction of the matching to $\eta^{-1}(p_\text{a})$ consists of edges $\sigma \to \tau$ with $\sigma = \tau \cup \{d\}$, so there is no alternating path of length $\geq 4$ and the matching is acyclic.
  The restriction to $\eta^{-1}(p_\text{c,d})$ is acyclic for the same reason (it is always the vertex $f+1$ which is added or removed).
  The restriction to $\eta^{-1}(p_\text{e})$ is acyclic by induction.

  \textbf{Part 2: the matching is weighted.}
  The proof is by induction on $n$, the case $n=0$ being trivial.
  In case (a), $v_\varphi(\sigma) = v_\varphi(\sigma \xor d)$ by what we have already said in Section \ref{sec:braid-groups} (see the proof of Theorem \ref{thm:casoAn}).
  In cases (c) and (d), both the simplices $\sigma$ and $\sigma \xor (f+1)$ do not contain $\{1,\dots,d-1\}$. Thus the (possibly empty) connected component of $1$ has size $\leq d-2$ and does not contribute to the weight.
  Therefore $v_\varphi(\sigma) = v_\varphi(\sigma\xor(f+1))$.
  Finally, suppose that $\sigma$ and $\tau$ are simplices that occur in case (e). Let $\hat\sigma,\hat\tau\in K_{n-f-1,\;d-2-f}$ be the simplices obtained ignoring the vertices $1$, \dots, $f+1$, as described above.
  Since $f \leq d-2$ we have that
  \[ v_\varphi(\sigma) = v_\varphi(\tau) \quad \text{if and only if} \quad v_\varphi(\hat\sigma) = v_\varphi(\hat\tau), \]
  so we are done by induction on $n$.
\end{proof}

The critical simplices of the matching on $K_{n,f}$ are quite simple to describe.
If $f \leq d-2$, $n>f$ and $n\equiv f$ or $-1\pmod d$, there are 2 critical simplices, one of weight $1$ and one of weight $0$. If $n=f$ there is 1 critical simplex (in fact there is only one simplex in $K_{n,f}$).
In all the other cases the matching has no critical simplices.
We are going to prove this in the following theorem.
See Table \ref{table:An-critical}, and Figures \ref{fig:An-critical-f} and \ref{fig:An-critical-1}, for an explicit description of the critical simplices.
Notice in particular that, when there are 2 critical simplices, one is a face of the other in $K_{n,f}$.

\begin{table}[htbp]
  {\renewcommand{\arraystretch}{1.4}\setlength\tabcolsep{4pt}\begin{tabular}{c|c|c|c|c}
    \multicolumn{2}{c|}{Case} & Simplices & $|\sigma|$ & $v_\varphi(\sigma)$ \\
    \hline \hline
    & \multirow{2}{*}{$n = kd+f$ {\footnotesize (Fig.\ \ref{fig:An-critical-f})}} & $(1^f01^{d-2-f}0)^{k-1}1^f01^{d-1}$ & $n-2k+1$ & 1 \\
    $n>f$ & & $(1^f01^{d-2-f}0)^k1^f$ & $n-2k$ & 0 \\
    \cline{2-5}
    $f \leq d-2$ & \multirow{2}{*}{$n = kd-1$ {\footnotesize (Fig.\ \ref{fig:An-critical-1})}} & $(1^f01^{d-2-f}0)^{k-1}1^{d-1}$ & $n-2k+2$ & 1 \\
    & & $(1^f01^{d-2-f}0)^{k-1}1^f01^{d-2-f}$ & $n-2k+1$ & 0 \\
    \hline
    \multicolumn{2}{c|}{$n=f$} & $1^f$ & $n$ & 1 or 0 \\
    \hline
  \end{tabular}
  }
  \vskip0.3cm
  \caption{Description of the critical simplices for $A_n$.}
  \label{table:An-critical}
\end{table}

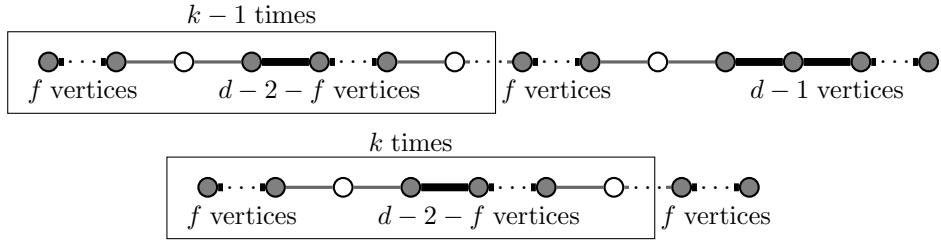
\begin{figure}[htbp]
  \begin{tikzpicture}
\begin{scope}[every node/.style={circle,thick,draw,inner sep=2.5}, every label/.style={rectangle,draw=none}]
  \node (0) at (0.0,0) [label={above,minimum height=13}:$ $,fill=black!50] {};
  \node (1) at (0.9,0) [label={above,minimum height=13}:$ $,fill=black!50] {};
  \node (2) at (1.8,0) [label={above,minimum height=13}:$ $] {};
  \node (3) at (2.7,0) [label={above,minimum height=13}:$ $,fill=black!50] {};
  \node (4) at (3.6,0) [label={above,minimum height=13}:$ $,fill=black!50] [label={below}:$d-2-f$\;vertices] {};
  \node (5) at (4.5,0) [label={above,minimum height=13}:$ $,fill=black!50] {};
  \node (6) at (5.4,0) [label={above,minimum height=13}:$ $] {};
  \node (7) at (6.3,0) [label={above,minimum height=13}:$ $,fill=black!50] {};
  \node (8) at (7.2,0) [label={above,minimum height=13}:$ $,fill=black!50] {};
  \node (9) at (8.1,0) [label={above,minimum height=13}:$ $] {};
  \node (10) at (9.0,0) [label={above,minimum height=13}:$ $,fill=black!50] {};
  \node (11) at (9.9,0) [label={above,minimum height=13}:$ $,fill=black!50] {};
  \node (12) at (10.8,0) [label={above,minimum height=13}:$ $,fill=black!50] {};
  \node (13) at (11.7,0) [label={above,minimum height=13}:$ $,fill=black!50] {};
\end{scope}
\begin{scope}[every edge/.style={draw=black!60,line width=1.2}]
  \path (1) edge node {} (2);
  \path (2) edge node {} (3);
  \path (5) edge node {} (6);
  \path (6) edge node [fill=white, rectangle, inner sep=1.0] {$\ldots$} (7);
  \path (8) edge node {} (9);
  \path (9) edge node {} (10);
\end{scope}
\begin{scope}[every edge/.style={draw=black, line width=3}]
  \path (0) edge node [fill=white, rectangle, inner sep=1.0, minimum height = 0.5cm] {$\ldots$} node [below=2] {$f$\;vertices} (1);
  \path (3) edge node {} (4);
  \path (4) edge node [fill=white, rectangle, inner sep=1.0, minimum height = 0.5cm] {$\ldots$} (5);
  \path (7) edge node [fill=white, rectangle, inner sep=1.0, minimum height = 0.5cm] {$\ldots$} node [below=2] {$f$\;vertices} (8);
  \path (10) edge node {} (11);
  \path (11) edge node {} node [below=2] {$d-1$\;vertices} (12);
  \path (12) edge node [fill=white, rectangle, inner sep=1.0, minimum height = 0.5cm] {$\ldots$} (13);
\end{scope}
\begin{scope}[every node/.style={draw,inner sep=11.5,yshift=-4}, every label/.style={rectangle,draw=none,inner sep=6.0}, every fit/.append style=text badly centered]
  \node [fit=(0) (6), label={above}:$k-1$\;times] {};
\end{scope}
\end{tikzpicture}
   \begin{tikzpicture}
\begin{scope}[every node/.style={circle,thick,draw,inner sep=2.5}, every label/.style={rectangle,draw=none}]
  \node (0) at (0.0,0) [label={above,minimum height=13}:$ $,fill=black!50] {};
  \node (1) at (0.9,0) [label={above,minimum height=13}:$ $,fill=black!50] {};
  \node (2) at (1.8,0) [label={above,minimum height=13}:$ $] {};
  \node (3) at (2.7,0) [label={above,minimum height=13}:$ $,fill=black!50] {};
  \node (4) at (3.6,0) [label={above,minimum height=13}:$ $,fill=black!50] [label={below}:$d-2-f$\;vertices] {};
  \node (5) at (4.5,0) [label={above,minimum height=13}:$ $,fill=black!50] {};
  \node (6) at (5.4,0) [label={above,minimum height=13}:$ $] {};
  \node (7) at (6.3,0) [label={above,minimum height=13}:$ $,fill=black!50] {};
  \node (8) at (7.2,0) [label={above,minimum height=13}:$ $,fill=black!50] {};
\end{scope}
\begin{scope}[every edge/.style={draw=black!60,line width=1.2}]
  \path (1) edge node {} (2);
  \path (2) edge node {} (3);
  \path (5) edge node {} (6);
  \path (6) edge node [fill=white, rectangle, inner sep=1.0] {$\ldots$} (7);
\end{scope}
\begin{scope}[every edge/.style={draw=black, line width=3}]
  \path (0) edge node [fill=white, rectangle, inner sep=1.0, minimum height = 0.5cm] {$\ldots$} node [below=2] {$f$\;vertices} (1);
  \path (3) edge node {} (4);
  \path (4) edge node [fill=white, rectangle, inner sep=1.0, minimum height = 0.5cm] {$\ldots$} (5);
  \path (7) edge node [fill=white, rectangle, inner sep=1.0, minimum height = 0.5cm] {$\ldots$} node [below=2] {$f$\;vertices} (8);
\end{scope}
\begin{scope}[every node/.style={draw,inner sep=11.5,yshift=-4}, every label/.style={rectangle,draw=none,inner sep=6.0}, every fit/.append style=text badly centered]
  \node [fit=(0) (6), label={above}:$k$\;times] {};
\end{scope}
\end{tikzpicture}
   \caption{Critical simplices for $f\leq d-2$ and $n=kd+f$ ($k\geq 1$).}
  \label{fig:An-critical-f}
\end{figure}

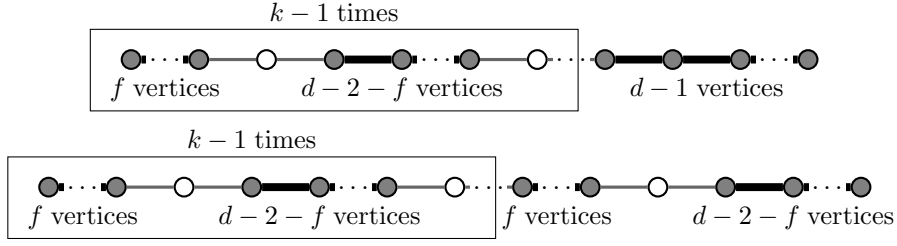
\begin{figure}[htbp]
  \begin{tikzpicture}
\begin{scope}[every node/.style={circle,thick,draw,inner sep=2.5}, every label/.style={rectangle,draw=none}]
  \node (0) at (0.0,0) [label={above,minimum height=13}:$ $,fill=black!50] {};
  \node (1) at (0.9,0) [label={above,minimum height=13}:$ $,fill=black!50] {};
  \node (2) at (1.8,0) [label={above,minimum height=13}:$ $] {};
  \node (3) at (2.7,0) [label={above,minimum height=13}:$ $,fill=black!50] {};
  \node (4) at (3.6,0) [label={above,minimum height=13}:$ $,fill=black!50] [label={below}:$d-2-f$\;vertices] {};
  \node (5) at (4.5,0) [label={above,minimum height=13}:$ $,fill=black!50] {};
  \node (6) at (5.4,0) [label={above,minimum height=13}:$ $] {};
  \node (7) at (6.3,0) [label={above,minimum height=13}:$ $,fill=black!50] {};
  \node (8) at (7.2,0) [label={above,minimum height=13}:$ $,fill=black!50] {};
  \node (9) at (8.1,0) [label={above,minimum height=13}:$ $,fill=black!50] {};
  \node (10) at (9.0,0) [label={above,minimum height=13}:$ $,fill=black!50] {};
\end{scope}
\begin{scope}[every edge/.style={draw=black!60,line width=1.2}]
  \path (1) edge node {} (2);
  \path (2) edge node {} (3);
  \path (5) edge node {} (6);
  \path (6) edge node [fill=white, rectangle, inner sep=1.0] {$\ldots$} (7);
\end{scope}
\begin{scope}[every edge/.style={draw=black, line width=3}]
  \path (0) edge node [fill=white, rectangle, inner sep=1.0, minimum height = 0.5cm] {$\ldots$} node [below=2] {$f$\;vertices} (1);
  \path (3) edge node {} (4);
  \path (4) edge node [fill=white, rectangle, inner sep=1.0, minimum height = 0.5cm] {$\ldots$} (5);
  \path (7) edge node {} (8);
  \path (8) edge node {} node [below=2] {$d-1$\;vertices} (9);
  \path (9) edge node [fill=white, rectangle, inner sep=1.0, minimum height = 0.5cm] {$\ldots$} (10);
\end{scope}
\begin{scope}[every node/.style={draw,inner sep=11.5,yshift=-4}, every label/.style={rectangle,draw=none,inner sep=6.0}, every fit/.append style=text badly centered]
  \node [fit=(0) (6), label={above}:$k-1$\;times] {};
\end{scope}
\end{tikzpicture}
   \begin{tikzpicture}
\begin{scope}[every node/.style={circle,thick,draw,inner sep=2.5}, every label/.style={rectangle,draw=none}]
  \node (0) at (0.0,0) [label={above,minimum height=13}:$ $,fill=black!50] {};
  \node (1) at (0.9,0) [label={above,minimum height=13}:$ $,fill=black!50] {};
  \node (2) at (1.8,0) [label={above,minimum height=13}:$ $] {};
  \node (3) at (2.7,0) [label={above,minimum height=13}:$ $,fill=black!50] {};
  \node (4) at (3.6,0) [label={above,minimum height=13}:$ $,fill=black!50] [label={below}:$d-2-f$\;vertices] {};
  \node (5) at (4.5,0) [label={above,minimum height=13}:$ $,fill=black!50] {};
  \node (6) at (5.4,0) [label={above,minimum height=13}:$ $] {};
  \node (7) at (6.3,0) [label={above,minimum height=13}:$ $,fill=black!50] {};
  \node (8) at (7.2,0) [label={above,minimum height=13}:$ $,fill=black!50] {};
  \node (9) at (8.1,0) [label={above,minimum height=13}:$ $] {};
  \node (10) at (9.0,0) [label={above,minimum height=13}:$ $,fill=black!50] {};
  \node (11) at (9.9,0) [label={above,minimum height=13}:$ $,fill=black!50] [label={below}:$d-2-f$\;vertices] {};
  \node (12) at (10.8,0) [label={above,minimum height=13}:$ $,fill=black!50] {};
\end{scope}
\begin{scope}[every edge/.style={draw=black!60,line width=1.2}]
  \path (1) edge node {} (2);
  \path (2) edge node {} (3);
  \path (5) edge node {} (6);
  \path (6) edge node [fill=white, rectangle, inner sep=1.0] {$\ldots$} (7);
  \path (8) edge node {} (9);
  \path (9) edge node {} (10);
\end{scope}
\begin{scope}[every edge/.style={draw=black, line width=3}]
  \path (0) edge node [fill=white, rectangle, inner sep=1.0, minimum height = 0.5cm] {$\ldots$} node [below=2] {$f$\;vertices} (1);
  \path (3) edge node {} (4);
  \path (4) edge node [fill=white, rectangle, inner sep=1.0, minimum height = 0.5cm] {$\ldots$} (5);
  \path (7) edge node [fill=white, rectangle, inner sep=1.0, minimum height = 0.5cm] {$\ldots$} node [below=2] {$f$\;vertices} (8);
  \path (10) edge node {} (11);
  \path (11) edge node [fill=white, rectangle, inner sep=1.0, minimum height = 0.5cm] {$\ldots$} (12);
\end{scope}
\begin{scope}[every node/.style={draw,inner sep=11.5,yshift=-4}, every label/.style={rectangle,draw=none,inner sep=6.0}, every fit/.append style=text badly centered]
  \node [fit=(0) (6), label={above}:$k-1$\;times] {};
\end{scope}
\end{tikzpicture}
   \caption{Critical simplices for $f\leq d-2$ and $n=kd-1$ ($k\geq 1$).}
  \label{fig:An-critical-1}
\end{figure}

\begin{theorem}[Critical simplices in case $A_n$]
  The critical simplices for the matching on $K_{n,f}$ are those listed in Table \ref{table:An-critical}.
  In particular the matching is always precise, and has no critical simplices if $n\not\equiv f,-1\pmod d$ or $n > f=d-1$.
  In addition, when there are two critical simplices there is only one alternating path between them (the trivial one).
  \label{thm:matching-An}
\end{theorem}

\begin{proof}
  \textbf{Part 1: critical simplices.}
  The proof is by induction on $n$, the case $n=0$ being trivial (there is one critical simplex for $f=0$, namely the empty simplex $\varnothing$, and no critical simplices for $f>0$).
  Let $\sigma\in K_{n,f}$ be a simplex. Let us consider each of the five cases that can occur in the construction of the matching.
  \begin{enumerate}[(a)]
    \item We have $\{1,\dots,d-1\}\subseteq \sigma$. Then $\sigma$ is critical if and only if $n=d-1$.
    If $n=d-1$ then $\sigma$ is indeed listed in Table \ref{table:An-critical}, as the first of the two critical simplices in the case $n=kd-1$ (here $k=1$).
    Conversely, the only simplex of Table \ref{table:An-critical} which contains $\{1,\dots,d-1\}$ is the first one of the case $n=kd-1$ when $k=1$.
    
    \item We have $n=f$ and $\sigma=\{1,\dots,f\}$.
    Then $\sigma$ is critical, and it is indeed listed as the second of the two critical simplices in the case $n=kd+f$ (here $k=0$).
    
    \item We have $\{1,\dots,f+1\} \subseteq \sigma$.
    Then $\sigma$ is not critical.
    The only listed simplex which contains $\{1,\dots,f+1\}$ is the first one of the case $n=kd-1$ for $k=1$, but it is equal to $\{1,\dots,d-1\}$ and so it must be different from $\sigma$.
    
    \item We have $f+1\not \in \sigma$ and $\{f+2,\dots,d-1\}\nsubseteq \sigma$.
    Then $\sigma$ is not critical.
    It is easy to check that all the listed simplices $\tau$ with $n>f$ and $f+1\not \in \tau$ satisfy $\{f+2,\dots,d-1\}\subseteq \tau$, so none of them can be equal to $\sigma$.
    
    \item We have $\{1,\dots,f,f+2,\dots,d-1\}\subseteq \sigma$ and $f+1\not\in\sigma$.
    Then $\sigma$ is critical if and only if the simplex $\hat\sigma \in K_{n-f-1,\;d-2-f}$ is critical, where $\hat\sigma$ is constructed from $\sigma$ by ignoring the first $f+1$ vertices.
    By induction, $\hat\sigma$ is critical if and only if it is listed in Table \ref{table:An-critical}.
    By taking the simplices of Table \ref{table:An-critical} for $K_{n-f-1,\;d-2-f}$ and adding $1^f0$ at the beginning, one obtains exactly the simplices of Table \ref{table:An-critical} for $K_{n,f}$ (but the two cases are exchanged).
    This concludes the induction argument.
  \end{enumerate}
  
  For fixed $n$ and $f$, the quantity $|\sigma|-v_\varphi(\sigma)$ is constant among the critical simplices $\sigma$.
  More precisely, it is equal to $n-2k$ for $n=kd+f$ (except for $f=d-1$) and to $n-2k+1$ for $n=kd-1$.
  Therefore the matching is precise.
  
  \textbf{Part 2: alternating paths.}
  We want to prove that the only alternating path between critical simplices is the trivial one.
  Consider the case $n=kd+f$ for $k\geq 1$ (the case $n=kd-1$ is analogous).
  Let $\sigma$ and $\tau$ be the two critical simplices (the ones of Figure \ref{fig:An-critical-f}). Notice that $\sigma = \tau \cup \{ n-f \}$.
  To establish a contradiction suppose we have a non-trivial alternating path
  \[ \sigma \vartriangleright \tau_1 \vartriangleleft \sigma_1 \vartriangleright \tau_2 \vartriangleleft \sigma_2 \vartriangleright \dots \vartriangleright \tau_m \vartriangleleft \sigma_m \vartriangleright \tau, \quad m\geq 1. \]
  Since $\tau_1\neq \tau$, we must have $\sigma = \tau_1 \cup \{v_1\}$ with $v_1\neq n-f$. Notice that $n-f$ is the only vertex $v\in \sigma$ which satisfies $v\equiv f+1$ or $v\equiv 0 \pmod d$. Therefore, by Remark \ref{rmk:removed-vertex}, the vertex $v_1$ will never be recovered throughout the alternating path. This is a contradiction since $v_1\in \tau$.
\end{proof}

As a consequence we can compute the homology $H_*(\GW; R)$ for $\GW$ of type $A_n$. This gives a proof of Theorem \ref{teo:simpleind}.

\begin{proof}[Proof of Theorem \ref{teo:simpleind} (Homology in case $A_n$)]
  We simply need to apply the formula given by Theorem \ref{thm:homology-artin-groups} using our precise matching on $K = K_{n,0}$.
  Since $f=0$, there are critical simplices for $n=kd$ or $n=kd-1$.
  The boundary $\delta_{m+1}^\varphi$ is non-trivial only for
  \[
    \begin{cases}
      m = n-2k = k(d-2) & \text{if } n=kd, \\
      m = n-2k+1 = k(d-2) & \text{if } n=kd-1.
    \end{cases}
  \]
  In both cases for $m=k(d-2)$ we have $\rk\delta_{m+1}^\varphi = 1$, and all the other boundaries are trivial.
  Theorem \ref{teo:simpleind} follows.
\end{proof}

\begin{remark}
  A relation with independence complexes can be found for $K_{n,f}$ also when $1 \leq f \leq d-2$.
  Indeed, choosing a suitable weighted matching (similar to the one of Theorem \ref{thm:casoAn}), the set of critical simplices of positive weight in $K_{n,f}$ is
  \[ \{ \sigma \sqcup A_{d-1} \mid \sigma \in \Ind_{d-2}(A_{n-d}) \cap K_{n-d,\,f} \}. \]
\end{remark}

\subsection{Case $B_n$}
\label{sec:precise-matching-Bn}
Consider now a Coxeter graph of type $B_n$, as in Figure \ref{fig:Bn}.
In this case $K$ is again the full simplex on vertices $\{1,2,\dots, n\}$.
The Poincaré polynomial of a Coxeter group of type $B_k$ is given by
\[ \W_{B_k}(q)= [2k]_q!! = \prod_{i=1}^n \, [2i]_q = \prod_\varphi \varphi^{\omega_\varphi(B_k)}, \]
where, for a given cyclotomic polynomial $\varphi = \varphi_d$, we have
\[ \omega_\varphi(B_k) =
  \begin{cases}
    \;\,\left\lfloor \frac{k}{d} \right\rfloor & \text{if $d$ is odd}, \\
    \left\lfloor \frac{k}{d/2} \right\rfloor & \text{if $d$ is even}
  \end{cases}
\]
(see for example \cite{bjorner2006combinatorics}).
Then we can compute the $\varphi$-weight of any simplex $\sigma\in K$ by looking at the subgraph $\Gamma(\sigma)$ induced by $\sigma$. Let $\Gamma_1(\sigma)$, \dots, $\Gamma_m(\sigma)$ be the connected components of $\Gamma(\sigma)$, with cardinality $n_1$, \dots, $n_m$ respectively, where $\Gamma_1(\sigma)$ is the (possibly empty) component that contains the vertex $1\in S$.
Then
\[ v_\varphi(\sigma) = \omega_\varphi(B_{n_1}) + \sum_{i=2}^m \omega_\varphi(A_{n_i}). \]

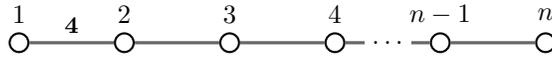
\begin{figure}[htbp]
  \begin{tikzpicture}
\begin{scope}[every node/.style={circle,thick,draw,inner sep=2.5}, every label/.style={rectangle,draw=none}]
  \node (1) at (0.0,0) [label={above,minimum height=13}:$1$] {};
  \node (2) at (1.4,0) [label={above,minimum height=13}:$2$] {};
  \node (3) at (2.8,0) [label={above,minimum height=13}:$3$] {};
  \node (4) at (4.2,0) [label={above,minimum height=13}:$4$] {};
  \node (n-1) at (5.6,0) [label={above,minimum height=13}:$n-1$] {};
  \node (n) at (7.0,0) [label={above,minimum height=13}:$n$] {};
\end{scope}
\begin{scope}[every edge/.style={draw=black!60,line width=1.2}]
  \path (1) edge node {} node[inner sep=3, above] {\bf\small 4} (2);
  \path (2) edge node {} (3);
  \path (3) edge node {} (4);
  \path (4) edge node [fill=white, rectangle, inner sep=3.0] {$\ldots$} (n-1);
  \path (n-1) edge node {} (n);
\end{scope}
\begin{scope}[every edge/.style={draw=black, line width=3}]
\end{scope}
\begin{scope}[every node/.style={draw,inner sep=11.5,yshift=-4}, every label/.style={rectangle,draw=none,inner sep=6.0}, every fit/.append style=text badly centered]
\end{scope}
\end{tikzpicture}
   \caption{A Coxeter graph of type $B_n$.} 
  \label{fig:Bn}
\end{figure}

The situation is quite different depending on the parity of $d$.
If $d$ is odd, the $\varphi$-weight of a $B_{k+1}$ component is equal to the $\varphi$-weight of an $A_k$ component:
\[ \omega_\varphi(B_{k+1}) = \left\lfloor \frac{k+1}{d} \right\rfloor = \omega_\varphi(A_k) \quad \text{($d$ odd)}. \]
For this reason it is possible to construct a very simple matching on $K$: match any simplex $\sigma\in K$ with $\sigma \xor 1$.

\begin{lemma}
  The matching constructed above for $d$ odd is an acyclic weighted matching on $K$, with no critical simplices.
\end{lemma}

\begin{proof}
  Clearly the matching is acyclic and there are no critical simplices.
  Let us proof that the matching is weighted.
  Let $\sigma\in K$.
  The only difference between $\Gamma(\sigma)$ and $\Gamma(\sigma\xor 1)$ is the leftmost connected component, which in one case is of type $B_{k+1}$ and in the other case is of type $A_k$.
  Therefore $v_\varphi(\sigma) = v_\varphi(\sigma \xor 1)$.
\end{proof}

Suppose from now on that $d$ is even.
The simplicial complex $K$ is partitioned as
\[ K = \bigsqcup_{q\geq 0} K_q, \quad \text{where} \quad K_q = \left\{ \sigma \in K \;\middle|\; \left\lfloor \frac{|\Gamma_1(\sigma)|}{d/2} \right\rfloor = q \right\}. \]
Here $\Gamma_1(\sigma)$ is the (possibly empty) connected component of $\Gamma(\sigma)$ which contains the vertex $1$.
Notice that each $K_q$ is a subposet of $K$, but not a subcomplex in general.
For a given simplex $\sigma\in K_q$, let $|\Gamma_1(\sigma)| = q\frac d2 + r$ with $0\leq r < \frac d2$.
The matching on $K_q$ is as follows.
\begin{enumerate}[(a)]
  \item If $r\geq 1$ (i.e.\ $q\frac d2 + 1 \in \sigma$) then match $\sigma$ with $\sigma \setminus \left\{q\frac d2+1\right\}$.
  \item If $r = 0$ (i.e.\ $q\frac d2 + 1 \not\in \sigma$) and $\left\{ q\frac d2+2, \dots, (q+1)\frac d2 \right\} \nsubseteq \sigma$, then match $\sigma$ with $\sigma \cup \left\{q\frac d2+1\right\}$ (unless $n=q\frac d2$, in which case $\sigma$ is critical).
  \item We are left with the simplices $\sigma$ for which neither (a) nor (b) apply, i.e.\ with $r=0$ and $\smash{\left\{ q\frac d2+2, \dots, (q+1)\frac d2 \right\} \subseteq \sigma}$. Ignore the first $q\frac d2 + 1$ vertices and relabel the remaining ones from $1$ to $\smash{n-q\frac d2 -1}$, so that we are left exactly with the simplices of $\smash{K^A_{n-q\frac d2 - 1, \; \frac d2 - 1}}$.
  Then construct the matching on $\smash{K^A_{n-q\frac d2 - 1, \; \frac d2 - 1}}$ as in Section \ref{sec:precise-matching-An}.
\end{enumerate}
Putting together the matchings on each $K_q$ we obtain a matching on the full simplicial complex $K$.

\begin{example}
  For $n=4$ and $d=4$, the simplicial complex $K$ contains $2^4 = 16$ simplices of which $12$ are matched and $4$ are critical.
  For instance, consider $\sigma = \{1,2\} \in K$. Then $q=1$ and $r=0$. Since $4\not\in\sigma$, case (b) occurs. Therefore $\sigma$ is matched with $\sigma \cup \{3\} = \{1,2,3\}$.
  See Table \ref{table:Bn-matching-example} for an explicit description of the matching in this case.
  \begin{table}[htbp]

{\renewcommand{\arraystretch}{1.4}

\begin{tabular}{rcl@{\hspace{15pt}}|c|c|c}

\multicolumn{3}{c|}{Simplices} & $v_\varphi(\sigma)$ & \quad$q$\;\;\;\; & Step \\

\hline

\begin{tikzpicture}
\begin{scope}[every node/.style={circle,thick,draw,inner sep=2.5}, every label/.style={rectangle,draw=none}]
  \node (1) at (0.0,0) [label={above,minimum height=13}:$ $,fill=black!50] {};
  \node (2) at (0.8,0) [label={above,minimum height=13}:$ $,fill=black!50] {};
  \node (3) at (1.6,0) [label={above,minimum height=13}:$ $,fill=black!50] {};
  \node (4) at (2.4,0) [label={above,minimum height=13}:$ $] {};
\end{scope}
\begin{scope}[every edge/.style={draw=black!60,line width=1.2}]
  \path (3) edge node {} (4);
\end{scope}
\begin{scope}[every edge/.style={draw=black, line width=3}]
  \path (1) edge node {} node[inner sep=3, above] {\bf\small 4} (2);
  \path (2) edge node {} (3);
\end{scope}
\begin{scope}[every node/.style={draw,inner sep=11.5,yshift=-4}, every label/.style={rectangle,draw=none,inner sep=6.0}, every fit/.append style=text badly centered]
\end{scope}
\end{tikzpicture}
 & $\longrightarrow$ & 
\begin{tikzpicture}
\begin{scope}[every node/.style={circle,thick,draw,inner sep=2.5}, every label/.style={rectangle,draw=none}]
  \node (1) at (0.0,0) [label={above,minimum height=13}:$ $,fill=black!50] {};
  \node (2) at (0.8,0) [label={above,minimum height=13}:$ $,fill=black!50] {};
  \node (3) at (1.6,0) [label={above,minimum height=13}:$ $] {};
  \node (4) at (2.4,0) [label={above,minimum height=13}:$ $] {};
\end{scope}
\begin{scope}[every edge/.style={draw=black!60,line width=1.2}]
  \path (2) edge node {} (3);
  \path (3) edge node {} (4);
\end{scope}
\begin{scope}[every edge/.style={draw=black, line width=3}]
  \path (1) edge node {} node[inner sep=3, above] {\bf\small 4} (2);
\end{scope}
\begin{scope}[every node/.style={draw,inner sep=11.5,yshift=-4}, every label/.style={rectangle,draw=none,inner sep=6.0}, every fit/.append style=text badly centered]
\end{scope}
\end{tikzpicture}
 & 1 & 1 & (a)/(b) \\

\begin{tikzpicture}
\begin{scope}[every node/.style={circle,thick,draw,inner sep=2.5}, every label/.style={rectangle,draw=none}]
  \node (1) at (0.0,0) [label={above,minimum height=13}:$ $,fill=black!50] {};
  \node (2) at (0.8,0) [label={above,minimum height=13}:$ $] {};
  \node (3) at (1.6,0) [label={above,minimum height=13}:$ $,fill=black!50] {};
  \node (4) at (2.4,0) [label={above,minimum height=13}:$ $,fill=black!50] {};
\end{scope}
\begin{scope}[every edge/.style={draw=black!60,line width=1.2}]
  \path (1) edge node {} node[inner sep=3, above] {\bf\small 4} (2);
  \path (2) edge node {} (3);
\end{scope}
\begin{scope}[every edge/.style={draw=black, line width=3}]
  \path (3) edge node {} (4);
\end{scope}
\begin{scope}[every node/.style={draw,inner sep=11.5,yshift=-4}, every label/.style={rectangle,draw=none,inner sep=6.0}, every fit/.append style=text badly centered]
\end{scope}
\end{tikzpicture}
 & $\longrightarrow$ & 
\begin{tikzpicture}
\begin{scope}[every node/.style={circle,thick,draw,inner sep=2.5}, every label/.style={rectangle,draw=none}]
  \node (1) at (0.0,0) [label={above,minimum height=13}:$ $] {};
  \node (2) at (0.8,0) [label={above,minimum height=13}:$ $] {};
  \node (3) at (1.6,0) [label={above,minimum height=13}:$ $,fill=black!50] {};
  \node (4) at (2.4,0) [label={above,minimum height=13}:$ $,fill=black!50] {};
\end{scope}
\begin{scope}[every edge/.style={draw=black!60,line width=1.2}]
  \path (1) edge node {} node[inner sep=3, above] {\bf\small 4} (2);
  \path (2) edge node {} (3);
\end{scope}
\begin{scope}[every edge/.style={draw=black, line width=3}]
  \path (3) edge node {} (4);
\end{scope}
\begin{scope}[every node/.style={draw,inner sep=11.5,yshift=-4}, every label/.style={rectangle,draw=none,inner sep=6.0}, every fit/.append style=text badly centered]
\end{scope}
\end{tikzpicture}
 & 0 & 0 & (a)/(b) \\

\begin{tikzpicture}
\begin{scope}[every node/.style={circle,thick,draw,inner sep=2.5}, every label/.style={rectangle,draw=none}]
  \node (1) at (0.0,0) [label={above,minimum height=13}:$ $,fill=black!50] {};
  \node (2) at (0.8,0) [label={above,minimum height=13}:$ $] {};
  \node (3) at (1.6,0) [label={above,minimum height=13}:$ $,fill=black!50] {};
  \node (4) at (2.4,0) [label={above,minimum height=13}:$ $] {};
\end{scope}
\begin{scope}[every edge/.style={draw=black!60,line width=1.2}]
  \path (1) edge node {} node[inner sep=3, above] {\bf\small 4} (2);
  \path (2) edge node {} (3);
  \path (3) edge node {} (4);
\end{scope}
\begin{scope}[every edge/.style={draw=black, line width=3}]
\end{scope}
\begin{scope}[every node/.style={draw,inner sep=11.5,yshift=-4}, every label/.style={rectangle,draw=none,inner sep=6.0}, every fit/.append style=text badly centered]
\end{scope}
\end{tikzpicture}
 & $\longrightarrow$ & 
\begin{tikzpicture}
\begin{scope}[every node/.style={circle,thick,draw,inner sep=2.5}, every label/.style={rectangle,draw=none}]
  \node (1) at (0.0,0) [label={above,minimum height=13}:$ $] {};
  \node (2) at (0.8,0) [label={above,minimum height=13}:$ $] {};
  \node (3) at (1.6,0) [label={above,minimum height=13}:$ $,fill=black!50] {};
  \node (4) at (2.4,0) [label={above,minimum height=13}:$ $] {};
\end{scope}
\begin{scope}[every edge/.style={draw=black!60,line width=1.2}]
  \path (1) edge node {} node[inner sep=3, above] {\bf\small 4} (2);
  \path (2) edge node {} (3);
  \path (3) edge node {} (4);
\end{scope}
\begin{scope}[every edge/.style={draw=black, line width=3}]
\end{scope}
\begin{scope}[every node/.style={draw,inner sep=11.5,yshift=-4}, every label/.style={rectangle,draw=none,inner sep=6.0}, every fit/.append style=text badly centered]
\end{scope}
\end{tikzpicture}
 & 0 & 0 & (a)/(b) \\

\begin{tikzpicture}
\begin{scope}[every node/.style={circle,thick,draw,inner sep=2.5}, every label/.style={rectangle,draw=none}]
  \node (1) at (0.0,0) [label={above,minimum height=13}:$ $,fill=black!50] {};
  \node (2) at (0.8,0) [label={above,minimum height=13}:$ $] {};
  \node (3) at (1.6,0) [label={above,minimum height=13}:$ $] {};
  \node (4) at (2.4,0) [label={above,minimum height=13}:$ $,fill=black!50] {};
\end{scope}
\begin{scope}[every edge/.style={draw=black!60,line width=1.2}]
  \path (1) edge node {} node[inner sep=3, above] {\bf\small 4} (2);
  \path (2) edge node {} (3);
  \path (3) edge node {} (4);
\end{scope}
\begin{scope}[every edge/.style={draw=black, line width=3}]
\end{scope}
\begin{scope}[every node/.style={draw,inner sep=11.5,yshift=-4}, every label/.style={rectangle,draw=none,inner sep=6.0}, every fit/.append style=text badly centered]
\end{scope}
\end{tikzpicture}
 & $\longrightarrow$ & 
\begin{tikzpicture}
\begin{scope}[every node/.style={circle,thick,draw,inner sep=2.5}, every label/.style={rectangle,draw=none}]
  \node (1) at (0.0,0) [label={above,minimum height=13}:$ $] {};
  \node (2) at (0.8,0) [label={above,minimum height=13}:$ $] {};
  \node (3) at (1.6,0) [label={above,minimum height=13}:$ $] {};
  \node (4) at (2.4,0) [label={above,minimum height=13}:$ $,fill=black!50] {};
\end{scope}
\begin{scope}[every edge/.style={draw=black!60,line width=1.2}]
  \path (1) edge node {} node[inner sep=3, above] {\bf\small 4} (2);
  \path (2) edge node {} (3);
  \path (3) edge node {} (4);
\end{scope}
\begin{scope}[every edge/.style={draw=black, line width=3}]
\end{scope}
\begin{scope}[every node/.style={draw,inner sep=11.5,yshift=-4}, every label/.style={rectangle,draw=none,inner sep=6.0}, every fit/.append style=text badly centered]
\end{scope}
\end{tikzpicture}
 & 0 & 0 & (a)/(b) \\

\begin{tikzpicture}
\begin{scope}[every node/.style={circle,thick,draw,inner sep=2.5}, every label/.style={rectangle,draw=none}]
  \node (1) at (0.0,0) [label={above,minimum height=13}:$ $,fill=black!50] {};
  \node (2) at (0.8,0) [label={above,minimum height=13}:$ $] {};
  \node (3) at (1.6,0) [label={above,minimum height=13}:$ $] {};
  \node (4) at (2.4,0) [label={above,minimum height=13}:$ $] {};
\end{scope}
\begin{scope}[every edge/.style={draw=black!60,line width=1.2}]
  \path (1) edge node {} node[inner sep=3, above] {\bf\small 4} (2);
  \path (2) edge node {} (3);
  \path (3) edge node {} (4);
\end{scope}
\begin{scope}[every edge/.style={draw=black, line width=3}]
\end{scope}
\begin{scope}[every node/.style={draw,inner sep=11.5,yshift=-4}, every label/.style={rectangle,draw=none,inner sep=6.0}, every fit/.append style=text badly centered]
\end{scope}
\end{tikzpicture}
 & $\longrightarrow$ & 
\begin{tikzpicture}
\begin{scope}[every node/.style={circle,thick,draw,inner sep=2.5}, every label/.style={rectangle,draw=none}]
  \node (1) at (0.0,0) [label={above,minimum height=13}:$ $] {};
  \node (2) at (0.8,0) [label={above,minimum height=13}:$ $] {};
  \node (3) at (1.6,0) [label={above,minimum height=13}:$ $] {};
  \node (4) at (2.4,0) [label={above,minimum height=13}:$ $] {};
\end{scope}
\begin{scope}[every edge/.style={draw=black!60,line width=1.2}]
  \path (1) edge node {} node[inner sep=3, above] {\bf\small 4} (2);
  \path (2) edge node {} (3);
  \path (3) edge node {} (4);
\end{scope}
\begin{scope}[every edge/.style={draw=black, line width=3}]
\end{scope}
\begin{scope}[every node/.style={draw,inner sep=11.5,yshift=-4}, every label/.style={rectangle,draw=none,inner sep=6.0}, every fit/.append style=text badly centered]
\end{scope}
\end{tikzpicture}
 & 0 & 0 & (a)/(b) \\

\begin{tikzpicture}
\begin{scope}[every node/.style={circle,thick,draw,inner sep=2.5}, every label/.style={rectangle,draw=none}]
  \node (1) at (0.0,0) [label={above,minimum height=13}:$ $,fill=black!50] {};
  \node (2) at (0.8,0) [label={above,minimum height=13}:$ $,fill=black!50] {};
  \node (3) at (1.6,0) [label={above,minimum height=13}:$ $,fill=black!50] {};
  \node (4) at (2.4,0) [label={above,minimum height=13}:$ $,fill=black!50] {};
\end{scope}
\begin{scope}[every edge/.style={draw=black!60,line width=1.2}]
\end{scope}
\begin{scope}[every edge/.style={draw=black, line width=3}]
  \path (1) edge node {} node[inner sep=3, above] {\bf\small 4} (2);
  \path (2) edge node {} (3);
  \path (3) edge node {} (4);
\end{scope}
\begin{scope}[every node/.style={draw,inner sep=11.5,yshift=-4}, every label/.style={rectangle,draw=none,inner sep=6.0}, every fit/.append style=text badly centered]
\end{scope}
\end{tikzpicture}
 & & \emph{(critical)} & 2 & 2 & (b) \\

\begin{tikzpicture}
\begin{scope}[every node/.style={circle,thick,draw,inner sep=2.5}, every label/.style={rectangle,draw=none}]
  \node (1) at (0.0,0) [label={above,minimum height=13}:$ $,fill=black!50] {};
  \node (2) at (0.8,0) [label={above,minimum height=13}:$ $,fill=black!50] {};
  \node (3) at (1.6,0) [label={above,minimum height=13}:$ $] {};
  \node (4) at (2.4,0) [label={above,minimum height=13}:$ $,fill=black!50] {};
\end{scope}
\begin{scope}[every edge/.style={draw=black!60,line width=1.2}]
  \path (2) edge node {} (3);
  \path (3) edge node {} (4);
\end{scope}
\begin{scope}[every edge/.style={draw=black, line width=3}]
  \path (1) edge node {} node[inner sep=3, above] {\bf\small 4} (2);
\end{scope}
\begin{scope}[every node/.style={draw,inner sep=11.5,yshift=-4}, every label/.style={rectangle,draw=none,inner sep=6.0}, every fit/.append style=text badly centered]
\end{scope}
\end{tikzpicture}
 & & \emph{(critical)} & 1 & 1 & (c) \\

\begin{tikzpicture}
\begin{scope}[every node/.style={circle,thick,draw,inner sep=2.5}, every label/.style={rectangle,draw=none}]
  \node (1) at (0.0,0) [label={above,minimum height=13}:$ $] {};
  \node (2) at (0.8,0) [label={above,minimum height=13}:$ $,fill=black!50] {};
  \node (3) at (1.6,0) [label={above,minimum height=13}:$ $,fill=black!50] {};
  \node (4) at (2.4,0) [label={above,minimum height=13}:$ $,fill=black!50] {};
\end{scope}
\begin{scope}[every edge/.style={draw=black!60,line width=1.2}]
  \path (1) edge node {} node[inner sep=3, above] {\bf\small 4} (2);
\end{scope}
\begin{scope}[every edge/.style={draw=black, line width=3}]
  \path (2) edge node {} (3);
  \path (3) edge node {} (4);
\end{scope}
\begin{scope}[every node/.style={draw,inner sep=11.5,yshift=-4}, every label/.style={rectangle,draw=none,inner sep=6.0}, every fit/.append style=text badly centered]
\end{scope}
\end{tikzpicture}
 & & \emph{(critical)} & 1 & 0 & (c) \\

\begin{tikzpicture}
\begin{scope}[every node/.style={circle,thick,draw,inner sep=2.5}, every label/.style={rectangle,draw=none}]
  \node (1) at (0.0,0) [label={above,minimum height=13}:$ $] {};
  \node (2) at (0.8,0) [label={above,minimum height=13}:$ $,fill=black!50] {};
  \node (3) at (1.6,0) [label={above,minimum height=13}:$ $,fill=black!50] {};
  \node (4) at (2.4,0) [label={above,minimum height=13}:$ $] {};
\end{scope}
\begin{scope}[every edge/.style={draw=black!60,line width=1.2}]
  \path (1) edge node {} node[inner sep=3, above] {\bf\small 4} (2);
  \path (3) edge node {} (4);
\end{scope}
\begin{scope}[every edge/.style={draw=black, line width=3}]
  \path (2) edge node {} (3);
\end{scope}
\begin{scope}[every node/.style={draw,inner sep=11.5,yshift=-4}, every label/.style={rectangle,draw=none,inner sep=6.0}, every fit/.append style=text badly centered]
\end{scope}
\end{tikzpicture}
 & $\longrightarrow$ & 
\begin{tikzpicture}
\begin{scope}[every node/.style={circle,thick,draw,inner sep=2.5}, every label/.style={rectangle,draw=none}]
  \node (1) at (0.0,0) [label={above,minimum height=13}:$ $] {};
  \node (2) at (0.8,0) [label={above,minimum height=13}:$ $,fill=black!50] {};
  \node (3) at (1.6,0) [label={above,minimum height=13}:$ $] {};
  \node (4) at (2.4,0) [label={above,minimum height=13}:$ $] {};
\end{scope}
\begin{scope}[every edge/.style={draw=black!60,line width=1.2}]
  \path (1) edge node {} node[inner sep=3, above] {\bf\small 4} (2);
  \path (2) edge node {} (3);
  \path (3) edge node {} (4);
\end{scope}
\begin{scope}[every edge/.style={draw=black, line width=3}]
\end{scope}
\begin{scope}[every node/.style={draw,inner sep=11.5,yshift=-4}, every label/.style={rectangle,draw=none,inner sep=6.0}, every fit/.append style=text badly centered]
\end{scope}
\end{tikzpicture}
 & 0 & 0 & (c) \\

\begin{tikzpicture}
\begin{scope}[every node/.style={circle,thick,draw,inner sep=2.5}, every label/.style={rectangle,draw=none}]
  \node (1) at (0.0,0) [label={above,minimum height=13}:$ $] {};
  \node (2) at (0.8,0) [label={above,minimum height=13}:$ $,fill=black!50] {};
  \node (3) at (1.6,0) [label={above,minimum height=13}:$ $] {};
  \node (4) at (2.4,0) [label={above,minimum height=13}:$ $,fill=black!50] {};
\end{scope}
\begin{scope}[every edge/.style={draw=black!60,line width=1.2}]
  \path (1) edge node {} node[inner sep=3, above] {\bf\small 4} (2);
  \path (2) edge node {} (3);
  \path (3) edge node {} (4);
\end{scope}
\begin{scope}[every edge/.style={draw=black, line width=3}]
\end{scope}
\begin{scope}[every node/.style={draw,inner sep=11.5,yshift=-4}, every label/.style={rectangle,draw=none,inner sep=6.0}, every fit/.append style=text badly centered]
\end{scope}
\end{tikzpicture}
 & & \emph{(critical)} & 0 & 0 & (c) \\

\end{tabular}

}

\vskip0.3cm

     \caption{Matching in the case $B_n$ with $n=4$ and $d=4$.}
    \label{table:Bn-matching-example}
  \end{table}
  \label{example:Bn-matching}
\end{example}

\begin{lemma}
  The matching constructed above for $d$ even is an acyclic weighted matching on $K$.
\end{lemma}

\begin{proof}
  \textbf{Part 1: the matching is acyclic.}
  The map $K \to (\N,\leq)$ which sends $\sigma$ to $q = \left\lfloor \frac{|\Gamma_1(\sigma)|}{d/2} \right\rfloor$ is a poset map compatible with the matching, and its fibers are exactly the subsets $K_q$ for $q\in\N$.
  Therefore we only need to prove that the matching on each fiber $K_q$ is acyclic.
  
  Let $P = \{ p_\text{c}, p_\text{a,b} \}$ be a two-element totally ordered poset with $p_\text{c} > p_\text{a,b}$.
  For a fixed $q\in\N$, consider the map $\eta\colon K_q \to P$ which sends $\sigma$ to the $p_\text{x}$ such that $\sigma$ occurs in case (x) (here cases (a) and (b) are united).
  Clearly $\eta$ is compatible with the matching. We want to prove that it is a poset map, and for this we only need to show that if $\eta(\tau) = p_\text{c}$ and $\tau\leq \sigma$ then $\eta(\sigma) = p_\text{c}$ also.
  We have that $\smash{\left\{ 1,\dots, q\frac d2, q\frac d2+2, \dots, (q+1)\frac d2 \right\} \subseteq \tau \subseteq \sigma}$. The simplex $\sigma$ cannot contain the vertex $\smash{q\frac d2 +1}$, because otherwise we would have $\sigma\in K_{q+1}$. Therefore $\eta(\sigma) = p_\text{c}$.
  
  On the fiber $\eta^{-1}(p_\text{a,b})$ the matching is acyclic because the same vertex $\smash{q\frac d2 +1}$ is always added or removed. On the fiber $\eta^{-1}(p_\text{c})$ the matching is acyclic by Lemma \ref{lemma:An-acyclic-weighted}.
  Therefore the entire matching on $K_q$ is acyclic.

  \textbf{Part 2: the matching is weighted.}
  Let $\sigma\in K_q$ be a simplex which occurs either in case (a) or case (b).
  We want to show that $\smash{v_\varphi(\sigma) = v_\varphi\left( \sigma \xor \left(q\frac d2 + 1\right) \right)}$.
  Suppose without loss of generality that $\sigma$ occurs in case (a), i.e.\ $r \geq 1$ and $q\frac d2 + 1 \in\sigma$. Let $\tau = \sigma \setminus \left\{ q\frac d2 + 1 \right\}$.
  The only difference between $\Gamma(\sigma)$ and $\Gamma(\tau)$ is that $\Gamma(\sigma)$ has a $\smash{B_{q\frac d2 +r}}$ component, whereas $\Gamma(\tau)$ has a $\smash{B_{q\frac d2}}$ component and an $A_{r-1}$ component instead.
  Since $\smash{1 \leq r < \frac d2}$, we have that
  \begin{align*}
    & \omega_\varphi(B_{q \frac d2 + r}) = \omega_\varphi(B_{q\frac d2}) = q \\
    & \text{and} \quad \omega_\varphi(A_{r-1}) = 0
  \end{align*}
  (the second equality holds because $r-1 \leq \frac d2 - 1 \leq d-2$).
  Therefore $v_\varphi(\sigma) = v_\varphi(\tau)$.
  
  For simplices occurring in case (c), the matching only involves changes in the connected components not containing the first vertex (i.e.\ connected components of type $A_k$). Therefore Lemma \ref{lemma:An-acyclic-weighted} applies.
\end{proof}

Now we are going to describe the critical simplices.
The matching has no critical simplices when $n$ is not a multiple of $\frac d2$.
On the other hand, if $n = k\frac d2$, we have two families of critical simplices: $\sigma_q$ (for $0\leq q \leq k-2$) and $\sigma_q'$ (for $0\leq q\leq k$).
See Table \ref{table:Bn-critical} and Figure \ref{fig:Bn-critical} for the definition of these simplices.
For instance, in Example \ref{example:Bn-matching}, the critical simplices are: $\sigma_0 = \{2,3,4\}$, $\sigma_2' = \{1,2,3,4\}$, $\sigma_1' = \{1,2,4\}$ and $\sigma_0' = \{1,3\}$.
See also Table \ref{table:Bn-critical-by-dimension}, where the critical simplices are listed by dimension.

\begin{table}[phtb]
  {\renewcommand{\arraystretch}{1.4}\begin{tabular}{c|c|c|c}
    Simplices & $|\sigma|$ & $v_\varphi(\sigma)$ & \\
    \hline \hline
    $\sigma_q = 1^{q\frac d2} (01^{\frac d2 -1})^{k-q-2}01^{d-1}$ & $n-k+q+1$ & $q+1$ & $0 \leq q \leq k-2$ \\
    $\sigma_q' = 1^{q\frac d2} (01^{\frac d2 -1})^{k-q}$ & $n-k+q$ & $q$ & $0 \leq q \leq k$ \\
    \hline
  \end{tabular}
  }
  \vskip0.3cm
  \caption{Description of the critical simplices for $B_n$, where $d$ is even and $n=k\frac d2$.}
  \label{table:Bn-critical}
\end{table}

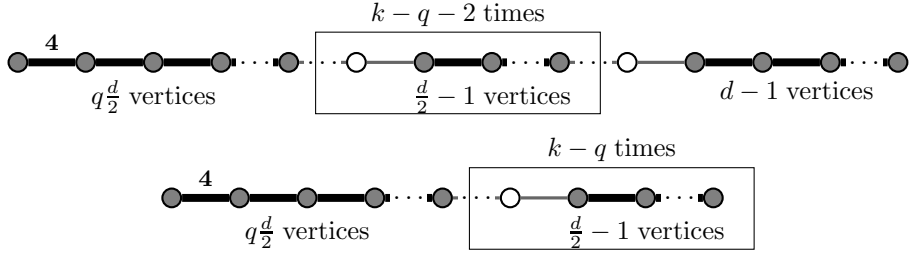
\begin{figure}[phtb]
  \begin{tikzpicture}
\begin{scope}[every node/.style={circle,thick,draw,inner sep=2.5}, every label/.style={rectangle,draw=none}]
  \node (0) at (0.0,0) [label={above,minimum height=13}:$ $,fill=black!50] {};
  \node (1) at (0.9,0) [label={above,minimum height=13}:$ $,fill=black!50] {};
  \node (2) at (1.8,0) [label={above,minimum height=13}:$ $,fill=black!50] [label={below}:$q\frac{d}{2}$\;vertices] {};
  \node (3) at (2.7,0) [label={above,minimum height=13}:$ $,fill=black!50] {};
  \node (4) at (3.6,0) [label={above,minimum height=13}:$ $,fill=black!50] {};
  \node (5) at (4.5,0) [label={above,minimum height=13}:$ $] {};
  \node (6) at (5.4,0) [label={above,minimum height=13}:$ $,fill=black!50] {};
  \node (7) at (6.3,0) [label={above,minimum height=13}:$ $,fill=black!50] [label={below}:$\frac{d}{2}-1$\;vertices] {};
  \node (8) at (7.2,0) [label={above,minimum height=13}:$ $,fill=black!50] {};
  \node (9) at (8.1,0) [label={above,minimum height=13}:$ $] {};
  \node (10) at (9.0,0) [label={above,minimum height=13}:$ $,fill=black!50] {};
  \node (11) at (9.9,0) [label={above,minimum height=13}:$ $,fill=black!50] {};
  \node (12) at (10.8,0) [label={above,minimum height=13}:$ $,fill=black!50] {};
  \node (13) at (11.7,0) [label={above,minimum height=13}:$ $,fill=black!50] {};
\end{scope}
\begin{scope}[every edge/.style={draw=black!60,line width=1.2}]
  \path (4) edge node [fill=white, rectangle, inner sep=1.0] {$\ldots$} (5);
  \path (5) edge node {} (6);
  \path (8) edge node [fill=white, rectangle, inner sep=1.0] {$\ldots$} (9);
  \path (9) edge node {} (10);
\end{scope}
\begin{scope}[every edge/.style={draw=black, line width=3}]
  \path (0) edge node {} node[inner sep=3, above] {\bf\small 4} (1);
  \path (1) edge node {} (2);
  \path (2) edge node {} (3);
  \path (3) edge node [fill=white, rectangle, inner sep=1.0, minimum height = 0.5cm] {$\ldots$} (4);
  \path (6) edge node {} (7);
  \path (7) edge node [fill=white, rectangle, inner sep=1.0, minimum height = 0.5cm] {$\ldots$} (8);
  \path (10) edge node {} (11);
  \path (11) edge node {} node [below=2] {$d-1$\;vertices} (12);
  \path (12) edge node [fill=white, rectangle, inner sep=1.0, minimum height = 0.5cm] {$\ldots$} (13);
\end{scope}
\begin{scope}[every node/.style={draw,inner sep=11.5,yshift=-4}, every label/.style={rectangle,draw=none,inner sep=6.0}, every fit/.append style=text badly centered]
  \node [fit=(5) (8), label={above}:$k-q-2$\;times] {};
\end{scope}
\end{tikzpicture}
   \begin{tikzpicture}
\begin{scope}[every node/.style={circle,thick,draw,inner sep=2.5}, every label/.style={rectangle,draw=none}]
  \node (0) at (0.0,0) [label={above,minimum height=13}:$ $,fill=black!50] {};
  \node (1) at (0.9,0) [label={above,minimum height=13}:$ $,fill=black!50] {};
  \node (2) at (1.8,0) [label={above,minimum height=13}:$ $,fill=black!50] [label={below}:$q\frac{d}{2}$\;vertices] {};
  \node (3) at (2.7,0) [label={above,minimum height=13}:$ $,fill=black!50] {};
  \node (4) at (3.6,0) [label={above,minimum height=13}:$ $,fill=black!50] {};
  \node (5) at (4.5,0) [label={above,minimum height=13}:$ $] {};
  \node (6) at (5.4,0) [label={above,minimum height=13}:$ $,fill=black!50] {};
  \node (7) at (6.3,0) [label={above,minimum height=13}:$ $,fill=black!50] [label={below}:$\frac{d}{2}-1$\;vertices] {};
  \node (8) at (7.2,0) [label={above,minimum height=13}:$ $,fill=black!50] {};
\end{scope}
\begin{scope}[every edge/.style={draw=black!60,line width=1.2}]
  \path (4) edge node [fill=white, rectangle, inner sep=1.0] {$\ldots$} (5);
  \path (5) edge node {} (6);
\end{scope}
\begin{scope}[every edge/.style={draw=black, line width=3}]
  \path (0) edge node {} node[inner sep=3, above] {\bf\small 4} (1);
  \path (1) edge node {} (2);
  \path (2) edge node {} (3);
  \path (3) edge node [fill=white, rectangle, inner sep=1.0, minimum height = 0.5cm] {$\ldots$} (4);
  \path (6) edge node {} (7);
  \path (7) edge node [fill=white, rectangle, inner sep=1.0, minimum height = 0.5cm] {$\ldots$} (8);
\end{scope}
\begin{scope}[every node/.style={draw,inner sep=11.5,yshift=-4}, every label/.style={rectangle,draw=none,inner sep=6.0}, every fit/.append style=text badly centered]
  \node [fit=(5) (8), label={above}:$k-q$\;times] {};
\end{scope}
\end{tikzpicture}
   \caption{Critical simplices for $B_n$, where $d$ is even and $n=k\frac d2$. Above is the simplex $\sigma_q$ ($0\leq q\leq k-2$) and below is the simplex $\sigma_q'$ ($0\leq q \leq k$).}
  \label{fig:Bn-critical}
\end{figure}

\begin{table}[phtb]
  {\renewcommand{\arraystretch}{1.4}\begin{tabular}{c|c|c}
    $|\sigma|$ & $v_\varphi(\sigma)$ & Simplices \\
    \hline \hline
    $n$ & $k$ & $\sigma_k'$ \\
    $n-1$ & $k-1$ & $\sigma_{k-2}$, $\sigma_{k-1}'$ \\
    $n-2$ & $k-2$ & $\sigma_{k-3}$, $\sigma_{k-2}'$ \\
    \vdots & \vdots & \vdots \\
    $n-k+1$ & $1$ & $\sigma_0$, $\sigma_1'$ \\
    $n-k$ & $0$ & $\sigma_0'$ \\
    \hline
  \end{tabular}
  }
  \vskip0.3cm
  \caption{Critical simplices for $B_n$ by dimension, where $d$ is even and $n=k\frac d2$.}
  \label{table:Bn-critical-by-dimension}
\end{table}

\begin{theorem}[Critical simplices in case $B_n$]
  The critical simplices for the matching on $K$ are those defined in Table \ref{table:Bn-critical}.
  In particular the matching is always precise, and has no critical simplices if $d$ is odd or if $d$ is even and $n\not\equiv 0 \pmod[\big]{\frac d2}$.
  In addition, if $d$ is even and $n\equiv 0 \pmod[\big]{\frac d2}$, the incidence numbers between the critical simplices in the Morse complex are as follows:
  \begin{align*}
    [\sigma_q: \sigma_{q-1}]^\M & = (-1)^{(q-1)\frac d2}, \\
    [\sigma_q: \sigma_q']^\M & = (-1)^{(k-1)(\frac d2-1) + q}, \\
    [\sigma_q': \sigma_{q-2}]^\M & = (-1)^{k(\frac d2-1)+q}, \\
    [\sigma_q': \sigma_{q-1}']^\M & = (-1)^{(q-1)\frac d2}.
  \end{align*}
  \label{thm:matching-Bn}
\end{theorem}

\begin{proof}
  
  \textbf{Part 1: critical simplices.}
  For $d$ odd there is nothing to prove. Suppose from now on that $d$ is even.
  For a given simplex $\sigma\in K$, let us consider each of the three cases that can occur in the construction of the matching.
  \begin{enumerate}[(a)]
    \item We have $\smash{q\frac d2+1\in \sigma}$, and $\sigma$ is not critical. Indeed, none of the simplices of Table \ref{table:Bn-critical} contains the vertex $\smash{q\frac d2+1}$.
    \item We have $\smash{q\frac d2+1\not\in\sigma}$ and $\smash{\left\{ q\frac d2 +2, \dots, (q+1)\frac d2\right\} \nsubseteq \sigma}$. In this case $\sigma$ is critical if and only if $\smash{n=q\frac d2}$.
    Indeed, the only simplex of this type in Table \ref{table:Bn-critical} is $\sigma_k'$ (which occurs for $q=k$ i.e.\ $\smash{n=q\frac d2}$).
    \item In the remaining case, we end up with the matching on $K^A_{n-q\frac d2 -1, \; \frac d2 - 1}$. By Theorem \ref{thm:matching-An}, this matching admits critical simplices if and only if
    \[ \textstyle n-q\frac d2 -1 \equiv \frac d2-1 \pmod d \quad \text{or} \quad n-q\frac d2 -1 \equiv -1 \pmod d, \]
    i.e.\ for $\smash{n\equiv 0 \pmod[\big]{\frac d2}}$.
    Notice that if $\smash{f=\frac d2 -1}$ then $\smash{d-2-f = \frac d2-1}$ also.
    Therefore, again by Theorem \ref{thm:matching-An}, for $\smash{n\equiv 0 \pmod[\big]{\frac d2}}$ the critical simplices are exactly the ones listed in Table \ref{table:Bn-critical}.
  \end{enumerate}
  
  For a fixed $\smash{n=k\frac d2}$, the quantity $|\sigma|-v_\varphi(\sigma)$ is constant among the critical simplices and is equal to $n-k$.
  Thus the matching is precise.
  
  \textbf{Part 2: incidence numbers.}
We are going to show that there is exactly one alternating path for each of the four pairs.
  Notice that, if $\sigma\to\tau$ is in the matching of $K_q$, then $\sigma = \tau \cup \{v\}$ with $\smash{v\equiv 1 \pmod[\big]{\frac d2}}$ and $\smash{v \geq q\frac d2 +1}$.
  In particular, if at a certain point of an alternating path one removes a vertex $v$ with $\smash{v\not\equiv 1 \pmod[\big]{\frac d2}}$, then that vertex will never be added again.
  But all the critical simplices contain every vertex $v$ with $\smash{v \not\equiv 1 \pmod[\big]{\frac d2}}$, so any alternating path between critical simplices cannot ever drop a vertex $v$ with $\smash{v \not\equiv 1 \pmod[\big]{\frac d2}}$.
  
  Let us now consider each of the four pairs.
  \begin{itemize}
    \item $(\sigma_q, \sigma_{q-1})$. We have that $\sigma_q = \sigma_{q-1} \cup \{v\}$ where $v = (q-1)\frac d2 + 1$, so there is the trivial alternating path $\sigma_q \to \sigma_{q-1}$ which contributes to the incidence number $[\sigma_q : \sigma_{q-1}]^\M$ by
    \[ [\sigma_q : \sigma_{q-1}] = (-1)^{|\{ w\in \sigma_q \mid w < v \}|} = (-1)^{(q-1)\frac d2}. \]
    Suppose by contradiction that there exists some other (non-trivial) alternating path
    \[ \sigma_q \rhd \tau_1 \lhd \rho_1\rhd \tau_2\lhd \rho_2\rhd  \dots \rhd \tau_m\lhd \rho_m \rhd \sigma_{q-1}, \quad m \geq 1. \]
    Let $\sigma_q = \tau_1 \cup \{u\}$. By the previous considerations, we must have $u\equiv 1 \pmod[\big]{\frac d2}$.
    If $u=(k-1)\frac d2 + 1$ then $\tau_1 = \sigma_q'$, but this is not possible since alternating paths stop at critical simplices. Similarly, if $u=(q-1)\frac d2 + 1$ then we would stop at $\tau_1 = \sigma_{q-1}$.
    Therefore we must have $u \leq (q-2)\frac d2 + 1$. But then $\tau_1 \in K_{q'}$ with $q' \leq q-2$, and by induction all the subsequent simplices in the alternating path must lie in
    \[ \bigcup_{q'' \leq q-2} K_{q''}. \]
    In particular $\sigma_{q-1} \in K_{q''}$ for some $q'' \leq q-2$, but this is a contradiction since $\sigma_{q-1}\in K_{q-1}$.
    
    \item $(\sigma_q, \sigma_q')$. This case is similar to the previous one, except for the fact that $\sigma_q = \sigma_q' \cup \{v \}$ for $v = (k-1)\frac d2 + 1$.
    So the only alternating path is the trivial one, which contributes to the incidence number by
    \begin{align*}
      [\sigma_q : \sigma_q'] & = (-1)^{|\{ w\in \sigma_q \mid w < v \}|} \\
      & = (-1)^{q\frac d2 + (k-q-1)(\frac d2-1)} \\
      & = (-1)^{(k-1)(\frac d2 - 1) + q}.
    \end{align*}
    
    \item $(\sigma_q', \sigma_{q-1}')$. This case is also similar to the previous ones. Here we have $\sigma_q' = \sigma_{q-1}' \cup \{v \}$ with $v = (q-1)\frac d2 + 1$, so the contribution to the incidence number due to the trivial alternating path is
    \[ [\sigma_q' : \sigma_{q-1}'] = (-1)^{|\{ w\in \sigma_q' \mid w < v \}|} = (-1)^{(q-1)\frac d2}. \]
    
    \item $(\sigma_q', \sigma_{q-2})$. This case is more complicated because the only alternating path is non-trivial.
    Suppose we have an alternating path
    \[ \sigma_q' \rhd \tau_1 \lhd \rho_1\rhd \tau_2\lhd \rho_2\rhd  \dots \rhd \tau_m\lhd \rho_m \rhd \sigma_{q-2}, \quad m \geq 1 \]
    ($m$ must be at least $1$ because $\sigma_{q-2}$ is not a face of $\sigma_q'$ in $K$).
    Let $\sigma_q' = \tau_1 \cup \{v\}$, with $v\equiv 1 \pmod[\big]{\frac d2}$.
    If $v = (q-1)\frac d2 + 1$ then $\tau_1 = \sigma_{q-1}'$ and we must stop because $\sigma_{q-1}'$ is already critical.
    If $v \leq (q-3)\frac d2 + 1$ we fall into some $K_{q'}$ with $q'\leq q-3$ and it is not possible to reach $\sigma_{q-2}\in K_{q-2}$.
    Therefore necessarily $v = (q-2) \frac d2 + 1$.
    Then $\tau_1$ is matched with
    \[ \rho_1 = \tau_1 \cup \textstyle\left\{ q\frac d2 + 1 \right\}. \]
    Then again, $\tau_2 = \rho_1 \setminus \{v_2\}$ and the only possibility for $v_2$ (in order to have $\tau_2 \neq \tau_1$ and not to fall into some $K_{q'}$ with $q' \leq q-3$) is $v_2 = (q-1)\frac d2 + 1$.
    Proceeding by induction we obtain that
    \begin{align*}
      \tau_i & = \rho_{i-1} \setminus \textstyle \left\{ (q-3+i)\frac d2 + 1 \right\}, \\
      \rho_i & = \tau_i \cup \textstyle \left\{ (q-1+i) \frac d2 + 1 \right\}.
    \end{align*}
    The path stops at $\tau_{k-q+1} = \sigma_{q-2}$, and its length is $m=k-q$.
    See Table \ref{table:Bn-alternating-path} for an example.
    The contribution of this path to the incidence number is
    \[ (-1)^{k-q} [\sigma_q':\tau_1][\rho_1:\tau_1] \cdots [\rho_{k-q}:\tau_{k-q}][\rho_{k-q}:\sigma_{q-2}] \]
    where
    \begin{align*}
      [\rho_{i-1}: \tau_i] & = (-1)^{(q-2)\frac d2 + (i-1)(\frac d2-1) }, \\
      [\rho_i : \tau_i] & = (-1)^{(q-2)\frac d2 + (i-1)(\frac d2-1) + d-1 } \\
      & = (-1)^{(q-2)\frac d2 + (i-1)(\frac d2-1) +1}
    \end{align*}
    (these formulas also hold for $\rho_0 = \sigma_q'$ and $\tau_{k-q+1} = \sigma_{q-2}$).
    Then
    \begin{align*}
      [\sigma_q' : \sigma_{q-2}]^\M & = (-1)^{k-q} \left(\prod_{i=1}^{k-q} [\rho_{i-1}:\tau_i][\rho_i:\tau_i]\right) [\rho_{k-q}: \sigma_{q-2}] \\
      & = (-1)^{k-q} (-1)^{k-q} (-1)^{(q-2)\frac d2 + (k-q)(\frac d2 - 1)} \\
      & = (-1)^{k(\frac d2 - 1) + q}. \qedhere
    \end{align*}
  \end{itemize}
\end{proof}

\begin{table}[htbp]

{\renewcommand{\arraystretch}{1.4}

\begin{tabular}{rcl@{\hspace{15pt}}|c}

\multicolumn{3}{c|}{Simplices} & $v_\varphi(\sigma)$ \\

\hline

$\sigma_2'$ & $=$ & 
\begin{tikzpicture}
\begin{scope}[every node/.style={circle,thick,draw,inner sep=2.5}, every label/.style={rectangle,draw=none}]
  \node (1) at (0.0,0) [label={above,minimum height=13}:$ $,fill=black!50] {};
  \node (2) at (0.8,0) [label={above,minimum height=13}:$ $,fill=black!50] {};
  \node (3) at (1.6,0) [label={above,minimum height=13}:$ $,fill=black!50] {};
  \node (4) at (2.4,0) [label={above,minimum height=13}:$ $,fill=black!50] {};
  \node (5) at (3.2,0) [label={above,minimum height=13}:$ $] {};
  \node (6) at (4.0,0) [label={above,minimum height=13}:$ $,fill=black!50] {};
  \node (7) at (4.8,0) [label={above,minimum height=13}:$ $] {};
  \node (8) at (5.6,0) [label={above,minimum height=13}:$ $,fill=black!50] {};
  \node (9) at (6.4,0) [label={above,minimum height=13}:$ $] {};
  \node (10) at (7.2,0) [label={above,minimum height=13}:$ $,fill=black!50] {};
\end{scope}
\begin{scope}[every edge/.style={draw=black!60,line width=1.2}]
  \path (4) edge node {} (5);
  \path (5) edge node {} (6);
  \path (6) edge node {} (7);
  \path (7) edge node {} (8);
  \path (8) edge node {} (9);
  \path (9) edge node {} (10);
\end{scope}
\begin{scope}[every edge/.style={draw=black, line width=3}]
  \path (1) edge node {} node[inner sep=3, above] {\bf\small 4} (2);
  \path (2) edge node {} (3);
  \path (3) edge node {} (4);
\end{scope}
\begin{scope}[every node/.style={draw,inner sep=11.5,yshift=-4}, every label/.style={rectangle,draw=none,inner sep=6.0}, every fit/.append style=text badly centered]
\end{scope}
\end{tikzpicture}
 & 2 \\

$\tau_1$ & $=$ & 
\begin{tikzpicture}
\begin{scope}[every node/.style={circle,thick,draw,inner sep=2.5}, every label/.style={rectangle,draw=none}]
  \node (1) at (0.0,0) [label={above,minimum height=13}:$ $] {};
  \node (2) at (0.8,0) [label={above,minimum height=13}:$ $,fill=black!50] {};
  \node (3) at (1.6,0) [label={above,minimum height=13}:$ $,fill=black!50] {};
  \node (4) at (2.4,0) [label={above,minimum height=13}:$ $,fill=black!50] {};
  \node (5) at (3.2,0) [label={above,minimum height=13}:$ $] {};
  \node (6) at (4.0,0) [label={above,minimum height=13}:$ $,fill=black!50] {};
  \node (7) at (4.8,0) [label={above,minimum height=13}:$ $] {};
  \node (8) at (5.6,0) [label={above,minimum height=13}:$ $,fill=black!50] {};
  \node (9) at (6.4,0) [label={above,minimum height=13}:$ $] {};
  \node (10) at (7.2,0) [label={above,minimum height=13}:$ $,fill=black!50] {};
\end{scope}
\begin{scope}[every edge/.style={draw=black!60,line width=1.2}]
  \path (1) edge node {} node[inner sep=3, above] {\bf\small 4} (2);
  \path (4) edge node {} (5);
  \path (5) edge node {} (6);
  \path (6) edge node {} (7);
  \path (7) edge node {} (8);
  \path (8) edge node {} (9);
  \path (9) edge node {} (10);
\end{scope}
\begin{scope}[every edge/.style={draw=black, line width=3}]
  \path (2) edge node {} (3);
  \path (3) edge node {} (4);
\end{scope}
\begin{scope}[every node/.style={draw,inner sep=11.5,yshift=-4}, every label/.style={rectangle,draw=none,inner sep=6.0}, every fit/.append style=text badly centered]
\end{scope}
\end{tikzpicture}
 & 1 \\

$\rho_1$ & $=$ & 
\begin{tikzpicture}
\begin{scope}[every node/.style={circle,thick,draw,inner sep=2.5}, every label/.style={rectangle,draw=none}]
  \node (1) at (0.0,0) [label={above,minimum height=13}:$ $] {};
  \node (2) at (0.8,0) [label={above,minimum height=13}:$ $,fill=black!50] {};
  \node (3) at (1.6,0) [label={above,minimum height=13}:$ $,fill=black!50] {};
  \node (4) at (2.4,0) [label={above,minimum height=13}:$ $,fill=black!50] {};
  \node (5) at (3.2,0) [label={above,minimum height=13}:$ $,fill=black!50] {};
  \node (6) at (4.0,0) [label={above,minimum height=13}:$ $,fill=black!50] {};
  \node (7) at (4.8,0) [label={above,minimum height=13}:$ $] {};
  \node (8) at (5.6,0) [label={above,minimum height=13}:$ $,fill=black!50] {};
  \node (9) at (6.4,0) [label={above,minimum height=13}:$ $] {};
  \node (10) at (7.2,0) [label={above,minimum height=13}:$ $,fill=black!50] {};
\end{scope}
\begin{scope}[every edge/.style={draw=black!60,line width=1.2}]
  \path (1) edge node {} node[inner sep=3, above] {\bf\small 4} (2);
  \path (6) edge node {} (7);
  \path (7) edge node {} (8);
  \path (8) edge node {} (9);
  \path (9) edge node {} (10);
\end{scope}
\begin{scope}[every edge/.style={draw=black, line width=3}]
  \path (2) edge node {} (3);
  \path (3) edge node {} (4);
  \path (4) edge node {} (5);
  \path (5) edge node {} (6);
\end{scope}
\begin{scope}[every node/.style={draw,inner sep=11.5,yshift=-4}, every label/.style={rectangle,draw=none,inner sep=6.0}, every fit/.append style=text badly centered]
\end{scope}
\end{tikzpicture}
 & 1 \\

$\tau_2$ & $=$ & 
\begin{tikzpicture}
\begin{scope}[every node/.style={circle,thick,draw,inner sep=2.5}, every label/.style={rectangle,draw=none}]
  \node (1) at (0.0,0) [label={above,minimum height=13}:$ $] {};
  \node (2) at (0.8,0) [label={above,minimum height=13}:$ $,fill=black!50] {};
  \node (3) at (1.6,0) [label={above,minimum height=13}:$ $] {};
  \node (4) at (2.4,0) [label={above,minimum height=13}:$ $,fill=black!50] {};
  \node (5) at (3.2,0) [label={above,minimum height=13}:$ $,fill=black!50] {};
  \node (6) at (4.0,0) [label={above,minimum height=13}:$ $,fill=black!50] {};
  \node (7) at (4.8,0) [label={above,minimum height=13}:$ $] {};
  \node (8) at (5.6,0) [label={above,minimum height=13}:$ $,fill=black!50] {};
  \node (9) at (6.4,0) [label={above,minimum height=13}:$ $] {};
  \node (10) at (7.2,0) [label={above,minimum height=13}:$ $,fill=black!50] {};
\end{scope}
\begin{scope}[every edge/.style={draw=black!60,line width=1.2}]
  \path (1) edge node {} node[inner sep=3, above] {\bf\small 4} (2);
  \path (2) edge node {} (3);
  \path (3) edge node {} (4);
  \path (6) edge node {} (7);
  \path (7) edge node {} (8);
  \path (8) edge node {} (9);
  \path (9) edge node {} (10);
\end{scope}
\begin{scope}[every edge/.style={draw=black, line width=3}]
  \path (4) edge node {} (5);
  \path (5) edge node {} (6);
\end{scope}
\begin{scope}[every node/.style={draw,inner sep=11.5,yshift=-4}, every label/.style={rectangle,draw=none,inner sep=6.0}, every fit/.append style=text badly centered]
\end{scope}
\end{tikzpicture}
 & 1 \\

$\rho_2$ & $=$ & 
\begin{tikzpicture}
\begin{scope}[every node/.style={circle,thick,draw,inner sep=2.5}, every label/.style={rectangle,draw=none}]
  \node (1) at (0.0,0) [label={above,minimum height=13}:$ $] {};
  \node (2) at (0.8,0) [label={above,minimum height=13}:$ $,fill=black!50] {};
  \node (3) at (1.6,0) [label={above,minimum height=13}:$ $] {};
  \node (4) at (2.4,0) [label={above,minimum height=13}:$ $,fill=black!50] {};
  \node (5) at (3.2,0) [label={above,minimum height=13}:$ $,fill=black!50] {};
  \node (6) at (4.0,0) [label={above,minimum height=13}:$ $,fill=black!50] {};
  \node (7) at (4.8,0) [label={above,minimum height=13}:$ $,fill=black!50] {};
  \node (8) at (5.6,0) [label={above,minimum height=13}:$ $,fill=black!50] {};
  \node (9) at (6.4,0) [label={above,minimum height=13}:$ $] {};
  \node (10) at (7.2,0) [label={above,minimum height=13}:$ $,fill=black!50] {};
\end{scope}
\begin{scope}[every edge/.style={draw=black!60,line width=1.2}]
  \path (1) edge node {} node[inner sep=3, above] {\bf\small 4} (2);
  \path (2) edge node {} (3);
  \path (3) edge node {} (4);
  \path (8) edge node {} (9);
  \path (9) edge node {} (10);
\end{scope}
\begin{scope}[every edge/.style={draw=black, line width=3}]
  \path (4) edge node {} (5);
  \path (5) edge node {} (6);
  \path (6) edge node {} (7);
  \path (7) edge node {} (8);
\end{scope}
\begin{scope}[every node/.style={draw,inner sep=11.5,yshift=-4}, every label/.style={rectangle,draw=none,inner sep=6.0}, every fit/.append style=text badly centered]
\end{scope}
\end{tikzpicture}
 & 1 \\

$\tau_3$ & $=$ & 
\begin{tikzpicture}
\begin{scope}[every node/.style={circle,thick,draw,inner sep=2.5}, every label/.style={rectangle,draw=none}]
  \node (1) at (0.0,0) [label={above,minimum height=13}:$ $] {};
  \node (2) at (0.8,0) [label={above,minimum height=13}:$ $,fill=black!50] {};
  \node (3) at (1.6,0) [label={above,minimum height=13}:$ $] {};
  \node (4) at (2.4,0) [label={above,minimum height=13}:$ $,fill=black!50] {};
  \node (5) at (3.2,0) [label={above,minimum height=13}:$ $] {};
  \node (6) at (4.0,0) [label={above,minimum height=13}:$ $,fill=black!50] {};
  \node (7) at (4.8,0) [label={above,minimum height=13}:$ $,fill=black!50] {};
  \node (8) at (5.6,0) [label={above,minimum height=13}:$ $,fill=black!50] {};
  \node (9) at (6.4,0) [label={above,minimum height=13}:$ $] {};
  \node (10) at (7.2,0) [label={above,minimum height=13}:$ $,fill=black!50] {};
\end{scope}
\begin{scope}[every edge/.style={draw=black!60,line width=1.2}]
  \path (1) edge node {} node[inner sep=3, above] {\bf\small 4} (2);
  \path (2) edge node {} (3);
  \path (3) edge node {} (4);
  \path (4) edge node {} (5);
  \path (5) edge node {} (6);
  \path (8) edge node {} (9);
  \path (9) edge node {} (10);
\end{scope}
\begin{scope}[every edge/.style={draw=black, line width=3}]
  \path (6) edge node {} (7);
  \path (7) edge node {} (8);
\end{scope}
\begin{scope}[every node/.style={draw,inner sep=11.5,yshift=-4}, every label/.style={rectangle,draw=none,inner sep=6.0}, every fit/.append style=text badly centered]
\end{scope}
\end{tikzpicture}
 & 1 \\

$\rho_3$ & $=$ & 
\begin{tikzpicture}
\begin{scope}[every node/.style={circle,thick,draw,inner sep=2.5}, every label/.style={rectangle,draw=none}]
  \node (1) at (0.0,0) [label={above,minimum height=13}:$ $] {};
  \node (2) at (0.8,0) [label={above,minimum height=13}:$ $,fill=black!50] {};
  \node (3) at (1.6,0) [label={above,minimum height=13}:$ $] {};
  \node (4) at (2.4,0) [label={above,minimum height=13}:$ $,fill=black!50] {};
  \node (5) at (3.2,0) [label={above,minimum height=13}:$ $] {};
  \node (6) at (4.0,0) [label={above,minimum height=13}:$ $,fill=black!50] {};
  \node (7) at (4.8,0) [label={above,minimum height=13}:$ $,fill=black!50] {};
  \node (8) at (5.6,0) [label={above,minimum height=13}:$ $,fill=black!50] {};
  \node (9) at (6.4,0) [label={above,minimum height=13}:$ $,fill=black!50] {};
  \node (10) at (7.2,0) [label={above,minimum height=13}:$ $,fill=black!50] {};
\end{scope}
\begin{scope}[every edge/.style={draw=black!60,line width=1.2}]
  \path (1) edge node {} node[inner sep=3, above] {\bf\small 4} (2);
  \path (2) edge node {} (3);
  \path (3) edge node {} (4);
  \path (4) edge node {} (5);
  \path (5) edge node {} (6);
\end{scope}
\begin{scope}[every edge/.style={draw=black, line width=3}]
  \path (6) edge node {} (7);
  \path (7) edge node {} (8);
  \path (8) edge node {} (9);
  \path (9) edge node {} (10);
\end{scope}
\begin{scope}[every node/.style={draw,inner sep=11.5,yshift=-4}, every label/.style={rectangle,draw=none,inner sep=6.0}, every fit/.append style=text badly centered]
\end{scope}
\end{tikzpicture}
 & 1 \\

$\sigma_0$ & $=$ & 
\begin{tikzpicture}
\begin{scope}[every node/.style={circle,thick,draw,inner sep=2.5}, every label/.style={rectangle,draw=none}]
  \node (1) at (0.0,0) [label={above,minimum height=13}:$ $] {};
  \node (2) at (0.8,0) [label={above,minimum height=13}:$ $,fill=black!50] {};
  \node (3) at (1.6,0) [label={above,minimum height=13}:$ $] {};
  \node (4) at (2.4,0) [label={above,minimum height=13}:$ $,fill=black!50] {};
  \node (5) at (3.2,0) [label={above,minimum height=13}:$ $] {};
  \node (6) at (4.0,0) [label={above,minimum height=13}:$ $,fill=black!50] {};
  \node (7) at (4.8,0) [label={above,minimum height=13}:$ $] {};
  \node (8) at (5.6,0) [label={above,minimum height=13}:$ $,fill=black!50] {};
  \node (9) at (6.4,0) [label={above,minimum height=13}:$ $,fill=black!50] {};
  \node (10) at (7.2,0) [label={above,minimum height=13}:$ $,fill=black!50] {};
\end{scope}
\begin{scope}[every edge/.style={draw=black!60,line width=1.2}]
  \path (1) edge node {} node[inner sep=3, above] {\bf\small 4} (2);
  \path (2) edge node {} (3);
  \path (3) edge node {} (4);
  \path (4) edge node {} (5);
  \path (5) edge node {} (6);
  \path (6) edge node {} (7);
  \path (7) edge node {} (8);
\end{scope}
\begin{scope}[every edge/.style={draw=black, line width=3}]
  \path (8) edge node {} (9);
  \path (9) edge node {} (10);
\end{scope}
\begin{scope}[every node/.style={draw,inner sep=11.5,yshift=-4}, every label/.style={rectangle,draw=none,inner sep=6.0}, every fit/.append style=text badly centered]
\end{scope}
\end{tikzpicture}
 & 1 \\

\end{tabular}

}

\vskip0.3cm

   \caption{The only alternating path from $\sigma_q'$ to $\sigma_{q-2}$ for $n=10$, $d=4$, $k=5$ and $q=2$.}
  \label{table:Bn-alternating-path}
\end{table}

Having a complete description of a precise matching on $K$, we can now compute the homology $H_*(\GW;R)$ for $\GW$ of type $B_n$.
We recover the result of \cite{de1999arithmetic}.

\begin{theorem}[Homology in case $B_n$ \cite{de1999arithmetic}]
  For an Artin group $\GW$ of type $B_n$, we have
  \[ H_m(\GW;R)_{\varphi_d} \cong
    \begin{cases}
      R / (\varphi_d) & \text{if $d$ is even, $n=k\frac d2$ and $n-k\leq m \leq n-1$}, \\
      0 & \text{otherwise}.
    \end{cases}
  \]
  \label{thm:homology-Bn}
\end{theorem}

\begin{proof}
  For a fixed $\varphi = \varphi_d$, we need to compute the boundary $\delta_{m+1}^\varphi$.
  Assume that $n=k\frac d2$, otherwise there are no critical simplices.
  In top dimension ($m+1=n$) we have $\rk \delta_n^\varphi = 1$, because $[\sigma_k':\sigma_{k-1}'] \neq 0$.
  For $m\leq n-k-1$ the boundary $\delta_{m+1}^\varphi$ vanishes, because there are no critical simplices in dimension $\leq n-k-1$.
  For $m = n-k$ we have $\rk \delta_{m+1}^\varphi = 1$, because $[\sigma_1':\sigma_0'] \neq 0$.
  Finally, for $n-k+1 \leq m \leq n-2$, we have (set $q=m-n+k$):
  \begin{align*}
    \rk \delta_{m+1}^\varphi & = \rk
    \left( \begin{matrix}
      [\sigma_q: \sigma_{q-1}] & [\sigma_{q+1}': \sigma_{q-1}] \\
      [\sigma_q: \sigma_q'] & [\sigma_{q+1}': \sigma_q']
    \end{matrix} \right) \\
    & = \rk \left( \begin{matrix}
      (-1)^{(q-1)\frac d2} & (-1)^{k(\frac d2 - 1) + q+1} \\
      (-1)^{(k-1)(\frac d2 - 1) + q} & (-1)^{q\frac d2}
    \end{matrix} \right) = 1,
  \end{align*}
  because
  \[ \det \left( \begin{matrix}
      (-1)^{(q-1)\frac d2} & (-1)^{k(\frac d2 - 1) + q+1} \\
      (-1)^{(k-1)(\frac d2 - 1) + q} & (-1)^{q\frac d2}
    \end{matrix} \right) = (-1)^{\frac d2} - (-1)^{\frac d2} = 0.
  \]
  To summarize, we have: $\rk \delta_{m+1}^\varphi = 1$ for $n-k \leq m \leq n-1$, and $\rk \delta_{m+1}^\varphi = 0$ otherwise.
  We conclude applying Theorem \ref{thm:homology-artin-groups}.
\end{proof}

\subsection{Case $\tilde A_n$}
\label{sec:precise-matching-tAn}

Consider now the case of affine Artin groups of type $\tilde A_n$ (see Figure \ref{fig:tAn} for a picture of the corresponding Coxeter graph).
Recall that, for affine Artin groups, the simplicial complex $K$ consists of all the simplices $\sigma \subseteq S = \{0,\dots, n\}$ except for the full simplex $\sigma = \{0,\dots, n\}$.
For any $\sigma\in K$, the induced subgraph $\Gamma(\sigma)$ consists only of connected components of type $A_k$. The weight $v_\varphi(\sigma)$ is then computed as in Section \ref{sec:precise-matching-An}.

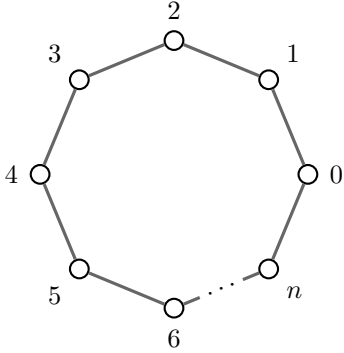
\begin{figure}[htbp]
  \begin{tikzpicture}
\begin{scope}[every node/.style={circle,thick,draw,inner sep=2.5}, every label/.style={rectangle,draw=none}]
  \node (0) at (0.00:1.78) [label={[label distance=1.5]0.00:$0$}] {};
  \node (1) at (45.00:1.78) [label={[label distance=1.5]45.00:$1$}] {};
  \node (2) at (90.00:1.78) [label={[label distance=1.5]90.00:$2$}] {};
  \node (3) at (135.00:1.78) [label={[label distance=1.5]135.00:$3$}] {};
  \node (4) at (180.00:1.78) [label={[label distance=1.5]180.00:$4$}] {};
  \node (5) at (225.00:1.78) [label={[label distance=1.5]225.00:$5$}] {};
  \node (6) at (270.00:1.78) [label={[label distance=1.5]270.00:$6$}] {};
  \node (n) at (315.00:1.78) [label={[label distance=1.5]315.00:$n$}] {};
\end{scope}
\begin{scope}[every edge/.style={draw=black!60,line width=1.2}]
  \path (0) edge node {} (1);
  \path (1) edge node {} (2);
  \path (2) edge node {} (3);
  \path (3) edge node {} (4);
  \path (4) edge node {} (5);
  \path (5) edge node {} (6);
  \path (6) edge node [fill=white, rectangle, inner sep=3.0, rotate=382.50] {$\ldots$} (n);
  \path (n) edge node {} (0);
\end{scope}
\end{tikzpicture}
   \caption{A Coxeter graph of type $\tilde A_n$.} 
  \label{fig:tAn}
\end{figure}

If $\sigma$ is a simplex of $K$ then $\sigma$ misses at least one vertex of $\{0,1, \dots , n\}$. Define $h$ to be the first vertex missing from $\sigma$ reading counterclockwise from $0$. Then $\sigma$ belongs to a unique
\[ K_h = \{ \sigma \in K \mid h\not\in \sigma \; \text{and} \; \{0,\dots, h-1\} \subseteq \sigma \}. \]
In other words, the subsets $K_h$ form a partition of $K$:
\[ K = \bigsqcup_{h=0}^n K_h. \]
Construct a matching on a fixed $K_h$ as follows.
Let $h=qd+r$ with $0\leq r \leq d-1$, and let $m=n-qd$.
Consider the following clockwise relabeling of the vertices: $r-1$ becomes $1$, $r-2$ becomes $2$, \dots, $0$ becomes $r$, $n$ becomes $r+1$, \dots, $h+2$ becomes $m-1$, $h+1$ becomes $m$ (the vertices $r$, $r+1$, \dots, $h$ are forgotten).
This relabeling induces a poset isomorphism
\[ K_h \xrightarrow{\cong} K^A_{m,r}. \]
Then equip $K_h$ with the pull-back of the matching on $K^A_{m,r}$ defined in Section \ref{sec:precise-matching-An} (for $h=n$ the only simplex in $K_h$ is critical).

\begin{lemma}
  The matching defined above is an acyclic weighted matching on $K$.
  \label{lemma:tAn-acyclic-weighted}
\end{lemma}

\begin{proof}
  Consider the map $\eta\colon K \to (\N, \leq)$ which sends $\sigma\in K$ to $\min \{ h\in \N \mid h\not\in\sigma\}$.
  Then $\eta$ is a poset map with fibers $\eta^{-1}(h) = K_h$.
  The matching on each $K_h\cong K^A_{m,r}$ is acyclic by Lemma \ref{lemma:An-acyclic-weighted}, therefore the whole matching on $K$ is acyclic.

  Fix $h\in \{0,\dots, n\}$. The bijection $K_h \xrightarrow{\cong} K^A_{m,r}$ is such that, if $\sigma \mapsto \hat\sigma$, then $v_\varphi(\sigma) = v_\varphi(\hat \sigma) + q$. Indeed, $\Gamma(\sigma)$ is obtained from $\Gamma(\hat\sigma)$ by adding $qd$ vertices to one (possibly empty) connected component, and this increases the weight by $q$.
  Then, since the matching on $K^A_{m,r}$ is weighted, its pull-back also is.
\end{proof}

\begin{table}[htbp]
  {\renewcommand{\arraystretch}{1.4}\setlength\tabcolsep{3.5pt}\begin{tabular}{c|c|c|c|c}
    Case & Simplices & $|\sigma|$ & $v_\varphi(\sigma)$ \\
    \hline \hline
    $n = kd + r$ & $\tau_q$ & $n-2(k-q)+1$ & $q+1$ & $0\leq q\leq k-1$\\
    $ \not\equiv -1 \pmod d$ & $\tau_q'$ & $n-2(k-q)$ & $q$ & $0\leq q \leq k$\\
    \hline
    \multirow{3}{*}{$n = kd - 1$} & $\sigma_{q,r}$ & $n-2(k-q)+2$ & $q+1$ & $0\leq q \leq k-2$, $0\leq r\leq d-2$ \\
    & $\sigma_{q,r}'$ & $n-2(k-q)+1$ & $q$ & $0\leq q \leq k-1$, $0\leq r\leq d-2$ \\
    & $\bar\sigma = 1^n0$ & $n$ & $k$ \\
    \hline
  \end{tabular}
  
  \vskip0.2cm
  
  \begin{tabular}{c}
    Simplices \\
    \hline \hline
    $\tau_q = 1^{qd+r}01^{d-1}0(1^r01^{d-2-r}0)^{k-q-1}$ \\
    $\tau_q' = 1^{qd+r}0(1^r01^{d-2-r}0)^{k-q}$ \\
    \hline
    $\sigma_{q,r} = 1^{qd+r}01^{d-1}0(1^{d-2-r}01^r0)^{k-q-2}1^{d-2-r}0$ \\
    $\sigma_{q,r}' = 1^{qd+r}0(1^{d-2-r}01^r0)^{k-q-1}1^{d-2-r}0$ \\
    \hline
  \end{tabular}
  }
  
\vskip0.3cm
  \caption{Description of the critical simplices for $\tilde A_n$. Below, the binary string notation is used.}
  \label{table:tAn-critical}
\end{table}

\begin{theorem}[Critical simplices in case $\tilde A_n$]
  The critical simplices for the matching on $K$ are those listed in Table \ref{table:tAn-critical}.
  The only non-trivial incidence numbers between critical cells are:
  \begin{align*}
   [\tau_q: \tau_q']^\M & = \pm 1 \quad \text{(for $n=kd+r$ with $0 \leq r \leq d-2$)};\\
   [\sigma_{q,r}: \sigma_{q,r}']^\M & = \pm 1, \\
   [\bar\sigma : \sigma_{k-1,r}']^\M & = \pm 1 \quad \text{(for $n=kd-1$)}.
  \end{align*}
  In particular the matching is precise.
  \label{thm:matching-tAn}
\end{theorem}

\begin{proof}
  \textbf{Part 1: critical simplices.}
  For $n\not\equiv -1 \pmod d$, the matching on $K_h$ has critical simplices only for $h\equiv n \pmod d$. If we write $n=kd+r$ with $0\leq r \leq d-2$, then $K_{qd+r} \cong K^A_{n-qd,\, r}$ has two critical simplices for $0\leq q \leq k-1$ and one critical simplex for $q=k$.
  These simplices are the ones listed in Table \ref{table:tAn-critical}.
  
  Suppose now that $n = kd - 1$. For any $h=qd+r$, with $0\leq r \leq d-2$, the complex $\smash{K_{qd+r} \cong K^A_{n-qd,\, r}}$ has two critical simplices because $n-qd \equiv -1 \pmod d$. Moreover, for $h=n$ the complex $\smash{K_n \cong K^A_{d-1,\,d-1}}$ has one critical simplex.
  In the remaining cases the matching on $K_h$ has no critical simplices.
  Again, the critical simplices are those listed in Table \ref{table:tAn-critical}.
  
  \textbf{Part 2: incidence numbers for $n=kd+r$.}
  We want to find the incidence numbers between critical simplices of consecutive dimensions. We start with the case $n=kd+r$, with $0\leq r \leq d-2$. 
  First, let us look for alternating paths from $\tau_q$ to $\tau_q'$ (for $0\leq q \leq k-1$). Set $h=qd+r$. Suppose we have one such path:
  \[ \tau_q \rhd \zeta_1 \lhd \rho_1\rhd \zeta_2\lhd \rho_2\rhd  \dots \rhd \zeta_m\lhd \rho_m \rhd \tau_q'. \]
  If at some point a vertex $v \in \{0, ..., h-1\}$ is removed, then the path falls into some $K_{h'}$ with $h' < h$ and can never return in $K_h$. Therefore the path must be entirely contained in $\smash{K_h \cong K^A_{n-qd,\,r}}$, and by Theorem \ref{thm:matching-An} it must be the trivial path $\tau_q \rhd \tau_q'$. Thus $[\tau_q: \tau_q'] = \pm 1$.
  
  The other pairs of critical simplices in consecutive dimensions are $(\tau_{q+1}', \tau_q)$ for $0\leq q \leq k-1$.
  There is a trivial path $\tau_{q+1}' \rhd \tau_q$ which consists in removing the vertex $h=qd+r$, and contributes to the incidence number by
  \[ [ \tau_{q+1}' : \tau_q ] = (-1)^h. \]
  Suppose we have some other (non-trivial) alternating path:
  \begin{equation}
    \tau_{q+1}' \rhd \zeta_1 \lhd \rho_1\rhd \zeta_2\lhd \rho_2\rhd  \dots \rhd \zeta_m\lhd \rho_m \rhd \tau_q, \quad m\geq 1.
    \label{eq:alternating-path-tAn}
  \end{equation}
  Let $\rho_m = \tau_q \cup \{v\}$. If $v=h$ then $\rho_m = \tau_{q+1}'$, which is excluded. Then $v$ must be one of the other vertices that do not belong to $\tau_q$. They are of the form $v=n-sd$ (for $0\leq s \leq k-q-1$) or $v=n-(d-1-r)-sd$ (for $0\leq s \leq k-q-2$).
  If $v=n-(d-1-r)-sd$, then $\rho_m$ is matched with $\rho_m \cup \{ n-(s+1)d \}$. This is impossible because $\rho_m$ must be matched with some $\zeta_m \lhd \rho_m$.
  Similarly, if $v=n-sd$ with $0\leq s \leq k-q-2$, then $\rho_m$ is matched with $\rho_m \cup \{ n-(d-1-r)-sd \}$ and not with some $\zeta_m \lhd \rho_m$.
  The only remaining possibility is $v=n-sd$ with $s=k-q-1$, i.e.\ $v=n-(k-q-1)d = (q+1)d+r$.
  In this case $\rho_m$ is matched with $\zeta_m = \rho_m \setminus \{ qd+2r+1 \}$.
  Going on with the same argument, we have exactly one way to continue the alternating path (from right to left in \eqref{eq:alternating-path-tAn}), and we eventually end up with $\tau_{q+1}'$.
  From left to right, the obtained alternating path is the following:
  \begin{align*}
    \zeta_1 &= \tau_{q+1}' \setminus\{ qd+2r+1 \}, \\
    \rho_1 &= \zeta_1 \cup \{ n \}, \\
    \zeta_2 &= \rho_1 \setminus\{ qd+r \}, \\
    \rho_2 &= \zeta_2 \cup \{ n- (d-1-r) \}, \\
    \zeta_3 &= \rho_2 \setminus \{ n \}, \\
    \rho_3 &= \zeta_3 \cup \{ n-d \}, \\
    \zeta_4 &= \rho_3 \setminus \{ n - (d-1-r) \}, \\
    &\vdots \\
    \zeta_m &= \rho_{m-1} \setminus \{ (q+1)d + 2r + 1 \}, \\
    \rho_m &= \zeta_m \cup \{ qd+2r+1 \}, \\
    \tau_q &= \rho_m \setminus \{ (q+1)d + r \}.
  \end{align*}
  The length of the path is $m=2(k-q)$.
  Apart from $qd+r$, which is the vertex in $\tau_{q+1}' \setminus \tau_q$, the other vertices are added and removed exactly once during the path. If a certain $v$ is added in some $\rho_i$ and removed in some $\zeta_j$ (the removal might possibly come before the addition), then
  \[ [\rho_i:\zeta_i][\rho_{j-1}: \zeta_j] =
    \begin{cases}
      1 & \text{if $v=n$}, \\
      -1 & \text{otherwise}.
    \end{cases}
  \]
  This is true because, except for $v=n$, between the addition and the removal of $v$ exactly one vertex $u$ with $u<v$ has been added/removed.
  Namely, between the addition and the removal of a vertex $v=jd+2r+1$ (for $q \leq j \leq k-1$) the vertex $u=jd+r$ is added/removed, and between the addition of the removal of a vertex $v=jd+r$ (for $q+1 \leq j \leq k-1$) the vertex $u=(j-1)d+2r+1$ is added.
Therefore the alternating path \eqref{eq:alternating-path-tAn} contributes to the incidence number by
  \[ (-1)^m \cdot (-1)^{m-1} \cdot [\zeta_2:\rho_1] = (-1)^{qd+r+1} = (-1)^{h+1}. \]
  Finally, the incidence number is given by
  \[ [ \tau_{q+1}' : \tau_q ]^\M = (-1)^h + (-1)^{h+1} = 0. \]
  
  \textbf{Part 3: incidence numbers for $n=kd-1$.}
Consider a generic alternating path starting from $\bar\sigma = \{0,\dots, n-1\}$:
  \[ \bar\sigma \rhd \zeta_1 \lhd \rho_1\rhd \zeta_2\lhd \rho_2\rhd  \dots \rhd \zeta_m\lhd \rho_m \rhd \zeta_{m+1}. \]
  Let $\bar\sigma = \zeta_1 \cup \{v \}$.
  If $n-d+1 \leq v \leq n-1$ then $\zeta_1 = \sigma_{k-1,r}'$ for $r=n-1-v$. Therefore we have trivial alternating paths from $\bar\sigma$ to any of the $\sigma_{k-1,r}'$.
  If $v \leq n-d$ then $\zeta_1 \in K_h$ with $h \leq n-d = (k-1)d - 1$.
  None of these $K_h$'s contains critical simplices $\zeta_{m+1}$ with $|\zeta_{m+1}| = n-1$, and the alternating path cannot return in any $K_{h'}$ with $h' \geq (k-1)d$.
  Thus there are no other alternating paths from $\bar\sigma$ to critical simplices of $K$.
  Then the non-trivial incidence numbers involving $\bar\sigma$ are:
  \[ [\bar\sigma : \sigma_{k-1,r}']^\M  = \pm 1. \]
  
  Consider now a generic alternating path from $\sigma_{q,r_1}$ to $\sigma_{q,r_2}'$:
  \[ \sigma_{q,r_1} \rhd \zeta_1 \lhd \rho_1\rhd \zeta_2\lhd \rho_2\rhd  \dots \rhd \zeta_m\lhd \rho_m \rhd \sigma_{q,r_2}'. \]
  Let $\rho_m = \sigma_{q,r_2}' \cup \{v\}$.
  Adding $v$ to $\sigma_{q,r_2}'$ causes the creation of a connected component $\Gamma_i$ with $|\Gamma_i|\equiv -1 \pmod d$.
  This means that $\rho_m$ is matched with a simplex of higher dimension, or is not matched at all (this happens for $v=(q+1)d-1$).
  Therefore the alternating path must be trivial, and it occurs only for $r_1 = r_2$.
  Then the non-trivial incidence numbers of the form $[\sigma_{q,r_1} : \sigma_{q,r_2}']^\M$ are:
  \[ [\sigma_{q,r} : \sigma_{q,r}']^\M = \pm 1. \]
  
  Finally consider a generic alternating paths from $\sigma_{q+1,r_1}'$ to $\sigma_{q,r_2}$:
  \begin{equation}
    \sigma_{q+1,r_1}' \rhd \zeta_1 \lhd \rho_1\rhd \zeta_2\lhd \rho_2\rhd  \dots \rhd \zeta_m\lhd \rho_m \rhd \sigma_{q,r_2}.
    \label{eq:alternating-path-tAn2}
  \end{equation}
  As before, we work backwards.
  Let $\rho_m = \sigma_{q,r_2} \cup \{v\}$.
  Apart from the choices $v=qd+r$ and $v=(q+1)d+r$, in all other cases $\rho_m$ has a connected component of size $\equiv -1 \pmod d$ and this prevents the continuation of the alternating path.
  For $v=qd+r$ we obtain $\rho_m = \sigma_{q+1,r_2}'$, so we have a trivial alternating path from $\sigma_{q+1,r_2}'$ to $\sigma_{q,r_2}$.
  For $v=(q+1)d+r$, iterating the same argument, we have exactly one way to continue the alternating path (from right to left in \eqref{eq:alternating-path-tAn2}) and we end up with $\sigma_{q+1,r_2}'$. From left to right, the path is as follows (set $r=r_2$):
  \begin{align*}
    \zeta_1 &= \sigma_{q+1,r}' \setminus\{ (q+1)d - 1 \}, \\
    \rho_1 &= \zeta_1 \cup \{ n \}, \\
    \zeta_2 &= \rho_1 \setminus\{ qd+r \}, \\
    \rho_2 &= \zeta_2 \cup \{ n-(d-1-r) \}, \\
    \zeta_3 &= \rho_2 \setminus \{ n \}, \\
    \rho_3 &= \zeta_3 \cup \{ n-d \}, \\
    \zeta_4 &= \rho_3 \setminus \{ n-(d-1-r) \}, \\
    &\vdots \\
    \zeta_m &= \rho_{m-1} \setminus \{ (q+2)d - 1 \}, \\
    \rho_m &= \zeta_m \cup \{ (q+1)d - 1 \}, \\
    \sigma_{q,r} &= \rho_m \setminus \{ (q+1)d+r \}.
  \end{align*}
  The length of the path is $m = 2(k-q)-1$.
  As happened in the case $n = kd+r$, the contribution of this alternating path to the incidence number is given by
  \[ (-1)^m \cdot (-1)^{m-1}\cdot [\zeta_2:\rho_1] = (-1)^{qd+r+1}. \]
  The contribution of the trivial path $\sigma_{q+1,r}' \rhd \sigma_{q,r}$ is given by $(-1)^{qd+r}$.
  Therefore
  \[ [\sigma_{q+1,r}' : \sigma_{q,r}]^\M = (-1)^{qd+r} + (-1)^{qd+r+1} = 0. \qedhere \]
\end{proof}

We are now able to recover the result of \cite{callegaro2008cohomology2} about the homology $H_*(\GW;R)$ when $\GW$ is an Artin group of type $\smash{\tilde A_n}$.

\begin{theorem}[Homology in case $\tilde A_n$ \cite{callegaro2008cohomology2}]
  For an Artin group $\GW$ of type $\tilde A_n$, we have
  \[ H_m(\GW;R)_{\varphi_d} \cong
    \begin{cases}
      \left( R/(\varphi_d) \right)^{\oplus d-1} & \text{if $n=kd-1$ and $m=n-2i+1$ ($1\leq i \leq k$)}, \\
      R/(\varphi_d) & \text{if $n=kd+r$ and $m=n-2i$ ($1\leq i \leq k$)}, \\
      0 & \text{otherwise},
    \end{cases}
  \]
  where $0\leq r\leq d-2$ in the second case.
  \label{thm:homology-tAn}
\end{theorem}

\begin{proof}
  We apply Theorems \ref{thm:homology-artin-groups} and \ref{thm:matching-tAn}.
  For $n=kd+r$ ($0\leq r\leq d-2$), the boundary map $\delta_{m+1}^\varphi$ has rank $1$ when $m+1=n-2(k-q)+1$ (for $0\leq q \leq k-1$, due to $\tau_q$); it has rank $0$ otherwise.
  For $n=kd-1$, the boundary map $\delta_{m+1}^\varphi$ has rank $d-1$ if $m+1 = n-2(k-q)+2$ (for $0\leq q \leq k-2$, due to the simplices $\sigma_{q,r}$) or if $m+1 = n$ (due to $\bar\sigma$); it has rank $0$ otherwise.
\end{proof}

\subsection{Case $\tilde C_n$}

The last case we consider is that of Artin groups of type $\smash{\tilde C_n}$. The corresponding Coxeter graph $\Gamma$ is shown in Figure \ref{fig:tCn}.
As in the $\smash{\tilde A_n}$ case, the simplicial complex $K$ consists of all the simplices $\sigma \subseteq S = \{0,\dots, n\}$ except for the full simplex $\sigma = \{0,\dots, n\}$.
For any $\sigma \in K$, the subgraph $\Gamma(\sigma)$ of $\Gamma$ splits as a union of connected components of type $B_k$ (those containing the first or the last vertex) and of type $A_k$ (the remaining ones).

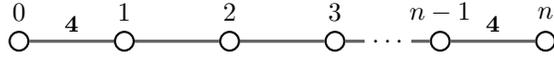
\begin{figure}[htbp]
  \begin{tikzpicture}
\begin{scope}[every node/.style={circle,thick,draw,inner sep=2.5}, every label/.style={rectangle,draw=none}]
  \node (0) at (0.0,0) [label={above,minimum height=13}:$0$] {};
  \node (1) at (1.4,0) [label={above,minimum height=13}:$1$] {};
  \node (2) at (2.8,0) [label={above,minimum height=13}:$2$] {};
  \node (3) at (4.2,0) [label={above,minimum height=13}:$3$] {};
  \node (n-1) at (5.6,0) [label={above,minimum height=13}:$n-1$] {};
  \node (n) at (7.0,0) [label={above,minimum height=13}:$n$] {};
\end{scope}
\begin{scope}[every edge/.style={draw=black!60,line width=1.2}]
  \path (0) edge node {} node[inner sep=3, above] {\bf\small 4} (1);
  \path (1) edge node {} (2);
  \path (2) edge node {} (3);
  \path (3) edge node [fill=white, rectangle, inner sep=3.0] {$\ldots$} (n-1);
  \path (n-1) edge node {} node[inner sep=3, above] {\bf\small 4} (n);
\end{scope}
\begin{scope}[every edge/.style={draw=black, line width=3}]
\end{scope}
\begin{scope}[every node/.style={draw,inner sep=11.5,yshift=-4}, every label/.style={rectangle,draw=none,inner sep=6.0}, every fit/.append style=text badly centered]
\end{scope}
\end{tikzpicture}
   \caption{A Coxeter graph of type $\tilde C_n$.} 
  \label{fig:tCn}
\end{figure}

For each $m$ with $1 \leq m \leq n$, let $K^B_m$ be the full simplicial complex on $\{1,\dots,m\}$, endowed with the weight function of $B_m$ (see Section \ref{sec:precise-matching-Bn}).
We are going to construct a precise matching for the case $\tilde C_n$ using the precise matching on $K^B_m$.

For $h \in \{0,\dots, n\}$ set
\[ K_h = \{ \sigma \in K \mid h\not\in \sigma \; \text{and} \; \{0,\dots, h-1\} \subseteq \sigma \}. \]
As in the $\tilde A_n$ case, since every simplex $\sigma\in K$ misses at least one vertex, the subsets $K_h$ form a partition of $K$:
\[ K = \bigsqcup_{h=0}^n K_h. \]
Construct a matching on a fixed $K_h$ as follows.
Ignore the vertices $0$, \dots, $h$ and relabel the remaining ones from right to left: $n$ becomes $1$, $n-1$ becomes $2$, \dots, $h+1$ becomes $n-h$.
This induces a poset isomorphism
\[ K_h \xrightarrow{\cong} K^B_{n-h}. \]
Then equip $K_h$ with the pull-back of the matching on $K^B_{n-h}$ constructed in Section \ref{sec:precise-matching-Bn} (for $h=n$ the single simplex in $K_h$ is critical).

\begin{remark}
  For $d$ odd we are simply matching $\sigma$ with $\sigma \xor n$ for all $\sigma$ except $\sigma = \{0,\dots,n-1\}$.
\end{remark}

\begin{lemma}
  The matching defined above is an acyclic weighted matching on $K$.
\end{lemma}

\begin{proof}
  As in the proof of Lemma \ref{lemma:tAn-acyclic-weighted}, the subsets $K_h$ are the fibers of a poset map and the matching is acyclic on each $K_h \cong K^B_{n-h}$.

  If a simplex $\sigma\in K_h$ is sent to $\hat\sigma$ by the isomorphism $K_h \xrightarrow{\cong} K^B_{n-h}$, then $v_\varphi(\sigma) = v_\varphi(\hat\sigma) + \omega_\varphi(B_h)$.
  Since the matching on $\smash{K^B_{n-h}}$ is weighted, the matching on $K_h$ also is.
\end{proof}

\begin{table}[phtb]
  {\renewcommand{\arraystretch}{1.4}\begin{tabular}{c|c|c|c|c}
    Case & Simplices & $|\sigma|$ & $v_\varphi(\sigma)$ \\
    \hline \hline
    $d$ odd & $\bar\sigma = 1^n0$ & $n$ & $1$ or $0$ \\
    \hline
    \multirow{2}{*}{$d$ even} & $\sigma_{q_1,q_2}$ & $n-k+q_1+q_2+1$ & $q_1+q_2+1$ & $0\leq q_1+q_2\leq k-2$\\
    & $\sigma_{q_1,q_2}'$ & $n-k+q_1+q_2$ & $q_1+q_2$ & $0\leq q_1+q_2\leq k$\\
    \hline
  \end{tabular}
  
  \vskip0.2cm
  
  \begin{tabular}{c}
    Simplices \\
    \hline \hline
    $\sigma_{q_1,q_2} = 1^{q_1\frac d2 + r}01^{d-1}0(1^{\frac d2-1}0)^{k-q_1-q_2-2}1^{q_2\frac d2}$ \\
    $\sigma_{q_1,q_2}' = 1^{q_1\frac d2 + r}0(1^{\frac d2-1}0)^{k-q_1-q_2}1^{q_2\frac d2}$ \\
    \hline
  \end{tabular}
  }
  
  \vskip0.3cm
  \caption{Description of the critical simplices for $\tilde C_n$. When $d$ is even, set $n=k\frac d2 + r$.}
  \label{table:tCn-critical}
\end{table}

\begin{figure}[phtb]
  \begin{tikzpicture}
\begin{scope}[every node/.style={circle,thick,draw,inner sep=2.5}, every label/.style={rectangle,draw=none}]
  \node (0) at (0.0,0) [label={above,minimum height=13}:$ $,fill=black!50] {};
  \node (1) at (0.8,0) [label={above,minimum height=13}:$ $,fill=black!50] {};
  \node (2) at (1.6,0) [label={above,minimum height=13}:$ $,fill=black!50] {};
  \node (3) at (2.5,0) [label={above,minimum height=13}:$ $,fill=black!50] {};
  \node (4) at (3.3,0) [label={above,minimum height=13}:$ $] {};
  \node (5) at (4.1,0) [label={above,minimum height=13}:$ $,fill=black!50] {};
  \node (6) at (4.9,0) [label={above,minimum height=13}:$ $,fill=black!50] [label={below}:$d-1$\;vertices] {};
  \node (7) at (5.7,0) [label={above,minimum height=13}:$ $,fill=black!50] {};
  \node (8) at (6.6,0) [label={above,minimum height=13}:$ $] {};
  \node (9) at (7.4,0) [label={above,minimum height=13}:$ $,fill=black!50] {};
  \node (10) at (8.2,0) [label={above,minimum height=13}:$ $,fill=black!50] [label={below}:$\frac{d}{2}-1$\;vertices] {};
  \node (11) at (9.0,0) [label={above,minimum height=13}:$ $,fill=black!50] {};
  \node (12) at (9.8,0) [label={above,minimum height=13}:$ $] {};
  \node (13) at (10.7,0) [label={above,minimum height=13}:$ $,fill=black!50] {};
  \node (14) at (11.5,0) [label={above,minimum height=13}:$ $,fill=black!50] [label={below}:$q_2\frac{d}{2}$\;vertices] {};
  \node (15) at (12.3,0) [label={above,minimum height=13}:$ $,fill=black!50] {};
\end{scope}
\begin{scope}[every edge/.style={draw=black!60,line width=1.2}]
  \path (3) edge node {} (4);
  \path (4) edge node [fill=white, rectangle, inner sep=1.0] {$\ldots$} (5);
  \path (7) edge node {} (8);
  \path (8) edge node [fill=white, rectangle, inner sep=1.0] {$\ldots$} (9);
  \path (11) edge node {} (12);
  \path (12) edge node [fill=white, rectangle, inner sep=1.0] {$\ldots$} (13);
\end{scope}
\begin{scope}[every edge/.style={draw=black, line width=3}]
  \path (0) edge node {} node[inner sep=3, above] {\bf\small 4} (1);
  \path (1) edge node {} node [below=2] {$q_1\frac{d}{2}+r$\;vertices} (2);
  \path (2) edge node [fill=white, rectangle, inner sep=1.0, minimum height = 0.5cm] {$\ldots$} (3);
  \path (5) edge node {} (6);
  \path (6) edge node [fill=white, rectangle, inner sep=1.0, minimum height = 0.5cm] {$\ldots$} (7);
  \path (9) edge node {} (10);
  \path (10) edge node [fill=white, rectangle, inner sep=1.0, minimum height = 0.5cm] {$\ldots$} (11);
  \path (13) edge node [fill=white, rectangle, inner sep=1.0, minimum height = 0.5cm] {$\ldots$} (14);
  \path (14) edge node {} node[inner sep=3, above] {\bf\small 4} (15);
\end{scope}
\begin{scope}[every node/.style={draw,inner sep=11.5,yshift=-4}, every label/.style={rectangle,draw=none,inner sep=6.0}, every fit/.append style=text badly centered]
  \node [fit=(9) (12), label={above}:$k-q_1-q_2-2$\;times] {};
\end{scope}
\end{tikzpicture}
   \begin{tikzpicture}
\begin{scope}[every node/.style={circle,thick,draw,inner sep=2.5}, every label/.style={rectangle,draw=none}]
  \node (0) at (0.0,0) [label={above,minimum height=13}:$ $,fill=black!50] {};
  \node (1) at (0.9,0) [label={above,minimum height=13}:$ $,fill=black!50] {};
  \node (2) at (1.8,0) [label={above,minimum height=13}:$ $,fill=black!50] {};
  \node (3) at (2.7,0) [label={above,minimum height=13}:$ $,fill=black!50] {};
  \node (4) at (3.6,0) [label={above,minimum height=13}:$ $] {};
  \node (5) at (4.5,0) [label={above,minimum height=13}:$ $,fill=black!50] {};
  \node (6) at (5.4,0) [label={above,minimum height=13}:$ $,fill=black!50] [label={below}:$\frac{d}{2}-1$\;vertices] {};
  \node (7) at (6.3,0) [label={above,minimum height=13}:$ $,fill=black!50] {};
  \node (8) at (7.2,0) [label={above,minimum height=13}:$ $] {};
  \node (9) at (8.1,0) [label={above,minimum height=13}:$ $,fill=black!50] {};
  \node (10) at (9.0,0) [label={above,minimum height=13}:$ $,fill=black!50] {};
  \node (11) at (9.9,0) [label={above,minimum height=13}:$ $,fill=black!50] {};
  \node (12) at (10.8,0) [label={above,minimum height=13}:$ $,fill=black!50] {};
\end{scope}
\begin{scope}[every edge/.style={draw=black!60,line width=1.2}]
  \path (3) edge node {} (4);
  \path (4) edge node [fill=white, rectangle, inner sep=1.0] {$\ldots$} (5);
  \path (7) edge node {} (8);
  \path (8) edge node [fill=white, rectangle, inner sep=1.0] {$\ldots$} (9);
\end{scope}
\begin{scope}[every edge/.style={draw=black, line width=3}]
  \path (0) edge node {} node[inner sep=3, above] {\bf\small 4} (1);
  \path (1) edge node {} node [below=2] {$q_1\frac{d}{2}+r$\;vertices} (2);
  \path (2) edge node [fill=white, rectangle, inner sep=1.0, minimum height = 0.5cm] {$\ldots$} (3);
  \path (5) edge node {} (6);
  \path (6) edge node [fill=white, rectangle, inner sep=1.0, minimum height = 0.5cm] {$\ldots$} (7);
  \path (9) edge node {} (10);
  \path (10) edge node [fill=white, rectangle, inner sep=1.0, minimum height = 0.5cm] {$\ldots$} node [below=2] {$q_2\frac{d}{2}$\;vertices} (11);
  \path (11) edge node {} node[inner sep=3, above] {\bf\small 4} (12);
\end{scope}
\begin{scope}[every node/.style={draw,inner sep=11.5,yshift=-4}, every label/.style={rectangle,draw=none,inner sep=6.0}, every fit/.append style=text badly centered]
  \node [fit=(5) (8), label={above}:$k-q_1-q_2$\;times] {};
\end{scope}
\end{tikzpicture}
   \caption{Critical simplices for $\tilde C_n$, with $d$ even and $n=k\frac d2+r$. The diagram for the simplex $\sigma_{q_1,q_2}$ is at the top and the diagram for the simplex $\sigma_{q_1,q_2}'$ is below it.}
  \label{fig:tCn-critical}
\end{figure}
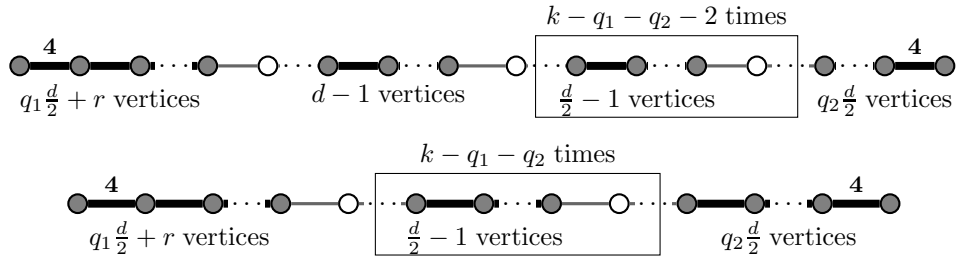

\begin{table}[phtb]
  {\renewcommand{\arraystretch}{1.4}\begin{tabular}{c|c}
    \multicolumn{2}{c}{Simplices} \\
    \hline \hline
    $\sigma_{l-1,0}$ & $\sigma_{l,0}'$ \\
    $\sigma_{l-2,1}$ & $\sigma_{l-1,1}'$ \\
    \vdots & \vdots \\
    $\sigma_{1,l-2}$ & $\sigma_{2,l-2}'$ \\
    $\sigma_{0,l-1}$ & $\sigma_{1,l-1}'$ \\
    & $\sigma_{0,l}'$ \\
    \hline
  \end{tabular}
  }
  \vskip0.3cm
  \caption{Critical simplices for $\tilde C_n$ in dimension $m=n-k+l$ ($0\leq l\leq k$), where $d$ is even and $n=k\frac d2+r$.
  For $l=k$ only the second column occurs.}
  \label{table:tCn-critical-by-dimension}
\end{table}

\begin{theorem}[Critical simplices in case $\tilde C_n$]
  \begin{samepage}
  The critical simplices for the matching on $K$ are those listed in Table \ref{table:tCn-critical}.
  In particular the matching is precise.
  In addition the only non-trivial incidence numbers between critical simplices in the Morse complex are as follows (for $d$ even and $n=k\frac d2 + r$):
  \begin{align*}
    [\sigma_{q_1,q_2}: \sigma_{q_1,\,q_2-1}]^\M & = (-1)^\alpha, \\
    [\sigma_{q_1,q_2}: \sigma_{q_1-1,\,q_2}]^\M & = (-1)^{\beta + 1}, \\
    [\sigma_{q_1,q_2}: \sigma_{q_1,q_2}']^\M & = (-1)^{\beta + \frac d2 + 1}, \\
    [\sigma_{q_1,q_2}': \sigma_{q_1,\,q_2-2}]^\M &= (-1)^{\beta + 1}, \\
    [\sigma_{q_1,q_2}': \sigma_{q_1-2,\, q_2}]^\M &= (-1)^{\beta}, \\
    [\sigma_{q_1,q_2}': \sigma_{q_1,\,q_2-1}']^\M & = (-1)^{\alpha+1}, \\
    [\sigma_{q_1,q_2}': \sigma_{q_1-1,\,q_2}']^\M &= (-1)^{\beta + \frac d2},
  \end{align*}
  where $\alpha = (k-q_2)\big(\frac d2 - 1\big) + q_1 + r + \frac d2$ and $\beta = q_1\frac d2 + r$.
  \end{samepage}
  \label{thm:matching-tCn}
\end{theorem}

\begin{proof}
  As we have already said, for $d$ odd there is exactly one critical simplex.
  Suppose from now on that $d$ is even.
  Let $\smash{n=k\frac d2 + r}$, with $0\leq r \leq \frac d2 - 1$.
  
  \textbf{Part 1: the critical simplices.}
  In $K_h$ there are critical simplices if and only if $\smash{n-h \equiv 0 \pmod[\big]{\frac d2}}$, i.e.\ when $\smash{h=q_1\frac d2+r}$ for some $q_1$ (with $0\leq q_1 \leq k$).
  By Theorem \ref{thm:matching-Bn} there are two families of critical simplices: $\sigma_{q_1,q_2}$ (for $0\leq q_1+q_2 \leq k-2$) and $\sigma_{q_1,q_2}'$ (for $0\leq q_1+q_2\leq k$), as shown in Table \ref{table:tCn-critical} and Figure \ref{fig:tCn-critical}.
  In a fixed dimension $m=n-k+l$ ($0\leq l \leq k$) there are $2l+1$ critical simplices if $l\leq k-1$ and $l$ critical simplices if $l=k$. See Table \ref{table:tCn-critical-by-dimension}.

\textbf{Part 2: paths ending in $\sigma_{q_1,q_2}$.}
  Consider a generic alternating path starting from any critical cell $\rho_0$ and ending in a critical cell of the form $\sigma_{q_1,q_2}$:
  \[ \rho_0 \rhd \tau_1 \lhd \rho_1\rhd \tau_2\lhd \rho_2\rhd  \dots \rhd \tau_m\lhd \rho_m \rhd \sigma_{q_1,q_2}. \]
  Let $\rho_m = \sigma_{q_1,q_2} \cup \{v\}$.
  If $\smash{v=q_1\frac d2 + r}$ then $\rho_m = \sigma_{q_1+2,\, q_2}'$ and the alternating path stops.
  If $\smash{v=n - q_2\frac d2}$ then the alternating path stops at $\rho_m = \sigma_{q_1,\,q_2+1}$.
  Suppose $q_1+q_2 \leq k-3$, otherwise there are no more cases.
  If $q_1 \frac d2 + r + d < v < n - q_2\frac d2$, then $\Gamma(\rho_m)$ has at least one connected component of size $d-1$ and therefore $\rho_m$ is matched with a simplex of higher dimension; thus the path stops without having reached a critical simplex.
  If $v = q_1 \frac d2 + r + d$ then $\rho_m$ is matched with
  \[ \tau_m = \rho_m \setminus \textstyle\left\{ q_1 \frac d2 + r + \frac d2 \right\} = 1^{q_1\frac d2 + r}01^{\frac d2 - 1}01^{d-1}0(1^{\frac d2 -1}0)^{k-q_1-q_2-3}01^{q_2\frac d2} \]
  in the binary string notation.
  From here the path can continue in many ways. Let $\rho_{m-1}  = \tau_m \cup \{w\}$.
  \begin{itemize}
    \item If $w = q_1\frac d2 + r$, we end up with $\rho_{m-1} = \sigma_{q_1+1, \, q_2}$.
    \item If $w = q_1\frac d2 + r + \frac d2$ we would be going back to $\rho_m$.
    \item If $w > q_1\frac d2 + r + \frac 32 d$, then $\Gamma(\rho_{m-1})$ has at least one connected component of size $d-1$ and therefore $\rho_{m-1}$ is matched with a simplex of higher dimension.
\item If $w = q_1\frac d2 + r + \frac 32 d$, then $\rho_{m-1}$ is matched with
    \begin{align*}
      \tau_{m-1} & = \rho_{m-1} \setminus \textstyle\left\{ q_1 \frac d2 + r + d \right\} \\
      & = 1^{q_1\frac d2 + r}01^{\frac d2 - 1}01^{\frac d2 - 1}01^{d-1}0(1^{\frac d2 -1}0)^{k-q_1-q_2-4}01^{q_2\frac d2}.
    \end{align*}
    By induction, repeating the same argument as above, this path can be continued in exactly one way and it eventually arrives at the critical simplex $\sigma_{q_1,\,q_2+2}'$.
    The path has length $m=k-q_1-q_2-2$ and is as follows.
    \begin{align*}
      \tau_1 & = \sigma_{q_1,\,q_2+2}' \setminus \textstyle \left\{ n - q_2\frac d2 \right\}, \\
      \rho_1 & = \tau_1 \cup \textstyle \left\{ n - q_2\frac d2 - d \right\}, \\
      \tau_2 & = \rho_1 \setminus \textstyle \left\{ n - q_2\frac d2 - \frac d2 \right\}, \\
      \rho_2 & = \tau_2 \cup \textstyle \left\{ n - q_2\frac d2 - \frac 32 d \right\}, \\
      \tau_3 & = \rho_2 \setminus \textstyle \left\{ n - q_2\frac d2 - d \right\}, \\
      & \vdots \\
      \rho_{m-1} & = \tau_{m-1} \cup \textstyle \left\{ q_1 \frac d2 + r + d \right\}, \\
      \tau_m & = \rho_{m-1} \setminus \textstyle \left\{ q_1\frac d2 + r + \frac 32 d \right\}, \\
      \rho_m & = \tau_m \cup \textstyle \left\{ q_1 \frac d2 + r + \frac d2 \right\}, \\
      \sigma_{q_1,q_2} & = \rho_m \setminus \textstyle \left\{ q_1\frac d2 + r + d \right\}.
    \end{align*}
  \end{itemize}
  
  \textbf{Part 3: paths ending in $\sigma_{q_1,q_2}'$.}
  Consider now a generic alternating path starting from any critical cell $\rho_0$ and ending in a critical cell of the form $\sigma_{q_1,q_2}'$:
  \begin{equation}
    \rho_0 \rhd \tau_1 \lhd \rho_1\rhd \tau_2\lhd \rho_2\rhd  \dots \rhd \tau_m\lhd \rho_m \rhd \sigma_{q_1,q_2}'.
    \label{eq:alternating-path-tCn}
  \end{equation}
  As usual, let $\rho_m = \sigma_{q_1,q_2}' \cup \{v\}$.
  For the same reasons as above, there are only three possibilities: $v = q_1\frac d2 + r$, $v = q_1\frac d2 + r + \frac d2$, $v = n - q_2\frac d2$; in all the other cases, $\rho_m$ is matched with a simplex of higher dimension.
  If $v = q_1\frac d2 + r$ then the path ends (to the left in \eqref{eq:alternating-path-tCn}) at $\rho_m = \sigma_{q_1+1,\,q_2}'$.
  If $v = q_1\frac d2 + r + \frac d2$ then the path ends at $\sigma_{q_1,q_2}$.
  Finally, if $v = n - q_2\frac d2$ then the path ends at $\sigma_{q_1, \,q_2+1}'$.
  
  \textbf{Part 4: incidence numbers.}
  We have seven families of incidence numbers to compute, each coming from one of the alternating paths we have found.
  \begin{itemize}
    \item From $\sigma_{q_1+2,\,q_2}'$ to $\sigma_{q_1,q_2}$.
    The alternating path is trivial and consists in removing the vertex $v=q_1\frac d2 + r$, so
    \begin{align*}
      & [\sigma_{q_1+2,\,q_2}' : \sigma_{q_1,q_2}]^\M = (-1)^{|\{w \in \sigma_{q_1,q_2} \mid w < v\}|} = (-1)^{q_1\frac d2 + r}, \\
      \Longrightarrow \quad & [\sigma_{q_1,q_2}' : \sigma_{q_1-2,\,q_2}]^\M = (-1)^{q_1\frac d2 + r} = (-1)^{\beta}.
    \end{align*}
    
    \item From $\sigma_{q_1,\,q_2+1}$ to $\sigma_{q_1,q_2}$.
    Again the path is trivial, and it consists in removing $v = n- q_2\frac d2$. Therefore
    \begin{align*}
      [\sigma_{q_1,\,q_2+1} : \sigma_{q_1,q_2}]^\M & = (-1)^{|\{w \in \sigma_{q_1,q_2} \mid w < v\}|} \\
      & = (-1)^{n - q_2\frac d2 - (k-q_1-q_2-1)} \\
      & = (-1)^{k \frac d2 + r - q_2\frac d2 - k + q_1 + q_2 + 1}\\
      & = (-1)^{(k-q_2)\big(\frac d2 - 1\big) + r + q_1 + 1}, \\
      \Longrightarrow \quad [\sigma_{q_1,q_2} : \sigma_{q_1,\,q_2-1}]^\M & = (-1)^{(k-q_2+1)\big(\frac d2 - 1 \big) + r + q_1 + 1}  \\
      & = (-1)^{\alpha}.
    \end{align*}
    
    \item From $\sigma_{q_1+1,\, q_2}$ to $\sigma_{q_1,q_2}$.
    The path consists in removing $q_1\frac d2 + r$, adding $q_1\frac d2 + r + \frac d2$, and removing $q_1\frac d2 + r + d$.
Therefore
    \begin{align*}
      [\sigma_{q_1+1,\,q_2} : \sigma_{q_1,q_2}]^\M & = (-1)(-1)^{q_1\frac d2 + r}(-1)^{q_1\frac d2 + r + \frac d2 - 1}(-1)^{q_1\frac d2 + r + d-1} \\
      & = (-1)^{q_1\frac d2 + r + \frac d2 + 1}, \\
      \Longrightarrow \quad [\sigma_{q_1,q_2} : \sigma_{q_1-1,\,q_2}]^\M & = (-1)^{(q_1-1)\frac d2 + r+ \frac d2 + 1}  \\
      & = (-1)^{\beta + 1}.
    \end{align*}
    
    \item From $\sigma_{q_1,\,q_2+2}'$ to $\sigma_{q_1,q_2}$. The path is the one we explicitly wrote at the end of Part 2.
    Notice that from $\rho_{i-1}$ to $\tau_i$ one removes the vertex $v_i = n-(q_2-i+1)\frac d2$, and from $\tau_i$ to $\rho_i$ one adds the vertex $v_i' = n-(q_2-i-1)\frac d2$.
    Therefore
    \[ [\rho_{i-1} : \tau_i][\rho_i : \tau_i] = (-1)^{|\{ w \in \tau_i \mid v_i' < w < v_i \}|} = (-1)^{d-1} = -1. \]
    Then we can compute the incidence number in the Morse complex:
    \begin{align*}
      [\sigma_{q_1,\,q_2+2}' : \sigma_{q_1,q_2}]^\M &= (-1)^m \left( \prod_{i=1}^m [\rho_{i-1} : \tau_i][\rho_i : \tau_i] \right) [\sigma_{q_1,q_2} : \rho_m] \\
      &= (-1)^m (-1)^m (-1)^{q_1\frac d2 + r + (d-1)} \\
      & = (-1)^{q_1\frac d2 + r + 1}, \\
      \Longrightarrow \quad [\sigma_{q_1,q_2}' : \sigma_{q_1,\,q_2-2}]^\M &= (-1)^{q_1\frac d2 + r + 1}  \\
      & = (-1)^{\beta + 1}.
    \end{align*}
    
    \item From $\sigma_{q_1+1,\,q_2}'$ to $\sigma_{q_1,q_2}'$.
    The path is trivial and consists of removing $v = q_1\frac d2 + r$.
    Then
    \begin{align*}
      & [\sigma_{q_1+1,\,q_2}' : \sigma_{q_1,q_2}']^\M = (-1)^{q_1\frac d2 + r}, \\
      \Longrightarrow \quad & [\sigma_{q_1,q_2}' : \sigma_{q_1-1,\,q_2}']^\M = (-1)^{(q_1-1)\frac d2 + r} = (-1)^{\beta+\frac d2}.
    \end{align*}
    
    \item From $\sigma_{q_1,q_2}$ to $\sigma_{q_1,q_2}'$.
    The path is trivial and consists in removing $v = q_1\frac d2 + r + \frac d2$.
    Then
    \begin{align*}
      & [\sigma_{q_1,\,q_2} : \sigma_{q_1,q_2}']^\M = (-1)^{q_1\frac d2 + r + \frac d2 - 1} = (-1)^{\beta + \frac d2 + 1}.
    \end{align*}
    
    \item From $\sigma_{q_1, \,q_2+1}'$ to $\sigma_{q_1,q_2}'$.
    The path is trivial and consists in removing $v = n - q_2\frac d2$.
    Then
    \begin{align*}
      [\sigma_{q_1, \,q_2+1}' : \sigma_{q_1,q_2}']^\M &= (-1)^{n - q_2\frac d2 - (k-q_1-q_2)} \\
      &= (-1)^{k\frac d2 + r - q_2 \frac d2 -k + q_1 + q_2} \\
      &= (-1)^{(k-q_2)\big(\frac d2 - 1\big) + r + q_1}, \\
      \Longrightarrow \quad [\sigma_{q_1,q_2}' : \sigma_{q_1,\,q_2-1}']^\M &= (-1)^{(k-q_2+1)\big(\frac d2 - 1\big) + r + q_1} \\
      &= (-1)^{\alpha + 1}. \qedhere
    \end{align*}
  \end{itemize}
\end{proof}

We can finally compute the homology $H_*(\GW; R)$ for Artin groups of type $\tilde C_n$.

\begin{theorem}[Homology in case $\tilde C_n$]
  \label{thm:homology-tCn}
  Let $\GW$ be an Artin group of type $\tilde C_n$.
  Then the $\varphi_d$-primary component of $H_*(\GW;R)$ is trivial for $d$ odd, and for $d$ even is as follows:
  \[ H_m(\GW;R)_{\varphi_d} \cong
    \begin{cases}
      \left( R/(\varphi_d) \right)^{\oplus m+k-n+1} & \text{if $n-k \leq m \leq n-1$}, \\
      0 & \text{otherwise},
    \end{cases}
  \]
  where $n=k\frac d2 + r$.
\end{theorem}

\begin{proof}
  In order to apply Theorem \ref{thm:homology-artin-groups} we need to find the rank the boundary maps $\delta_{m+1}^\varphi$ of the Morse complex.
  For $d$ odd there is only one critical cell, thus all the boundaries vanish.
  Suppose from now on that $d$ is even, and let $\smash{n = k\frac d2 + r}$.
  In order to have a non-trivial boundary $\delta_{m+1}^\varphi$ we must have at least one critical simplex both in dimension $m$ and in dimension $m+1$, thus $m=n-k+l$ with $0 \leq l \leq k-1$ (see Table \ref{table:tCn-critical-by-dimension}).
  
  \textbf{Case 1: $l\leq k-2$.}
  We are going to prove that, for $l\leq k-2$, a basis for the image of $\delta = \delta_{m+1}^\varphi$ is given by
  \[ \mathcal B = \big\{ \delta \sigma_{q_1,q_2} \mid q_1+q_2 = l \big\} . \]
  By Theorem \ref{thm:matching-tCn} we obtain the following formula for $\delta \sigma_{q_1,q_2}$:
  \[ \delta \sigma_{q_1, q_2} = (-1)^\alpha \, \sigma_{q_1,\,q_2-1} + (-1)^{\beta + 1} \, \sigma_{q_1-1,\,q_2} + (-1)^{\beta + \frac d2 + 1} \, \sigma_{q_1,q_2}', \]
  where $\alpha = (k-q_2)\big(\frac d2 - 1\big) + q_1 + r + \frac d2$ and $\beta = q_1\frac d2 + r$.
  In this formula the $\sigma_{q_1,\,q_2-1}$ (resp.\ $\sigma_{q_1-1,\,q_2}$) term vanishes if $q_2 = 0$ (resp.\ $q_1 = 0$).
  The term $\sigma_{q_1,q_2}'$ appears in $\delta \sigma_{q_1,q_2}$ but not in any other element of $\mathcal B$, thus $\mathcal B$ is a linearly independent set.
  In addition, if $q_1 + q_2 = l+1$, we have that
  \begin{align*}
    & \delta \sigma_{q_1,q_2}' = (-1)^{\beta + 1} \sigma_{q_1,\,q_2-2} + (-1)^{\beta} \sigma_{q_1-2,\,q_2} + (-1)^{\alpha+1} \sigma_{q_1,\,q_2-1}' + (-1)^{\beta + \frac d2} \sigma_{q_1-1,\,q_2}' \\
    &= (-1)^{\alpha + \beta + \frac d2} \Big( (-1)^{\alpha + \frac d2 + 1}\, \sigma_{q_1,\,q_2-2} + (-1)^{\beta+1}\, \sigma_{q_1-1,\,q_2-1} + (-1)^{\beta + \frac d2 + 1} \, \sigma_{q_1,\,q_2-1}' \Big) \\
    & \quad + (-1)^{\frac d2 + 1} \Big( (-1)^{\alpha+1} \, \sigma_{q_1-1,\,q_2-1} + (-1)^{\beta + \frac d2 + 1} \, \sigma_{q_1-2,\,q_2} + (-1)^{\beta + 1}\, \sigma_{q_1-1,\,q_2}' \Big) \\
    &= (-1)^{\alpha + \beta + \frac d2}\, \delta \sigma_{q_1,\,q_2-1} + (-1)^{\frac d2 + 1}\, \delta \sigma_{q_1-1,\,q_2}.
  \end{align*}
  Therefore $\mathcal B$ generates the image of $\delta_{m+1}^\varphi$. Thus $\rk \delta_{m+1}^\varphi = |\mathcal B| = l+1$ for $l \leq k-2$.
  
  \textbf{Case 2: $l = k-1$.}
  For $l=k-1$, i.e.\ $m=n-1$, the situation is a bit different because there are no critical simplices of the form $\sigma_{q_1,q_2}$ with $q_1+q_2 = l$.
  However we can still define
  \[ \epsilon_{q_1, q_2} = (-1)^\alpha \, \sigma_{q_1,\,q_2-1} + (-1)^{\beta + 1} \, \sigma_{q_1-1,\,q_2} + (-1)^{\beta + \frac d2 + 1} \, \sigma_{q_1,q_2}' \]
  for $q_1+q_2 = l$, and 
  \[ \mathcal B' = \big\{ \epsilon_{q_1,q_2} \mid q_1+q_2 = l \big\} . \]
  The term $\sigma_{q_1,q_2}'$ appears in $\epsilon_{q_1,q_2}$ but not in any other element of $\mathcal B'$, thus $\mathcal B'$ is a linearly independent set.
  As above we have that, for $q_1+q_2 = l+1$,
  \[ \delta \sigma_{q_1,q_2}' = (-1)^{\alpha + \beta + \frac d2}\, \epsilon_{q_1,\,q_2-1} + (-1)^{\frac d2 + 1}\, \epsilon_{q_1-1,\,q_2}. \]
  Then $\mathcal B'$ generates (and so it is a basis of) the image of $\delta_n^\varphi$.

  We have proved that, for $0\leq l \leq k-1$, the rank of $\delta_{m+1}^\varphi$ is equal to $l+1 = m+k-n+1$. Then we conclude applying Theorem \ref{thm:homology-artin-groups}.
\end{proof}

\addtocontents{toc}{\protect\setcounter{tocdepth}{-1}}

\section*{Concluding remarks}

\addtocontents{toc}{\protect\setcounter{tocdepth}{2}}
\addcontentsline{toc}{section}{Concluding remarks, Acknowledgements, References}
\addtocontents{toc}{\protect\setcounter{tocdepth}{-1}}

Future works will focus on other families of Artin groups.
In particular it seems that precise matchings can be constructed in all finite and affine cases (see \cite{paolini2017local}), possibly allowing explicit homology computations.

The methods developed in this paper are particularly powerful when the coefficients are over a PID, but Artin groups with roots of different lengths (e.g.\ $B_n$, $\smash{\tilde B_n}$, and $\smash{\tilde C_n}$) also admit natural representations over polynomial rings with more than one variable.
We believe that some of our theory can be extended to such cases, and this can also be the aim of future works.

\section*{Acknowledgements}
The authors are grateful to the anonymous referee for his/her useful comments and suggestions.
This work was partially supported by Ministero dell'Istruzione, dell'Università e della Ricerca, and by University of Pisa, Project no.\ PRA\_67.

\bibliographystyle{amsalpha-abbr}
\bibliography{bibliography}

\bigskip\bigskip\bigskip

\end{document}